\documentclass[reqno,12pt]{amsart}
\usepackage[english]{babel}
\usepackage{amsmath}

\usepackage{amssymb}
\usepackage[scr=boondoxo]{mathalfa}
\usepackage{graphicx}
\usepackage{epstopdf}

\usepackage{subfigure}

\newcommand{\bpr}{\begin{trivlist} \item[]{\bf Proof. }}
\newcommand{\epr}{\hspace*{\fill} $\qed$\end{trivlist}}

\newcommand{\be}{\begin{eqnarray}}
\newcommand{\ee}{\end{eqnarray}}
\newcommand{\ba}{\begin{align}}
\newcommand{\ea}{\end{align}}
\newcommand{\bi}{\begin{itemize}}
\newcommand{\ei}{\end{itemize}}

\newcommand{\secref}[1]{Section~\ref{sec:#1}}

\newcommand{\seclab}[1]{\label{sec:#1}}
\newcommand{\eqlab}[1]{\label{eq:#1}}
\renewcommand{\eqref}[1]{(\ref{eq:#1})}

\newcommand{\figref}[1]{Fig.~\ref{fig:#1}}
\newcommand{\figlab}[1]{\label{fig:#1}}
\newcommand{\propref}[1]{Proposition~\ref{proposition:#1}}
\newcommand{\proplab}[1]{\label{proposition:#1}}
\newcommand{\lemmaref}[1]{Lemma~\ref{lemma:#1}}
\newcommand{\lemmalab}[1]{\label{lemma:#1}}

\newcommand{\thmref}[1]{Theorem~\ref{theorem:#1}}
\newcommand{\thmlab}[1]{\label{theorem:#1}}

\newcommand{\R}{\mathbb R}

\newtheorem{theorem}{Theorem}[section]
\newtheorem{proposition}[theorem]{Proposition}

\newtheorem{lemma}[theorem]{Lemma}
\newtheorem{cor}[theorem]{Corollary}

\newtheorem{remark}[theorem]{Remark}

\numberwithin{equation}{section}

\usepackage{fullpage}
\usepackage{color,comment,xcolor}
\definecolor{orange}{RGB}{255,127,0}
\newcommand\WM[1]{{\color{black}{#1}}} 
\newcommand\SJ[1]{{\color{black}{#1}}} 
\newcommand\SJnew[1]{{\color{black}{#1}}} 
\newcommand\SJnewtwo[1]{{\color{black}{#1}}} 
\newcommand\PS[1]{{\color{black}{#1}}} 
\newcommand\PSnew[1]{{\color{black}{#1}}} 
\newcommand\KUK[1]{{\color{black}{#1}}} 

\begin{document}

\title{Singularly perturbed oscillators with exponential nonlinearities}
 
\author[1]{S. Jelbart${}^\ast$}\thanks{${}^\ast$School of Mathematics \& Statistics, University of Sydney, Camperdown, NSW 2006, Australia}
\author[1]{K. U. Kristiansen${}^\dagger$}\thanks{${}^\dagger$Department of Applied Mathematics and Computer Science, Technical University of Denmark, Lyngby, Kgs. 2800, Denmark}
\author[1]{P. Szmolyan${}^\ddagger$}\thanks{${}^\ddagger$Institute for Analysis and Scientific Computing, Vienna University of Technology, Wiedner Hauptstraße 8-10, Vienna 1040, Austria}
\author[1]{M. Wechselberger${}^\ast$}

\maketitle

 \begin{abstract}
 Singular exponential nonlinearities of the form $e^{h(x)\epsilon^{-1}}$ with $\epsilon>0$ small occur in many different applications. These terms have essential singularities for $\epsilon=0$ leading to very different behaviour depending on the sign of $h$. In this paper, we consider two prototypical singularly perturbed oscillators with such exponential nonlinearities. We apply a suitable normalization for both systems such that the $\epsilon\rightarrow 0$ limit is a piecewise smooth system. The convergence to this nonsmooth system is exponential due to the nonlinearities we study. By working on the two model systems we use a blow-up approach to demonstrate that this exponential convergence can be harmless in some cases while in other scenarios it can lead to further degeneracies. For our second model system, we deal with such degeneracies due to exponentially small terms by extending the space dimension, following the approach in \cite{kristiansen2017a}, and prove  -- for both systems -- existence of (unique) limit cycles by perturbing away from singular cycles having desirable hyperbolicity properties.

\bigskip
\smallskip

\noindent \textbf{keywords.} singular perturbations, non-smooth systems, blow-up method, exponential asymptotics, relaxation oscillations. 
 \end{abstract}
 \section{Introduction}
 
Exponential nonlinearities arise in many different areas of mathematical modelling. In electronic oscillators, for example, the Ebers-Moll model for an NPN transistor provides an exponential relationship between the `emitter current' and the `base-emitter voltage'. See \cite{ebers1954a,hester1968a}. Also, in chemical  kinetics, the reaction rates are, by the Arrhenius equation, exponential functions of the temperature. Frequently, the temperature is assumed constant in such models, but in systems where large temperature variations occur (e.g. in explosions), the resulting exponential nonlinearity becomes important for the dynamics. Similar nonlinearities appear in other settings when the effect of temperature becomes important, see e.g. \cite{br2005a,estrin1980a} for exponential nonlinearities in plastic deformation. In the related area of friction, exponential nonlinearities also play an important role, for example in rate-and-state friction laws, see \cite{dieterich1978a,dieterich1979a,ruina1983a,woodhouse2015a} and \cite{bossolini2017a,putelat2017a,2019arXiv190312232U} for dynamical studies of such models. Although these friction models were first derived from experiments, the exponential nonlinearities have later been connected to Arrhenius process resulting from breaking atomic bonds at the atomic level \cite{rice2001a}. Sometimes modellers also introduce exponentials more heuristically, for example when regularizing a switch by a $\tanh$-function.

All of the examples of exponential nonlinearities highlighted above, are also examples of (within relevant parameter regimes) singularly perturbed systems. Over the past decades, these type of systems have been successfully described by geometric singular perturbation theory (GSPT) and blow-up, see e.g. \cite{dumortier_1996,fen3,jones_1995,krupa_extending_2001}\WM{, but} singular exponential nonlinearities like $e^{h(x)\epsilon^{-1}}$, with essential singularities along $h(x)=0$ as $\epsilon\rightarrow 0$, have traditionally been seen as an obstacle to such analysis. The problem is two-fold. Firstly, such systems approach piecewise smooth systems (upon proper normalizations) as $\epsilon\rightarrow 0$, having very different behaviour for $h(x)>0$ and $h(x)<0$. This area is currently an active area of research, see e.g. \cite{kosiuk2015a,kristiansen2018a,kristiansen2019e}. 

Secondly, for systems with nonlinearities of the form $e^{h(x)\epsilon^{-1}}$ the convergence of the smooth system to its nonsmooth counterpart happens at an exponential rate. \KUK{Whereas this exponential convergence can be harmless in some cases, it can also lead to further degeneracies due to `exponential loss of hyperbolicity'. We will demonstrate this through the study of two prototypical systems, which we introduce in the following. The usual blowup method \cite{dumortier_1996,krupa_extending_2001} is adapted to algebraic loss of hyperbolicity, and cannot compensate for these exponential degeneracies. Nevertheless, recently in \cite{kristiansen2017a} it was shown how one can modify this approach (basically by extending the space dimension) to deal with these special degeneracies. We will use this modified blowup approach in the present paper.}

%


\subsection{Two prototypical oscillators: Hester and Le Corbeiller}
The aim of our paper, is to shed further light on singularly perturbed systems with exponential nonlinearities. We will do so by considering two `prototypical' singularly perturbed oscillators with exponential nonlinearities:
\begin{align}\eqlab{hester}
 \text{The Hester system:}\quad \begin{cases}\dot x &= y,\\
  \dot y &= -x-2 \gamma y +\mu \left(e^{ y\epsilon^{-1}}-\kappa e^{(1+\alpha) y\epsilon^{-1}}\right),
  \end{cases}
  \end{align}
  with $\gamma\in (0,1)$, $\alpha>0$, $\mu>0$, \SJnew{$\kappa > 0$,}
and 
\begin{align}\eqlab{corbeiller}
 {\text{The Le Corbeiller system:}\quad \begin{cases}
\dot x &=y+a,\\
\dot y &=-x+b y (2-e^{y\epsilon^{-1}}),
\end{cases}}
  \end{align}
with $b\in (0,1)$, $a>0$. 
In both systems, $0<\epsilon\ll 1$ is the singular perturbation parameter. 

The system in \eqref{hester} is a model of a transistor oscillator, see \cite{hester1968a}, based \SJ{on the Ebers-Moll large-signal approximation}. In reference to \cite{hester1968a}, we will refer to \eqref{hester} as the `Hester' system. The constant $\epsilon^{-1}$ in the exponentials is given by $e/kT$ where $k$ is the Boltzmann constant, $e$ is the magnitude of the electronic charge and $T$ is the temperature in Kelvin
. At room temperature this gives  $\epsilon \approx 10^{-2}$ and the approximation $0< \epsilon \ll 1$ is therefore reasonable, we believe. \figref{hester01} shows the phase portrait and associated oscillations with parameter values
\begin{align}
 \alpha = 0.5,\,\mu = 0.4,\,\kappa = 0.2,\,\gamma=0.3,\eqlab{hesterpara}
\end{align}
and $\epsilon = 0.1$. The observed sharp transitions indicate singular dynamics, even for this `large' value of $\epsilon$. The oscillations become increasingly slow-fast in kind with decreasing values of $\epsilon$, as shown in \figref{hester001}, which shows the phase portrait and associated oscillations with the same parameter values \eqref{hesterpara}, except with $\epsilon = 0.01$. The oscillations in \figref{hester01} and \figref{hester001} are known as \textit{two-stroke} relaxation oscillations, by reference to the two distinct components to the oscillation; similar oscillations have been considered in the context of GSPT in \cite{jelwex19}.

\begin{figure}[h!]
\begin{center}
\subfigure[]{\includegraphics[width=.4\textwidth]{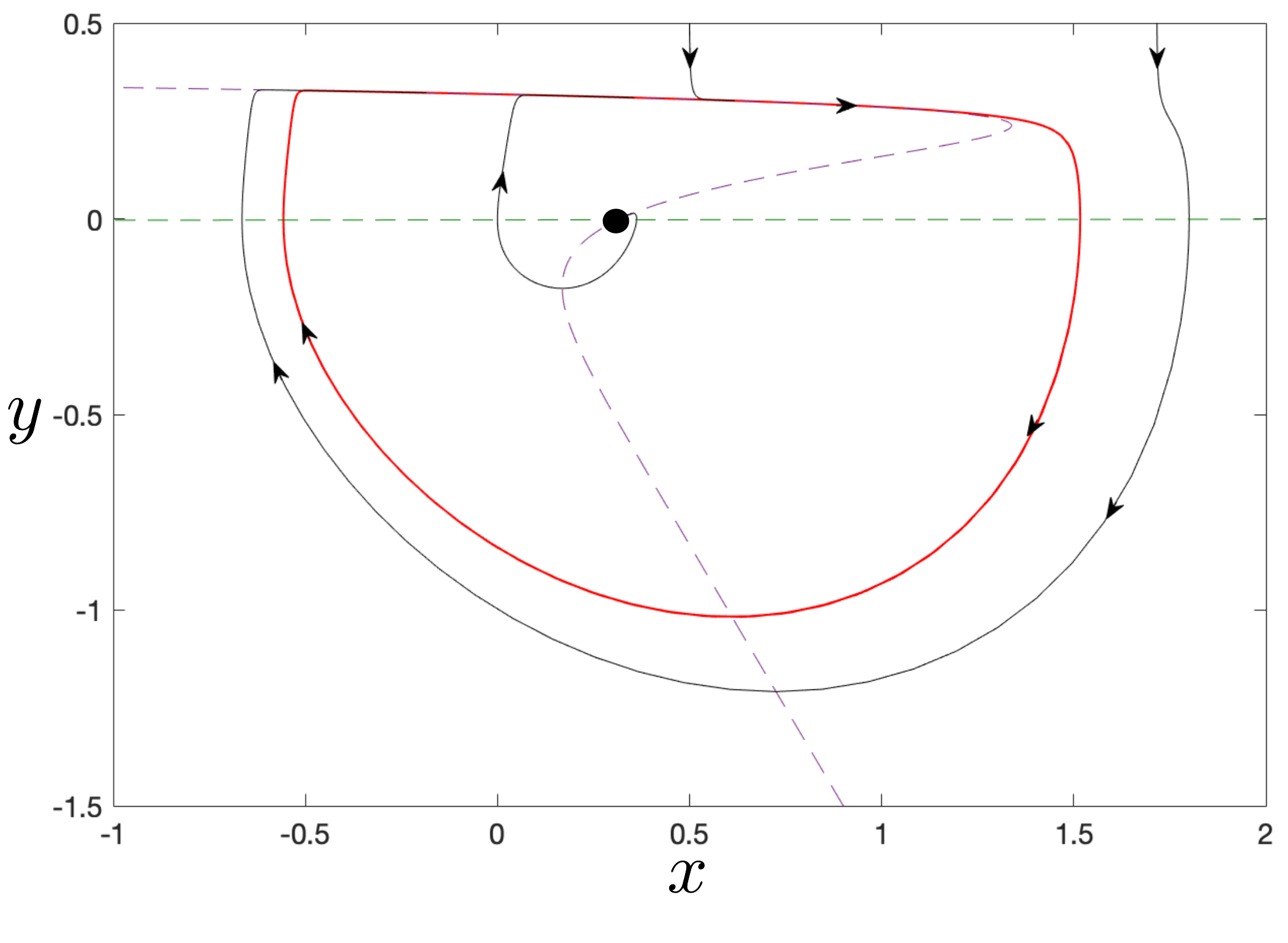}}
\subfigure[]{\includegraphics[width=.405\textwidth]{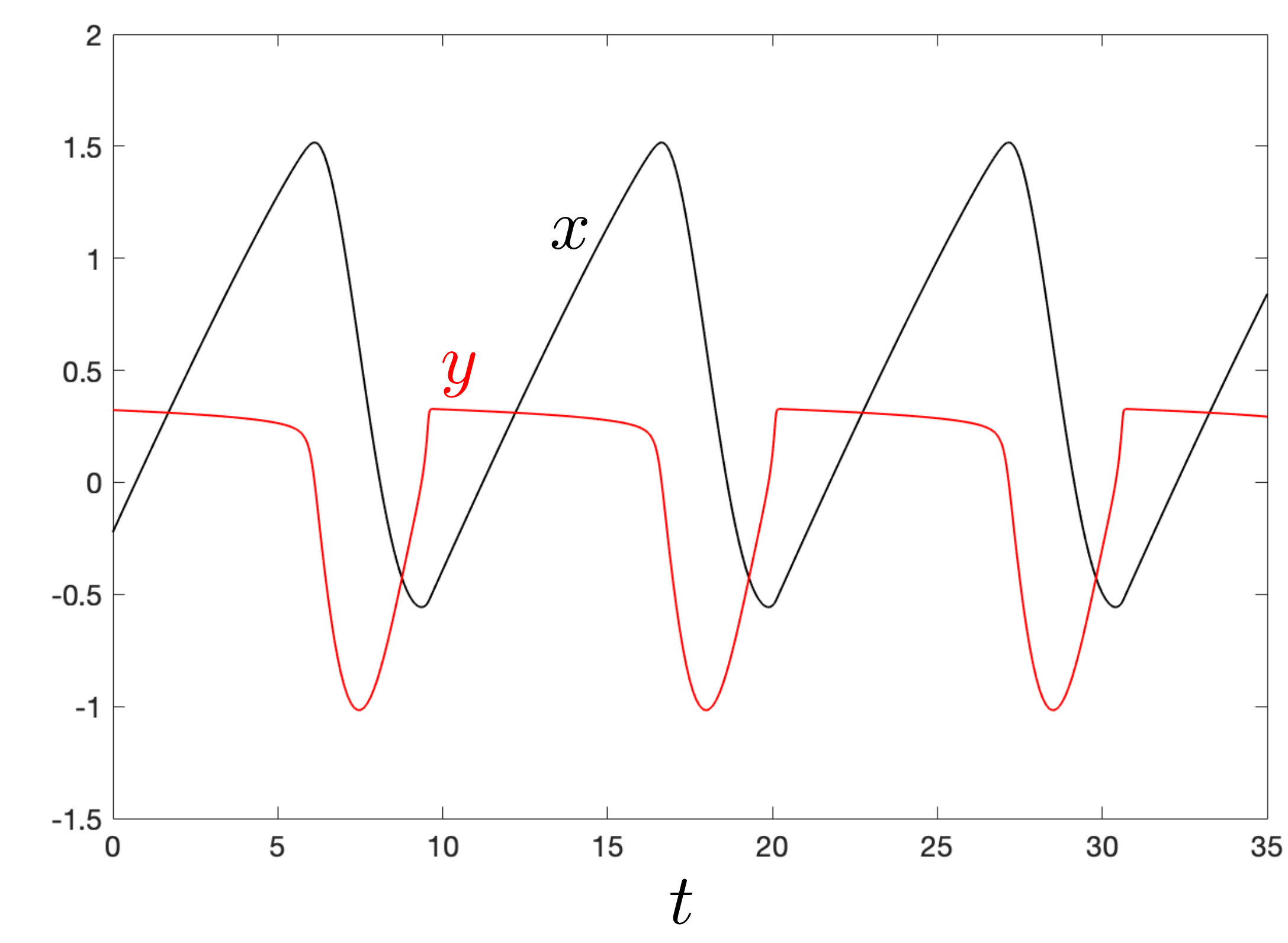}}
\caption{In (a): Phase portrait of \eqref{hester} for the parameter values in \eqref{hesterpara} and $\epsilon=0.1$. A stable limit cycle is \PS{shown} in red. \PS{The repelling equilibrium is marked as a black dot.} In (b): $x(t)$ and $y(t)$ along the limit cycle shown in (a). Also in (a): The nullclines are dashed and unstable focus is indicated by a black disk.  }\figlab{hester01}
\end{center}
              \end{figure}
          
 \begin{figure}[h!]
 	\begin{center}
 		\subfigure[]{\includegraphics[width=.4\textwidth]{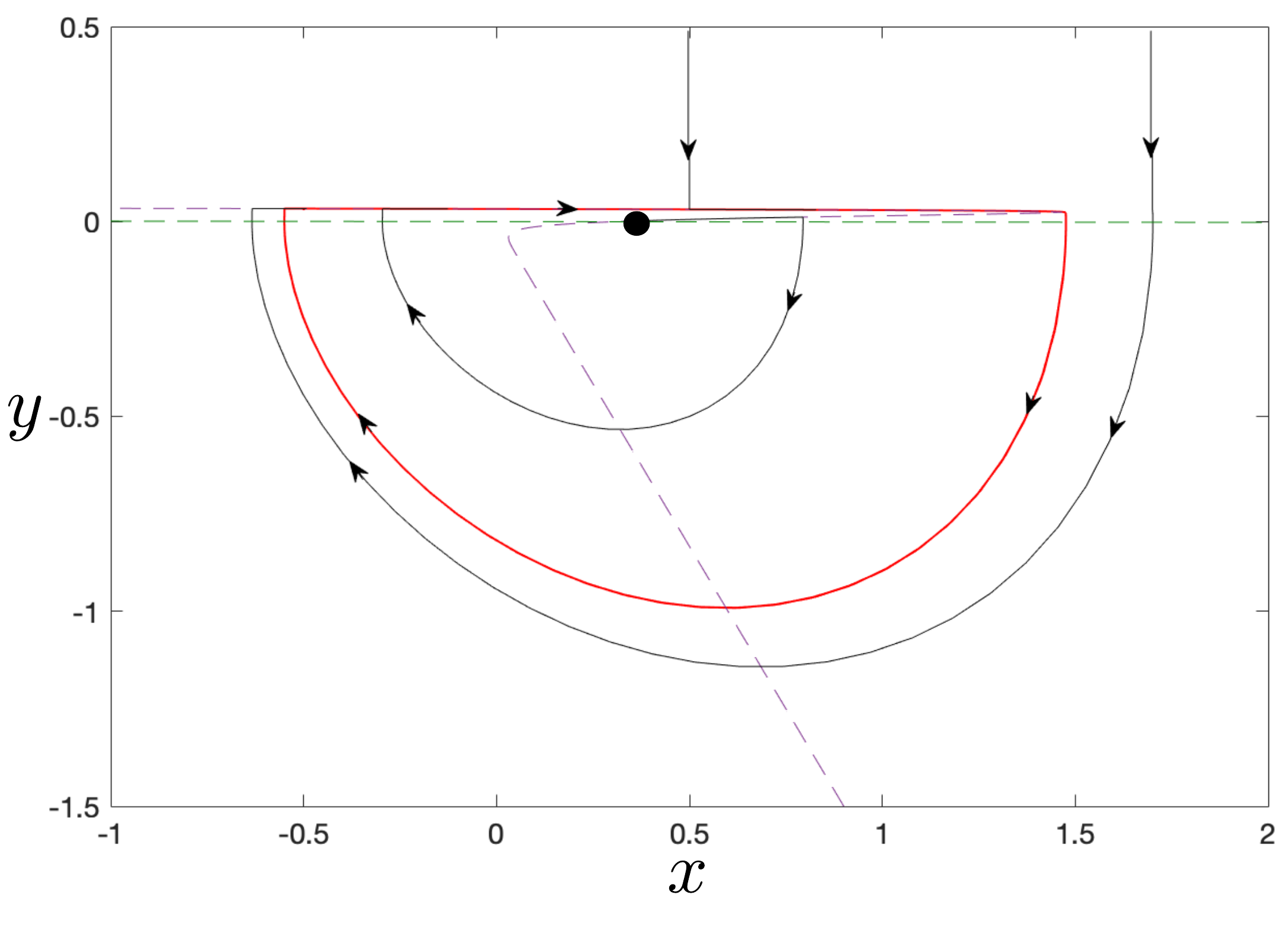}}
 		\subfigure[]{\includegraphics[width=.405\textwidth]{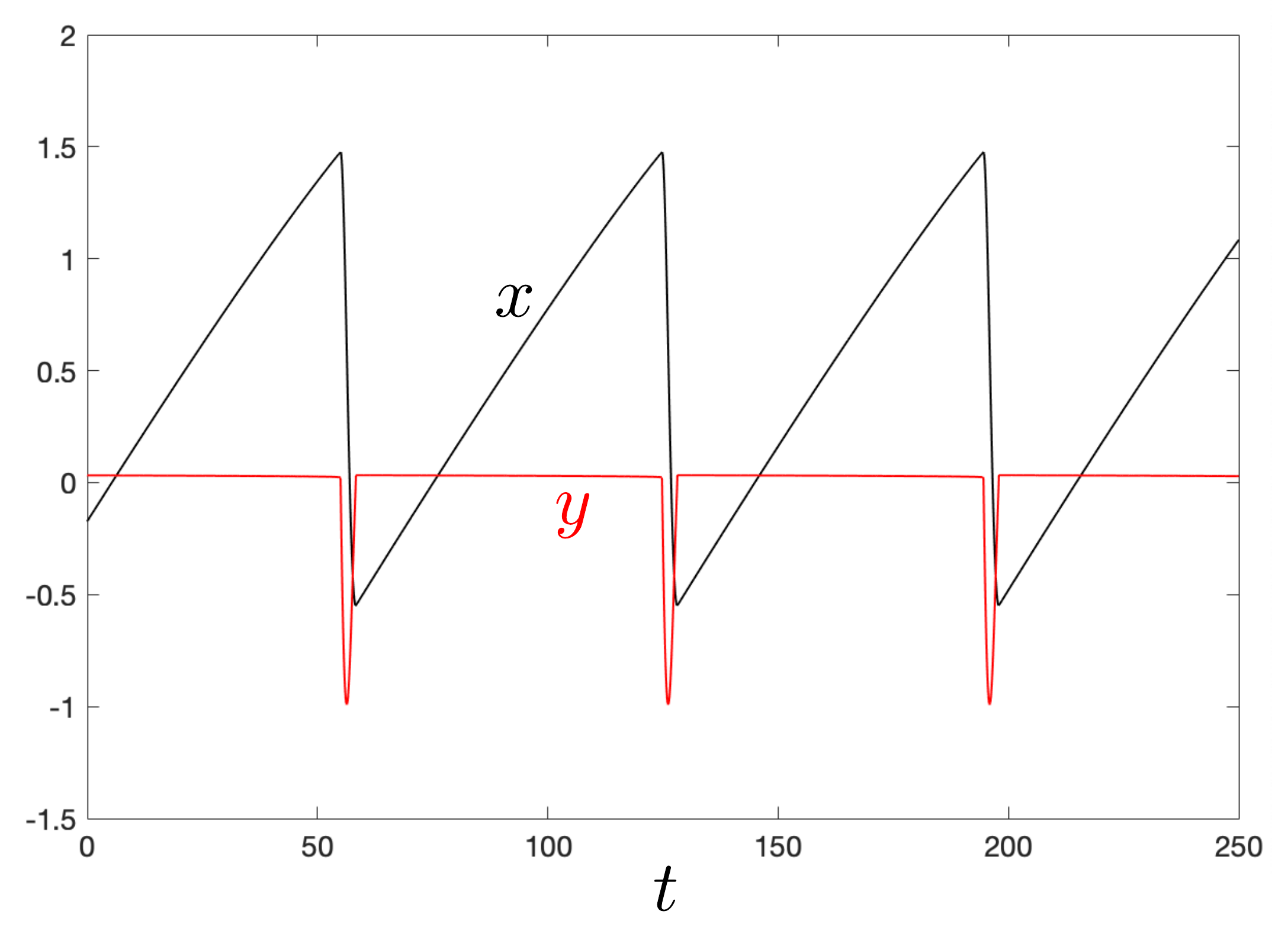}}
 		\caption{In (a): Phase portrait of \eqref{hester} for the parameter values in \eqref{hesterpara} and $\epsilon=0.01$. A stable limit cycle is \PS{shown} in red. In (b): $x(t)$ and $y(t)$ along the limit cycle shown in (a).}\figlab{hester001}
 	\end{center}
 \end{figure}

Regarding the motivation for the system in \eqref{corbeiller}, we first point out that for $a=0$ the $\epsilon$ can be scaled out by setting 
 $x = \epsilon x_2$,
such that 
\begin{align}
 \ddot{x}_2 +x_2+\mu \dot{x}_2 (e^{\dot{x}_2}-2)=0,\eqlab{Corbeiller0}
\end{align}
when writing the system as a second order equation. This equation appears in \cite[Eq. (25)]{le1960a} as an example of a simple system of `electric-oscillator-type' \SJ{exhibiting} 
two-stroke oscillation. Within the framework of electronic oscillators, we may therefore consider \eqref{corbeiller} with $a>0$ as a forced version of  \eqref{Corbeiller0} (by analogy with the `forced van der Pol oscillator'). In reference to \cite{le1960a}, we will refer to \eqref{corbeiller} as the `Le Corbeiller' system.  As in \figref{hester01} and \figref{hester001}, the phase portrait and associated oscillations for \eqref{corbeiller} are shown in \figref{corbeiller01} and \figref{corbeiller001}, for parameter values
\begin{align}
 a= 1,\,b = 0.25,\eqlab{corbeillerpara}
\end{align}
$\epsilon=0.1$ in \figref{corbeiller01}, and $\epsilon = 0.01$ in \figref{corbeiller001}. As with the Hester system \eqref{hester}, sharp transitions between distinct components of the oscillations indicates singular dynamics even for the `large' $\epsilon$ value in \figref{corbeiller01}. In contrast to the Hester system, however, \KUK{$\dot x=a>0$ along the (noninvariant) set defined by $y=0$ for \eqref{corbeiller} and -- as a result -- the oscillations, spending a fraction of their time near this set, do not become slow-fast in kind with decreasing $\epsilon$. Rather it appears that the period has a well-defined limit as $\epsilon\rightarrow 0$. Nevertheless, the `singular' nature of the oscillations does become more pronounced as $\epsilon\rightarrow 0$, insofar as the transition between the two distinct components of the oscillation becomes sharper.}

\begin{figure}[h!]
\begin{center}
\subfigure[]{\includegraphics[width=.4\textwidth]{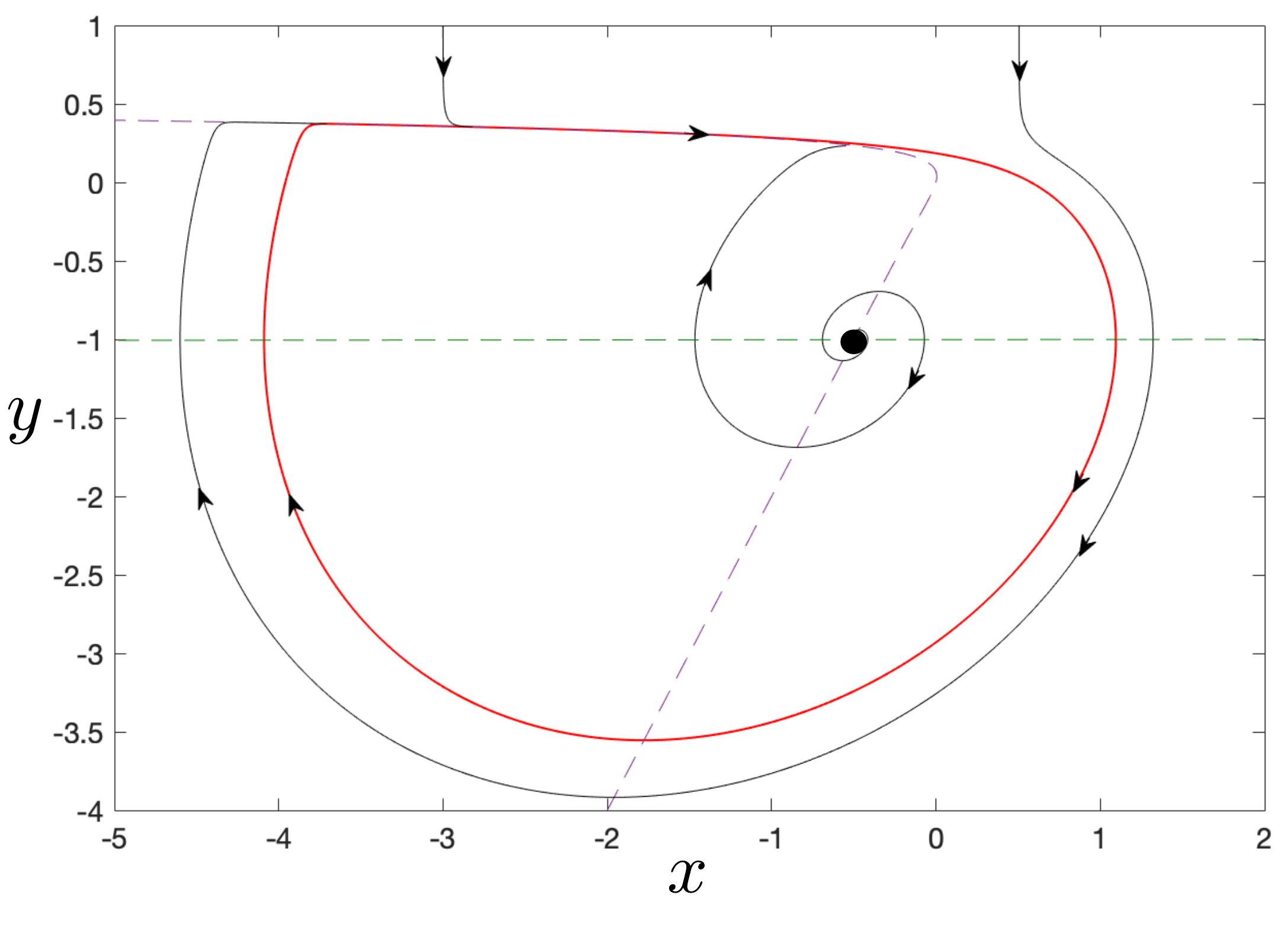}}
\subfigure[]{\includegraphics[width=.405\textwidth]{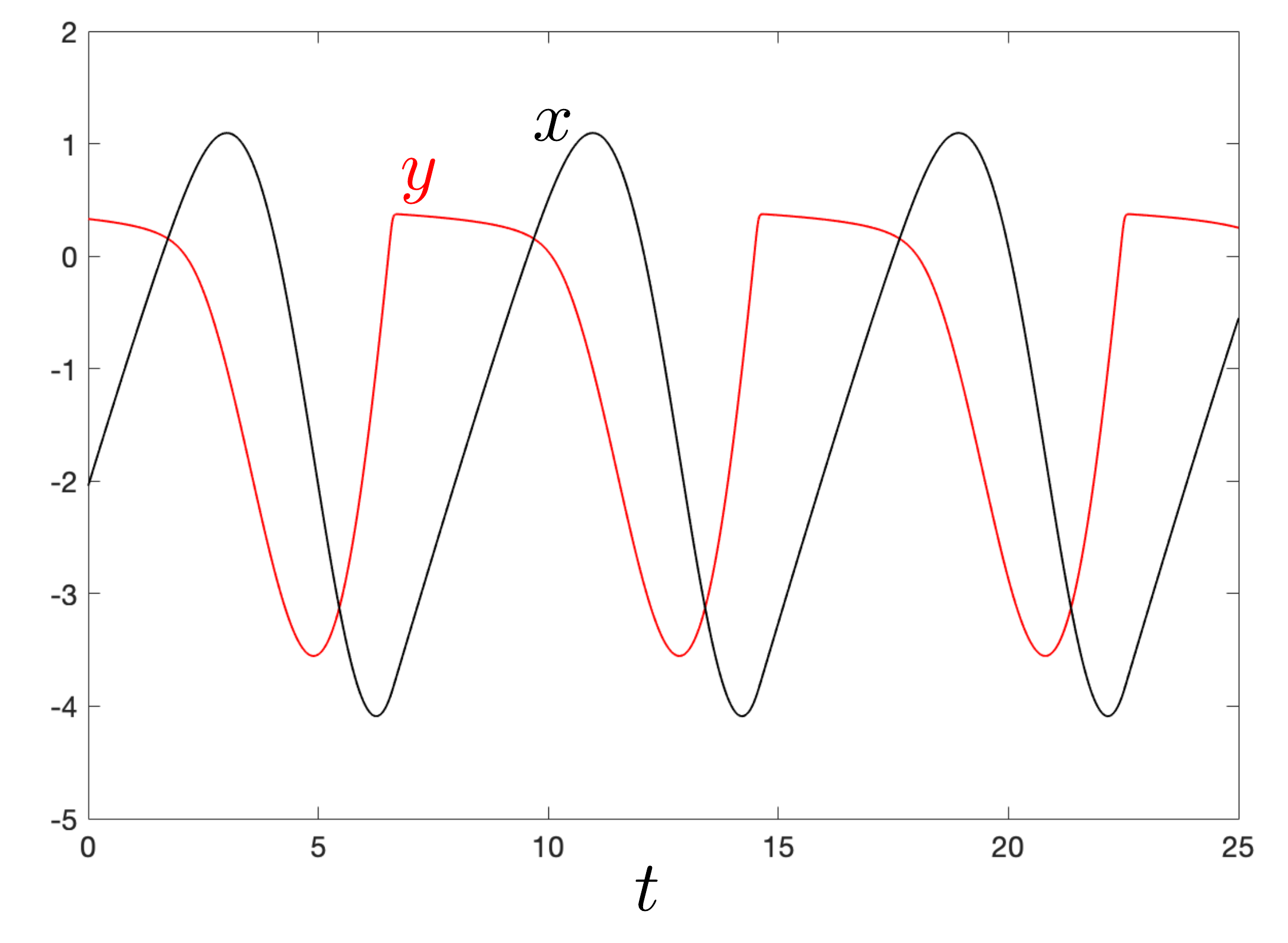}}
\caption{In (a): Phase portrait of \eqref{corbeiller} for the parameter values in \eqref{corbeillerpara} and $\epsilon=0.1$. 
\PS{The attracting limit cycle is shown in red, the repelling equilibrium is marked by a black dot.}
Even for this large value of $\epsilon$ the system clearly displays \PS{ a form of multi-scale dynamics, e.g. note the rather sharp corner 
of the limit cycle.} 
 In (b): $x(t)$ and $y(t)$ along the limit cycle shown in (a).}
  \figlab{corbeiller01}
 \end{center} 
              \end{figure}
          
 \begin{figure}[h!]
 	\begin{center}
 		\subfigure[]{\includegraphics[width=.4\textwidth]{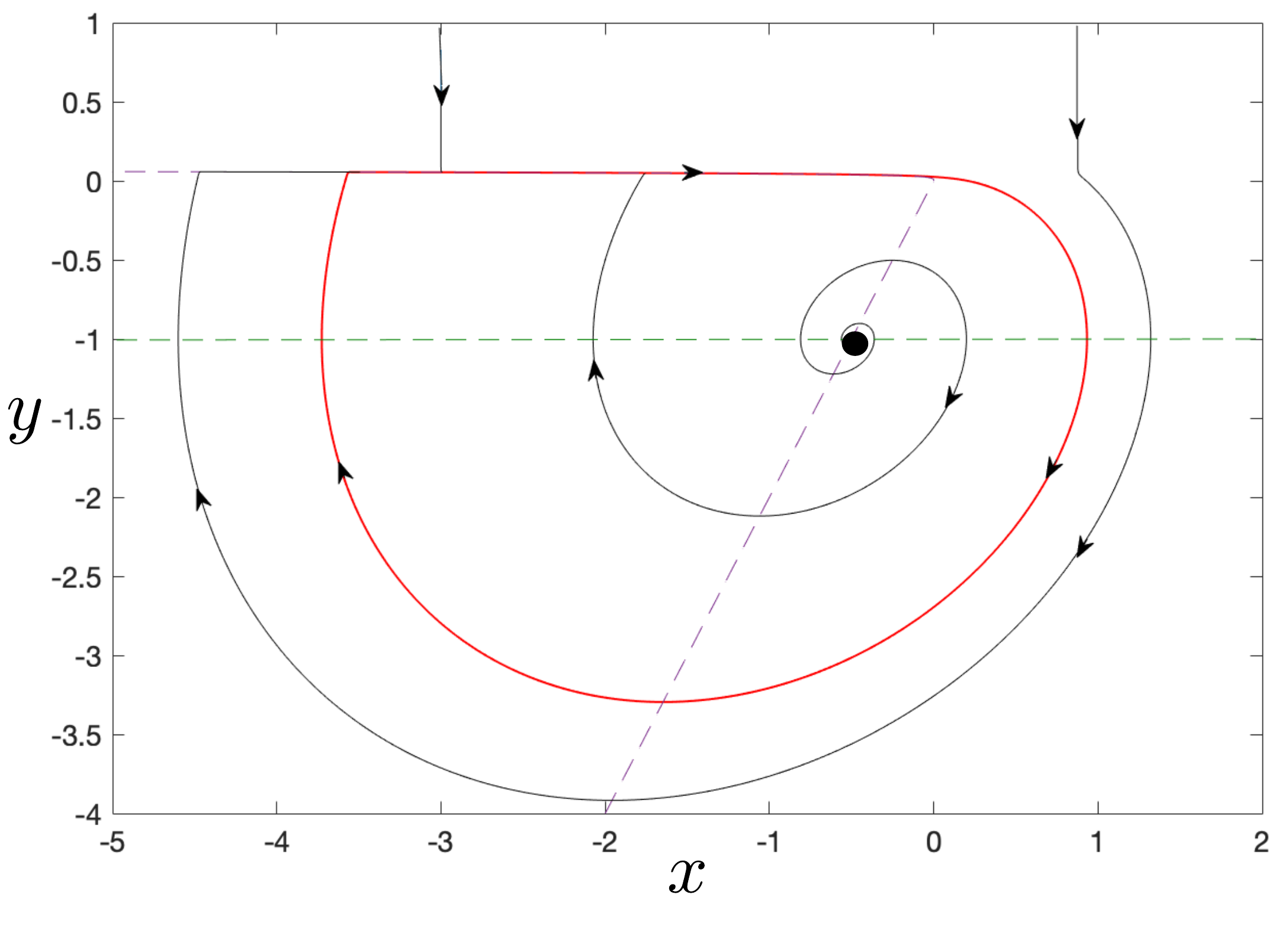}}
 		\subfigure[]{\includegraphics[width=.405\textwidth]{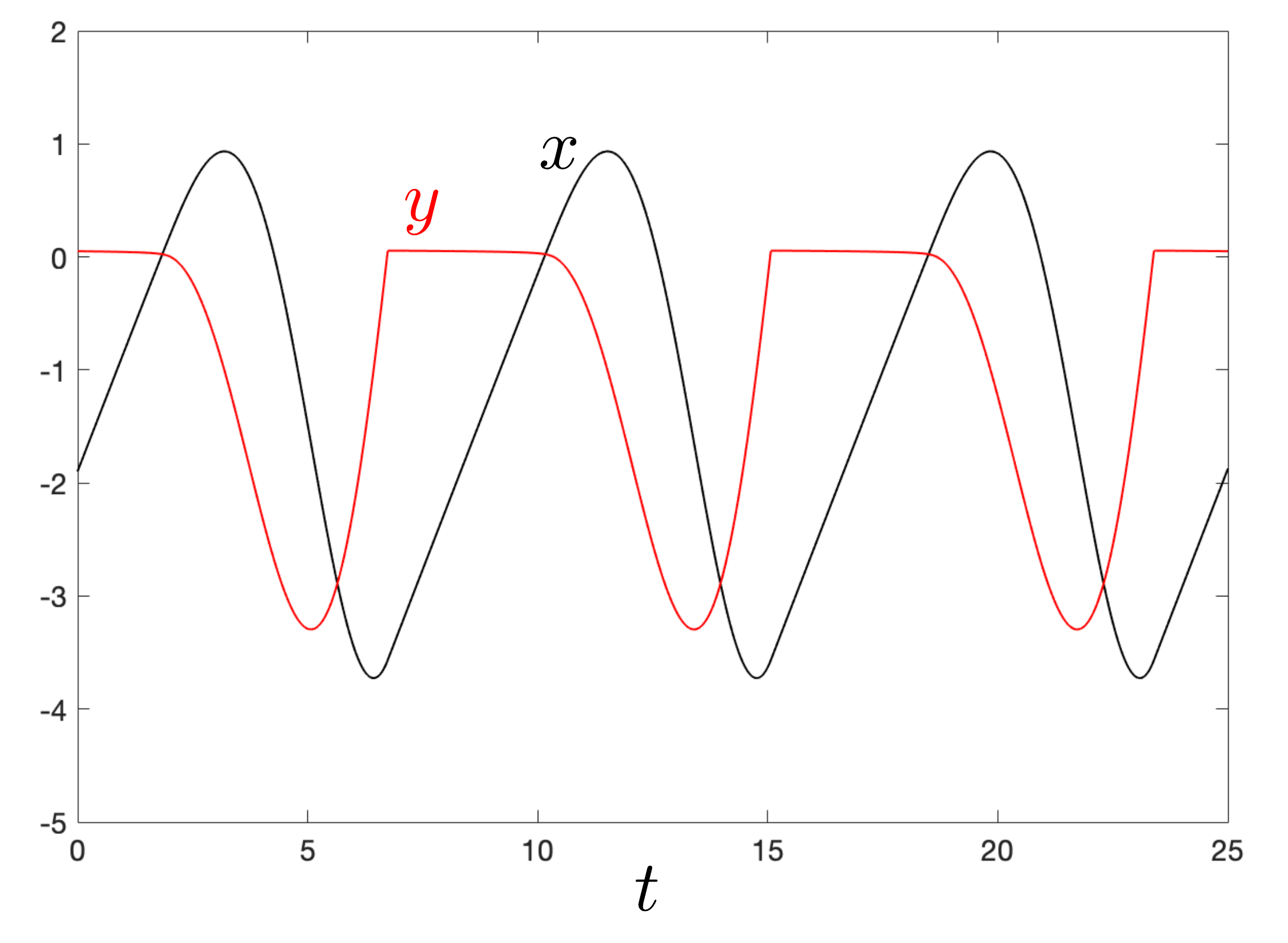}}
 		\caption{In (a): Phase portrait of \eqref{corbeiller} for the parameter values in \eqref{corbeillerpara} and $\epsilon=0.01$.
 			A stable limit cycle is \PS{shown}  in red. \PS{A sharp corner in the limit cycle is now clearly visible.} In (b): $x(t)$ and $y(t)$ along the limit cycle shown in (a).}
 		\figlab{corbeiller001}
 	\end{center} 
 \end{figure}

             \subsection{Main results}
             We prove existence of limit cycles for \eqref{hester} and \eqref{corbeiller} for all $0<\epsilon\ll 1$ using \SJ{a combination of GSPT and the blow-up method adapted in \cite{kristiansen2017a} for the study of} degeneracies caused by exponential \SJ{nonlinearities}. We present these results in the following, considering each system separately. 
              
             \begin{remark}
              While systems \eqref{hester} and \eqref{corbeiller} seem similar in nature, the level of
              difficulty to analyse them using the GSPT toolbox is remarkably different.
             \end{remark}
            
             \subsubsection{The Hester system}
             Due to the singular exponential nonlinearity $e^{y\epsilon^{-1}}$, system \eqref{hester} has no limit for $\epsilon\rightarrow 0$ for $y>0$. This problem can be circumvented by appealing to the 
             notion of topological equivalence and applying a time transformation 
             \begin{align}\eqlab{dt1}
             dt_1=(1+e^{(1+\alpha)y\epsilon^{-1}})dt,
             \end{align} which gives
\begin{equation}
\eqlab{hester2}
\begin{split}
               \dot x &= \frac{y}{1+e^{(1+\alpha) y\epsilon^{-1}}},\\
               \dot y &= \frac{-x-2 \gamma y+\mu e^{y\epsilon^{-1}}}{1+e^{(1+\alpha) y\epsilon^{-1}}}-\frac{\kappa \mu e^{(1+\alpha)y\epsilon^{-1}}}{1+e^{(1+\alpha)y\epsilon^{-1}}},
               \end{split}
               \end{equation}
               where with a slight abuse of notation the overdot denotes differentiation with respect to the new time $t_1$.
                              For any $\epsilon>0$, this corresponds to a smooth transformation of time; \eqref{hester} and \eqref{hester2} therefore have the same orbits for any $\epsilon>0$. 
\begin{remark}                              
\KUK{Viewed less abstractly, the time transformation defined by \eqref{dt1} corresponds to a multiplication of the vector-field by the strictly positive function   $(1+e^{(1+\alpha)y\epsilon^{-1}})^{-1}$. This leaves all orbits unchanged. Such 
\WM{space dependent rescalings together with the blow-up method }
are commonly used in GSPT, 
see e.g. 
 \cite{Gucwa2009783,kosiuk2015a,krupa_relaxation_2001} and will also be used frequently in the present paper.}
              \end{remark}

           \KUK{Clearly, the time transformation defined by \eqref{dt1} is singular for $\epsilon=0$, but it has the advantage that the new system \eqref{hester2}                  
has a well-defined pointwise limit as $\epsilon\rightarrow 0$ for any $y\ne 0$. In fact, in this limit we obtain the} {\em piecewise smooth} (PWS) system:
\begin{align}\eqlab{hesteryPos}
\begin{split}
 \dot x &=0,\\
 \dot y &=-\kappa \mu,
 \end{split}
\end{align}
for $y>0$ and 
\begin{align}\eqlab{hesteryNeg}
\begin{split}
 \dot x &=y,\\
 \dot y &=-x-2 \gamma y, 
 \end{split}
\end{align}
for $y<0$, \PS{the dynamics of which} we sketch in \figref{hesterPWS}. The discontinuity set $\Sigma=\{y=0\}$ is called the switching manifold in the PWS literature \cite{Bernardo08}. \SJnewtwo{Given $\gamma \in (0,1)$,} system \eqref{hesteryNeg} for $y<0$ has a stable focus at $(x,y)=(0,0)$ which is on the switching manifold $\Sigma = \{y=0\}$. In the PWS literature, this situation is known as a {\em boundary focus}, see \cite{Kuznetsov2003}. Under the forward flow of \eqref{hesteryPos} or \eqref{hesteryNeg}, respectively, every point with $y>0$ or $y<0$ will  reach $y=0$ in finite time. \SJnewtwo{The case $\gamma \geq 1$ for which $(0,0)$ is a stable node is also interesting (though not considered in \cite{hester1968a}); this is discussed in \secref{outlook}.}


Frequently, in PWS systems one prescribes a Filippov vector-field \cite{Bernardo08,filippov1988differential} on $\Sigma$ to have a well-defined forward flow. \KUK{However, since $\dot x=0$ on $y=0$ for both \eqref{hesteryPos} and \eqref{hesteryNeg}, the Filippov system is completely degenerate on $\Sigma$, consisting entirely of (pseudo-)equilibria \cite{Bernardo08,Kuznetsov2003}}\SJnew{.}
%
Our analysis of the Hester problem will reveal a slow flow \SJ{near} the switching manifold $\Sigma$ for all $0<\epsilon\ll 1$ (see \lemmaref{slowManifold-Hester}), and allow us to define a {\em singular relaxation cycle} 
\begin{align*}
 \Gamma_0 = \Gamma_1\cup \Gamma_2.
\end{align*}
Here $\Gamma_1$ is an orbit segment ({\em a fast jump}) of \eqref{hesteryNeg}, obtained by flowing the uniquely identified
{\em jump-off point} $(x_j,0)\in \Sigma$ with
\begin{align}
 x_j = \frac{\mu \alpha}{(1+\alpha)^{(1+\alpha)/\alpha}\kappa^{1/\alpha}},\eqlab{hesterx1}
\end{align}
forward until the first return to $\Sigma$ at the {\em drop point} $(x_d,0)$ with $x_d=x_d(x_j)<0$, whereas $\Gamma_2$ is a `slow orbit'  segment on $\Sigma$ connecting $(x_d,0)$ with $(x_j,0)$; see \figref{hesterPWS}.

  \begin{figure}[hbt]
\begin{center}
\includegraphics[width=.75\textwidth]{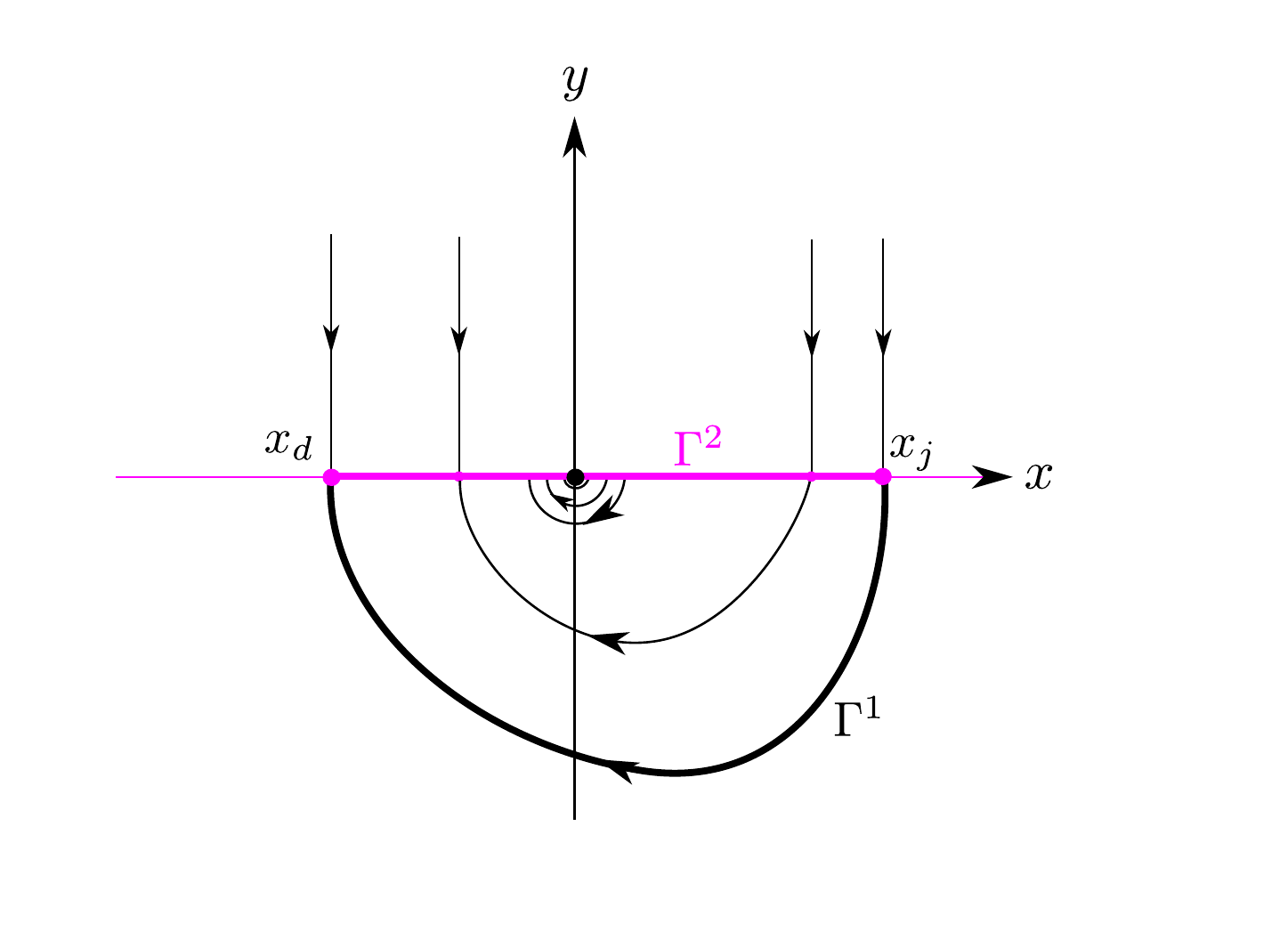}
\PS{
\caption{Phase portrait of the piecewise smooth system  \eqref{hesteryPos}, \eqref{hesteryNeg} in black.
The switching manifold  $\Sigma=\{y=0\}$ is shown in magenta.
 Under the conditions \eqref{hesterCondition}, the singular cycle $\Gamma_0=\Gamma^1\cup \Gamma^2$ perturbs to a stable limit cycle by \thmref{main1}.}\figlab{hesterPWS}
 }
 \end{center}
              \end{figure}


 
\PS{Our main result on the Hester system is the following theorem.}

              \begin{theorem}\thmlab{main1}
               Consider the Hester system \eqref{hester} and fix a large ball $B_r$ of radius $r$.  Suppose that
               \begin{align}
               \kappa (1+\alpha) \in (0,1),\,\gamma \in (0,1),\,\mu>0.\eqlab{hesterCondition}
               \end{align}
Then there exists an $\epsilon_0>0$ such that for all $\epsilon\in (0,\epsilon_0)$, system \eqref{hester} has a unique limit cycle $\Gamma_\epsilon$ in $B_r$. Furthermore, $\Gamma_\epsilon$ is attracting and
\begin{align*}
 \Gamma_\epsilon\rightarrow \Gamma_0,
\end{align*}
in Hausdorff-distance, as $\epsilon\rightarrow 0^+$. 

If $\kappa (1+\alpha) \notin (0,1]$ then no limit cycles exist for $0<\epsilon \ll 1$ in $B_r$. 
 \end{theorem}

\PS{In addition} we are 
able to identify the following asymptotics of a corresponding locally invariant (slow) manifold:
 
 \begin{lemma}\lemmalab{slowManifold-Hester}
             Let $I\subset (-\infty,x_j)$, with $x_j$ as in \eqref{hesterx1}, be a closed interval. Then for \eqref{hester}, there exists an exponentially attracting locally invariant (slow) manifold \PS{given as a graph:}
            \begin{align*}
             y = \epsilon h(x,\epsilon),\,x\in I,\,\epsilon \in [0,\epsilon_0) ,
            \end{align*}
            with $h$ smooth in both variables and $0<\epsilon_0\ll 1$.
                           \end{lemma}

 \subsubsection{The Le Corbeiller system}\seclab{the_le_corbeiller_system}
As in the `Hester case\SJnew{'}, due to the exponential term $e^{y\epsilon^{-1}}$, system \eqref{corbeiller} does not have a  limit as $\epsilon\rightarrow 0$ for $y>0$. Again, we introduce a time transformation 
$$dt_1=(1+e^{y\epsilon^{-1}})dt,$$
corresponding to multiplication of the vector-field by the strictly positive function $(1+e^{y\epsilon^{-1}})^{-1}$,
 to obtain
 \begin{align}
 \eqlab{corbeiller2}
 \begin{split}
\dot x &=\frac{y+a}{1+e^{y\epsilon^{-1}}},\\
\dot y &=\frac{-x+2b y}{1+e^{y\epsilon^{-1}}}-\frac{b y e^{y\epsilon^{-1}}}{1+e^{y\epsilon^{-1}}}.
\end{split}
  \end{align}
  \SJnew{For this system, t}he \SJnew{pointwise} limit \SJnew{as} $\epsilon\rightarrow 0$ is well\SJnew{-}defined for all $y\ne 0$\SJnew{,} and gives the following PWS system:
\begin{align}
\eqlab{corbeilleryPos}
\begin{split}
 \dot x &=0,\\
 \dot y &=-b y,\end{split}
\end{align}
for $y>0$ and 
\begin{align}\eqlab{corbeilleryNeg}
\begin{split}
 \dot x &=y+a,\\
 \dot y &=-x+2 b y, \end{split}
\end{align}
for $y<0$, \PS{with $\{y=0\}$ as switching manifold $\Sigma$.} Some orbits of this limiting PWS system are shown  in \figref{corbeillerPWS}.

Now, the $y<0$ system \eqref{corbeilleryNeg} has an unstable focus for $b\in (0,1)$ at $(x,y)=-a(2b,1)$, but also a quadratic, {\em visible fold} tangency \cite{Bernardo08} on the switching manifold at $(x,y)=(0,0)$. 
%
On the other hand, the $y>0$ system \eqref{corbeilleryPos} has a line of equilibria along $\Sigma$. The Filippov system is therefore again completely degenerate along $\Sigma$.

\PS{Our analysis of the Corbeiller problem will reveal a reduced flow on an invariant manifold near the switching manifold $\Sigma$ for all $0<\epsilon\ll 1$ (see \lemmaref{slowManifold-Corb}). Anticipating this, we define a {\em singular relaxation cycle} 
\begin{align}
\eqlab{SingCyc1}
 \Gamma_0 = \Gamma^1\cup \Gamma^2,
\end{align}
where $\Gamma^1$ is the orbit segment of \eqref{corbeilleryNeg} obtained by flowing the tangency point $(0,0)$ forward until the first return $(x_d,0)$, with $x_d<0$\SJnew{,} to $\Sigma$, see \figref{corbeillerPWS}. The set $\Gamma^2$ is  defined as the segment $(x_d,0)$
on the switching manifold $\Sigma$.}

 \begin{figure}[h!]
\begin{center}
\includegraphics[width=.75\textwidth]{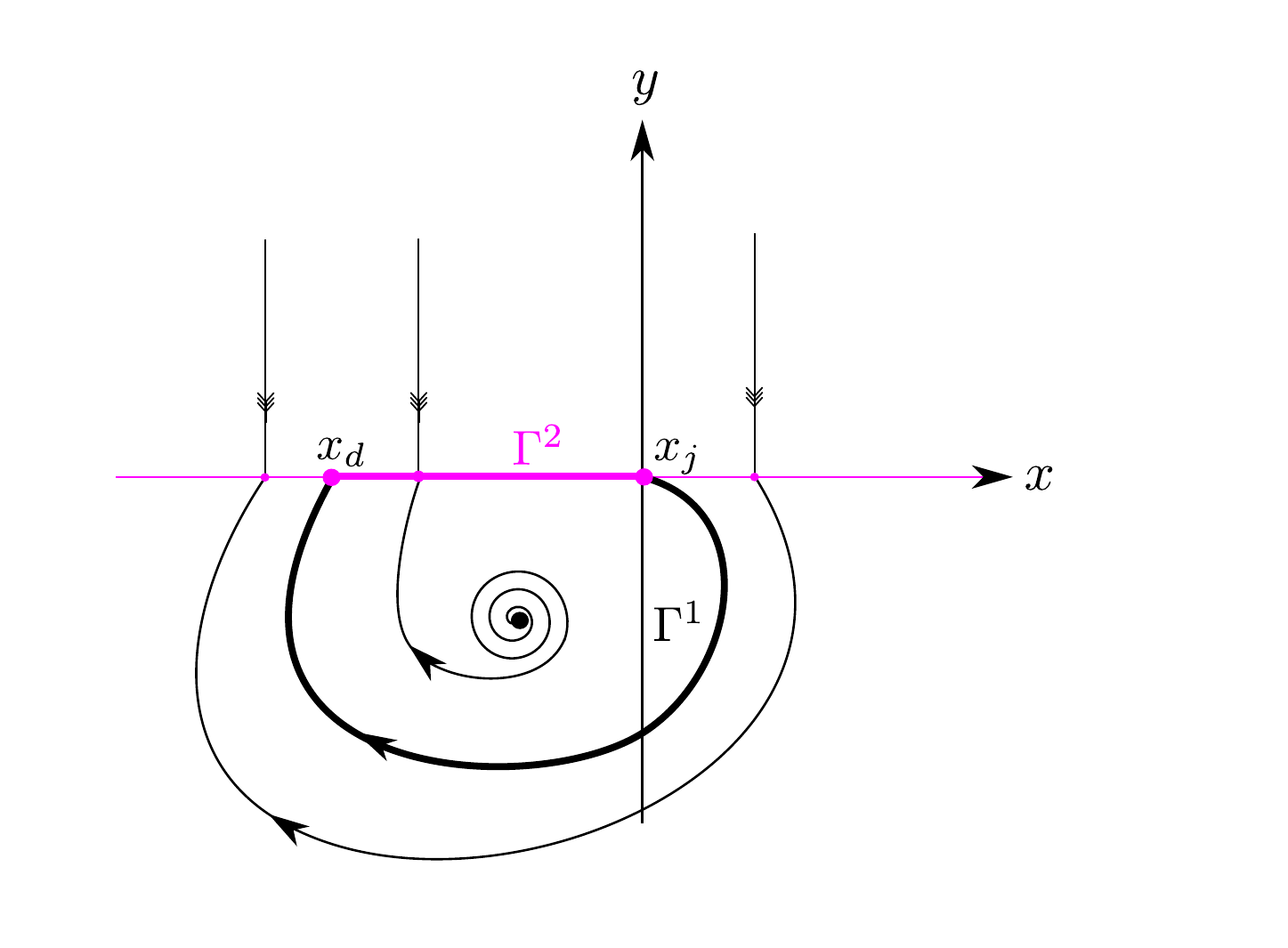}
\caption{\PS{Phaseportrait of the piecewise smooth system  \eqref{corbeilleryPos}, \eqref{corbeilleryNeg} in black.
The switching manifold  $\Sigma=\{y=0\}$ is shown in magenta.}
 Under the conditions \eqref{lecorbconditions}, the singular cycle $\Gamma_0=\Gamma^1\cup \Gamma^2$ perturbs to a stable limit cycle by \thmref{main2}.
\figlab{corbeillerPWS}
}
\end{center}
              \end{figure}

 \PS{Our main result on the Le Corbeiller system is:}
 
\begin{theorem}\thmlab{main2}
               Consider the Le Corbeiller system \eqref{corbeiller} and fix a large ball $B_r$ of radius $r$.  Suppose that
               \begin{align}
                b \in (0,1),\,a>0.\eqlab{lecorbconditions}
               \end{align}
Then there exists an $\epsilon_0>0$ such that for all $\epsilon\in (0,\epsilon_0)$, system \eqref{corbeiller} has a unique limit cycle $\Gamma_\epsilon$ in $B_r$. Furthermore, $\Gamma_\epsilon$ is attracting and 
\begin{align*}
 \Gamma_\epsilon\rightarrow \Gamma_0,
\end{align*}
as $\epsilon\rightarrow 0^+$, in Hausdorff-distance.
 \end{theorem}
 
  Again, we are also able to identify the following asymptotics of a corresponding locally invariant manifold with our methods: 

             \begin{lemma}\lemmalab{slowManifold-Corb}
    %
  For system \eqref{corbeiller}, fix $I\subset (-\infty,0)$. Then there exists an exponentially attracting locally invariant manifold 
\PS{given as a graph:}
  \begin{align}
   y = \epsilon W(-x/(\epsilon b)) \left(1- \epsilon b x^{-1} h(x,W(-x/(\epsilon b))^{-1})\right),\eqlab{corbeillerManifold}   
  \end{align}
  for $x\in I$, $\epsilon \in \SJ{(} 0,\epsilon_0)$,
  \KUK{with another smooth function} $h$ satisfying $h(x,0) = 2$,
  and where $W:(-e^{-1},\infty)\rightarrow (-1,\infty)$ is the principal Lambert-W function, defined by $z=W(ze^{z})$ for all $z\in (-1,\infty)$.
                           \end{lemma}
                           Using the asymptotics 
  \begin{align}
   W(w) = \log w(1+\mathcal O(\log^{-1} w \log \log w)),\eqlab{wasymp}
  \end{align}
of $W$ for $w\rightarrow \infty$, we realise that the invariant manifold \eqref{corbeillerManifold} has the following leading order asymptotics
\begin{align*}
 y \approx \epsilon \log (-x/(\epsilon b)),
\end{align*}
as $\epsilon\rightarrow 0^+$.

\begin{remark}
The two oscillations in \figref{hester001} and \figref{corbeiller001}, and the singular versions in \figref{hesterPWS} and \figref{corbeillerPWS}, look qualitatively similar. Also the statements in \thmref{main1} and \thmref{main2} are almost identical. However, their PWS versions have different degeneracies along $\Sigma$ and we shall see that the systems are very different 
\PS{in the singular limit $\epsilon \to 0$}\SJnew{,}
requiring different techniques for their analysis. For the Hester system \eqref{hester} under the assumption \eqref{hesterCondition} the exponentials do in fact not \PS{cause significant complications.} 
 In this respect the Le Corbeiller system \eqref{corbeiller} behaves very differently -- the exponential terms lead to several complications which require 
a much more involved analysis.
This is also reflected in the different asymptotics of their invariant manifolds presented in \lemmaref{slowManifold-Hester} and \lemmaref{slowManifold-Corb}.
\end{remark}
              
              \subsection{Overview}
                            In \secref{hestergeo}, we first study \eqref{hester} and prove \thmref{main1}. The proof, based upon the \PS{blow-up method and GSPT}, is 
                            fairly straightforward, in particular in comparison with the proof of \thmref{main2}, which makes up the rest of the paper, see \secref{corgeo}. Obviously, an essential step in the proof of \thmref{main2} will be to prove \lemmaref{slowManifold-Corb}. This is done in \secref{proofSlowManifold} after having described our blow-up approach. Subsequently, we present two lemmas \lemmaref{map1} and \lemmaref{map2} that prove \thmref{main2}, see \secref{PoincareMap} and \secref{thm2Proof}. In \secref{map1} we then prove \lemmaref{map1} before proving \lemmaref{map2} in \secref{sec:map2}. Further details of the blow-up used to prove \lemmaref{map2} are delayed to \secref{Details}. In \secref{outlook}, we conclude our paper by presenting an outlook. 
%
 \PS{
\begin{remark} Throughout this paper, we will assume some familiarity with geometric singular perturbation theory
and in particular with the blow-up method. The interested reader is referred to, e.g., \cite{kristiansen2018a,krupa_extending_2001,kuehn2015}
for background and more references.
\end{remark}
}

             \section{\PS{Blow-up} analysis of the Hester system}\seclab{hestergeo}


\subsection{The scaling approach}\seclab{scalingapproaches}
\PS{\SJnew{I}mportant insight into the dynamics of the Hester model \eqref{hester} for small values of $\epsilon$ can be gained 
by \SJnew{first} rescaling $y$ according to
              \begin{align}
               y=\epsilon y_2\,,
               \eqlab{hestery2}
              \end{align}
which provides a  zoom into the dynamics close to} the switching manifold $\Sigma$. This  gives
\begin{align}
\eqlab{hesterKappa2}\begin{split}
 \dot x &= \epsilon y_2,\\
 \epsilon \dot y_2 &=-x-2\gamma \epsilon y_2 + \mu \left(e^{y_2}-\kappa e^{(1+\alpha) y_2}\right)\,.
 \end{split}
\end{align}
System \eqref{hesterKappa2} is a standard slow-fast system\SJnew{,} \PS{which has the form 
\begin{align}
\eqlab{hesterKappa2fast}\begin{split}
  x' &= \epsilon^2 y_2,\\
 y'_2 &=-x-2\gamma \epsilon y_2 + \mu \left(e^{y_2}-\kappa e^{(1+\alpha) y_2}\right)\,
 \end{split}
\end{align}
on the fast time scale $\tau = t/\epsilon$. By setting $\epsilon=0$ we obtain the corresponding {\em layer problem}}
\begin{align*}
 x' &=0,\\
 y_2' &=-x+\mu \left(e^{y_2}-\kappa e^{(1+\alpha) y_2}\right).
\end{align*}
The set
\begin{align*}
 C = \left\{(x,y_2): x= \mu \left(e^{y_2}-\kappa e^{(1+\alpha) y_2}\right),y_2\in \mathbb R\right\},
\end{align*}
see  \figref{hesterKappa2},
is a manifold of equilibria, \PS{which in GSPT is called the critical manifold  of system~\eqref{hesterKappa2}.}
\PS{The critical manifold    $C $ 
is the disjoint union of the following sets
\begin{align*}
 C_{a}& = C\cap \{y_2 > y_j\},\\
 F &= C\cap \{y_2=y_j\} = \{(x_j,y_j)\},
\end{align*}
recall \eqref{hesterx1}, 
and
\begin{align*}
 C_{r} &= C\cap \{y_2<y_j\},
\end{align*}
where
\begin{align*}
 y_j  := -\frac{1}{\alpha}\log\left(\kappa(1+\alpha)\right).
\end{align*}
}

Here $C_{a}$ ($C_{r}$) is normally hyperbolic and attracting (repelling, respectively), whereas $F$ is a regular fold point. Since $C_{a}$ is normally hyperbolic, any compact submanifold perturbs to an \PS{attracting}  locally invariant slow manifold \PS{for $0 < \epsilon \ll 1$} by Fenichel's theory \cite{fen3}. This proves the statement regarding the invariant manifold in \lemmaref{slowManifold-Hester}.

The {\em reduced problem} on $C$ 
\begin{align*}
 \dot x &=y_2,\\
 0&= -x+\mu \left(e^{y_2}-\kappa e^{(1+\alpha) y_2}\right),
\end{align*}
is obtained by going to a (super) slow time and setting $\epsilon=0$. There is a unique equilibrium on $C$ at $y_2=0$. For $\kappa$ satisfying \eqref{hesterCondition} the equilibrium is on the repelling branch $C_r$ and it is an unstable node. 
\PS{In this case the slow flow on $C_a$ is towards the right and approaches the fold point, where a fast downward jump along 
the unstable fiber of the layer problem occurs, see  \figref{hesterKappa2}.}

\begin{figure}[h!]
\begin{center}
\includegraphics[width=.75\textwidth]{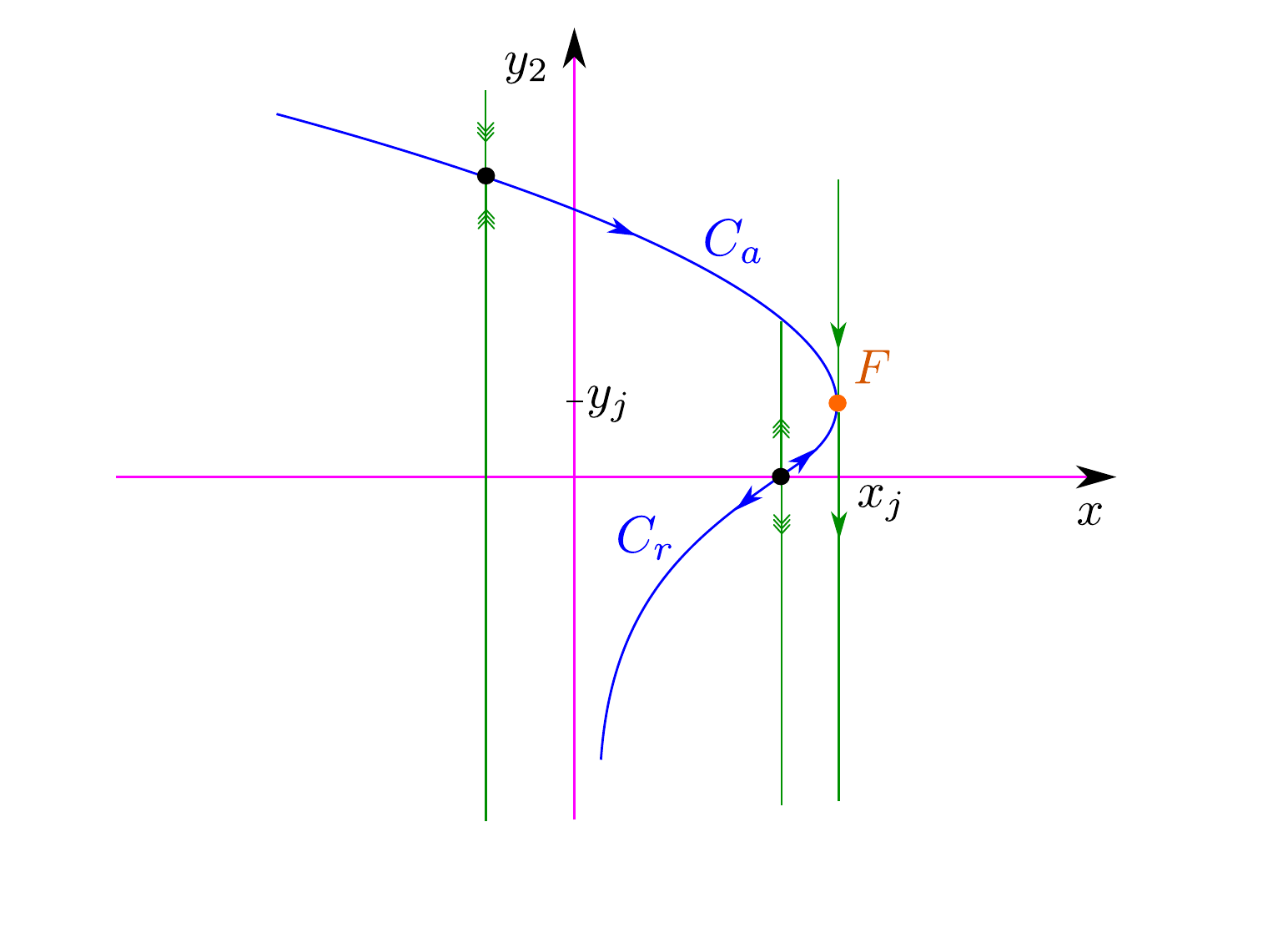}
\PS{ \caption{Dynamics of system  \eqref{hesterKappa2}  in the singular limit $\epsilon =0$. 
Orbits of the layer problem are shown in green. The critical manifold $C$ and the reduced flow on it are shown in blue.
The critical manifold $C$ has a fold point $F$ (orange), dividing $C$ into an attracting branch $C_a$ and repelling branch $C_r$. Under the conditions \eqref{hesterCondition} there is an unstable node on the critical branch $C_r$ and the slow flow on $C_a$ has $\dot x>0$.}
 \figlab{hesterKappa2}
 }
 \end{center}
              \end{figure}

\begin{remark}
 \PS{Notice that in \figref{hesterKappa2}, and in all subsequent figures, we are following some conventions
which are commonly used in GSPT to distinguish  between the dynamics of the different limiting problems, which need to be shown simultaneously in the same figure:  Green segments indicate ``fast'' orbits (of layer problems) whereas blue indicates slow flow (on critical manifolds)  that is obtained upon (further) desingularization (by speeding up time).
Orbits which approach equilibria (in forward or backward time) in  hyperbolic directions are  highlighted by triple-headed arrows, whereas flows  in slow and center directions are highlighted by single-headed arrows.  Degenerate points, e.g. the fold point $F$, which need to be blown-up, are given individual colors, which are then 
also  used for the corresponding blown up higher-dimensional objects.  }
\end{remark}

\PS{Summing up, we have shown that the upper half of the limit cycle in \SJnew{the Hester} problem can be obtained by
the GSPT analysis of  system \eqref{hesterKappa2}: 
For $x<0$ there is fast dynamics towards $C_a$ along orbits of the layer problem, then there is slow flow along $C_a$ towards the fold 
point $F$, where a fast jump occurs with $y_2$ going to $- \infty$. The $x\SJnew{-}$coordinate of the fold point $F$ is the value $x_j$
defining the right boundary of the segment $\Gamma^2$ of the singular cycle $\Gamma_0 = \Gamma^1 \cup \Gamma^2$,
see \figref{hesterPWS}.

However, we have to keep in mind that solutions of system  \eqref{hesterKappa2} with $y_2 = \SJnew{\mathcal O}(1)$ 
correspond to solutions of system \eqref{hester} in a narrow strip $y = \SJnew{\mathcal O}(\epsilon) $ around the switching surface $\Sigma$.  
To obtain the full limit cycle,  we also have to consider the region $y = \SJnew{\mathcal O}(1) <0$; this is not covered by the scaling \eqref{hestery2}.
\SJnew{The} analysis and the matching of these different regimes based on blow-up \SJnew{is} carried out in the next subsection.
Not surprisingly, we will see that the segment $\Gamma^1$ of the singular cycle $\Gamma_0$ is the orbit of the PWS system  \eqref{hesteryNeg}
starting at $(x_j,0)$, see \figref{hesterPWS}.
}

              \subsection{The blow-up approach}\seclab{hereee}
             To connect \eqref{hesterKappa2} with the PWS system \eqref{hesteryPos} and \eqref{hesteryNeg}, we apply a version of the blow-up method \cite{dumortier_1996,krupa_extending_2001}, see also \cite{2019arXiv190806781U,kristiansen2018a,kristiansen2019e}.

\PS{As always in the blow-up approach one has to consider the extended system} 
             \begin{equation}
\eqlab{hester3}
\begin{split}
               x' &= \epsilon \frac{y}{1+e^{(1+\alpha) y\epsilon^{-1}}},\\
               y' &= \epsilon\left(\frac{-x-2 \gamma y+\mu e^{y\epsilon^{-1}}}{1+e^{(1+\alpha) y\epsilon^{-1}}}-\frac{\kappa \mu e^{(1+\alpha)y\epsilon^{-1}}}{1+e^{(1+\alpha)y\epsilon^{-1}}}\right),\\
               \epsilon' &=0,
               \end{split}
               \end{equation}
\PS{in $\R^3$ obtained from \eqref{hester2} written on the fast time defined by $(\,)'=\epsilon \dot{(\,)}$ 
by adding the trivial equation for $\epsilon$.} 
 This extended system 
has the $(x,y,0)$-plane as a set of equilibria. \PS{The line $x \in \R$, $(y, \epsilon)=(0,0)$ is singular in the sense of lack of smoothness of the vector field \eqref{hester3}
as $\epsilon\rightarrow 0$. Recall that this degenerate line is precisely the switching manifold $\Sigma$ (embedded into~$\R^3$)
of the  piecewise smooth system defined by \eqref{hesteryPos} for $y>0$ and  \eqref{hesteryNeg} for $y<0$. }
%
 We regain smoothness by applying the blow-up transformation
%
%
              \begin{align*}
          r\ge 0,\,(\bar y,\bar \epsilon)\in S^1 \mapsto \begin{cases}
                                                          y&=r\bar y,\\
                                                          \epsilon &=r\bar \epsilon,
                                                         \end{cases}
              \end{align*}
\PS{blowing up the degenerate line to the cylinder $\R \times S^1$ with $r=0$, 
see \figref{hesterBlowup}.
Since we are only interested in $\epsilon \geq 0$, only the part of the cylinder with $\bar \epsilon \geq 0$ is relevant.
The edges $\bar y = \pm 1$, $\bar \epsilon =0$, $r=0$ of this half cylinder will be important later. 
}

\PS{
\begin{remark}
Note the color-coding which will be used frequently: the switching manifold $\Sigma$ shown in magenta in  \figref{hesterPWS}
is blown-up to the cylinder  shown in magenta in \figref{hesterBlowup}.
\end{remark}}

\PS{
The vector field \eqref{hester3} induces a vector field on the blown-up space. As always, the cylinder, corresponding to $r=0$, 
and the plane $\bar \epsilon =0$ are invariant and capture the crucial dynamics, both corresponding to $\epsilon =0$.
Notice, that the scaling \eqref{hestery2}
can be viewed as a directional chart (obtained by setting $\bar \epsilon=1$)
of the blow-up transformation, which covers the side of the cylinder corresponding to $\bar \epsilon >0$.
 In contrast to the usual blow-up approach \cite{dumortier_1996,krupa_extending_2001}, we will not divide by $r$. 
Thus, we find that the slow-fast system  \PS{\eqref{hesterKappa2fast}}
\PS{multiplied by the positive and smooth function
\[ \frac{1}{1 + e^{(1+\alpha)y_2}}, \]
which does not change the orbits,}
describes the blown-up dynamics in the
chart corresponding to  $\bar \epsilon=1$. 
In particular, on the cylinder we recover the limiting dynamics
shown in \figref{hesterKappa2}.

In addition blow-up provides a compactification of  system \eqref{hesterKappa2} as $y_2 \to \pm \infty$. 
Thus the unstable fiber of the fold point $F$ limits now on a point on the edge $\bar y = - 1$, $\bar \epsilon =0$ , see \figref{hesterBlowup}. Actually, the two edges $\bar y = \pm 1$, $\bar \epsilon =0$, $r=0$ of the half cylinder
are lines of equilibria of the blown-up system, 
which  must be studied in directional charts corresponding to $\bar y=\pm 1$. In these charts  
 we  recover the PWS system (with improved hyperbolicity properties) within $\bar \epsilon =0$ after 
dividing out factors of $\bar \epsilon$, respectively. We illustrate our findings\mbox{ in \figref{hesterBlowup}.}}

\begin{figure}[h!]
\begin{center}
\includegraphics[width=.9\textwidth]{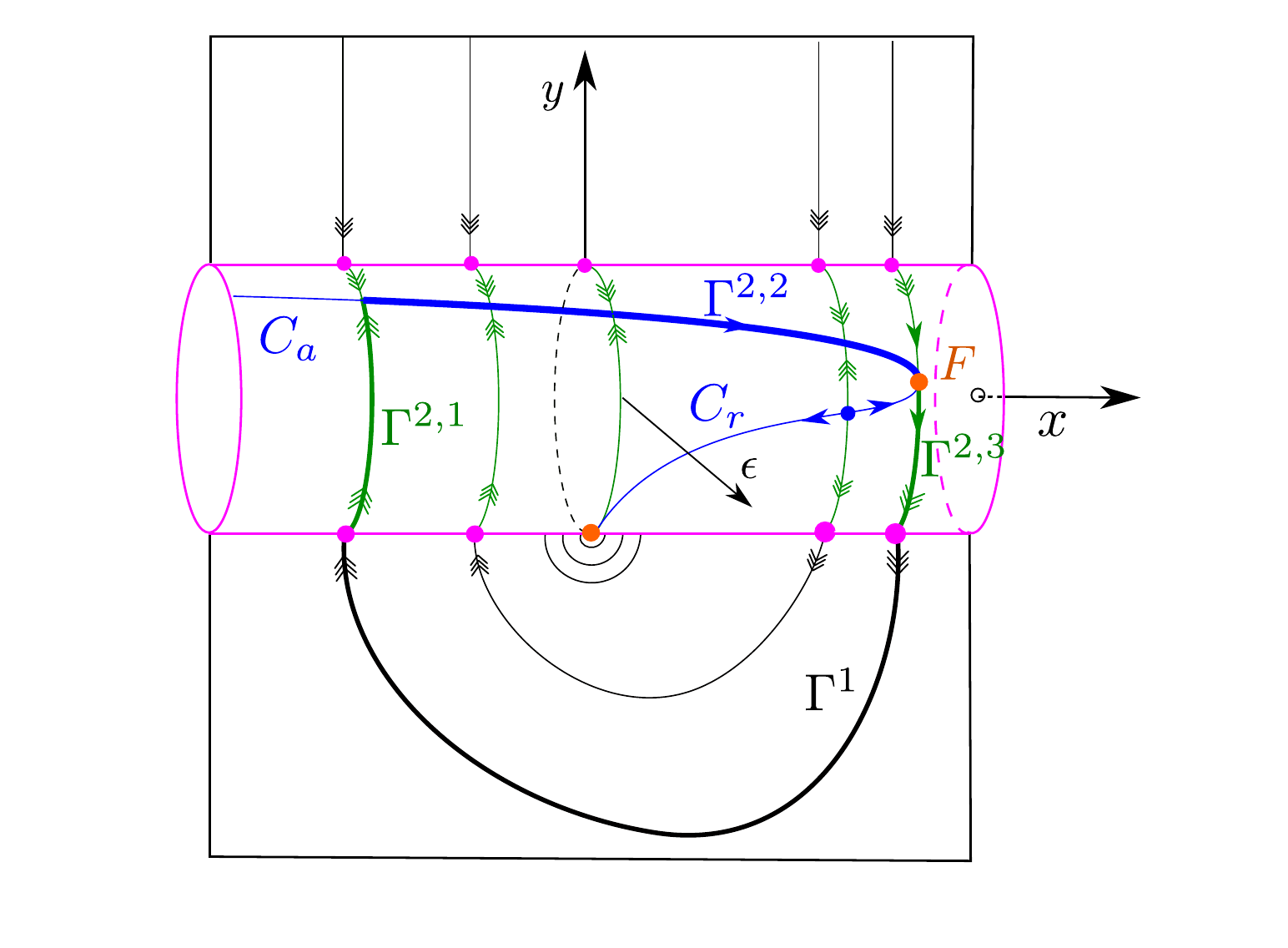}
\PS{\caption{Dynamics of system \eqref{hester3} after blow-up on the cylinder (magenta), corresponding to  $r=0$, and in the 
plane (black) $\bar \epsilon =0$. We show a view from the top, i.e. from $\bar \epsilon >0$, with the orientation of the  $x,y,\epsilon$ axis being indicated by black arrows labeled with $x,y,\epsilon$.
The edges $\bar y = \pm 1$, $\bar \epsilon =0$, $r=0$ of the half cylinder are lines of  equilibria, which are hyperbolic except at the point (orange)
with  $\bar y =-1$, $x=0$. In the plane $\bar \epsilon =0$ we recover  the PWS dynamics (black) of  system \eqref{hesteryPos}, \eqref{hesteryNeg}. On the cylinder we recover the layer problem (green) and the reduced problem (blue) of system  \eqref{hesterKappa2}. 
This allows to define an improved singular cycle 
\[ \Gamma_0 = \Gamma^1 \cup \Gamma^{2,1} \cup \Gamma^{2,2} \cup \Gamma^{2,3} \]
which perturbs to a true cycle $\Gamma_{\epsilon}$ for $\epsilon  \ll 1$. } 
\figlab{hesterBlowup}
}
 \end{center}
              \end{figure}

 The required analysis -- \PS{to establish this rigorously} -- is standard, see e.g. \cite{Gucwa2009783,2019arXiv190806781U,kristiansen2018a,kristiansen2019e}. \PS{Also, very similar computations are 
carried out in detail for the (more complicated)  Le Corbeiller system below. We therefore only summarise the results 
for the Hester system.}
Along the edges $\bar y=\pm 1$, $\bar \epsilon=0$ we find lines of equilibria (magenta in \figref{hesterBlowup}) having a hyperbolic saddle-structure, except for $x=0,\bar y=-1,\bar \epsilon=0$ (orange circle) which is fully non-hyperbolic. This structure provides an improved singular cycle \[ \Gamma_0 = \Gamma^1 \cup \Gamma^{2,1} \cup \Gamma^{2,2} \cup \Gamma^{2,3}, \] (thick closed curve in \figref{hesterBlowup}), having \PS{good}  hyperbolicity properties except at the fold $F$ of the critical manifold $C$ (blue).
\PS{Therefore},  it is easy to perturb this singular cycle into an actual limit cycle for $0<\epsilon\ll 1$ by first considering a return mapping to the section $\{y=-\delta\}$ near $x_j$, for example, and then applying e.g. \cite{krupa_extending_2001} to the passage near fold $F$ (working in the scaled coordinates \eqref{hestery2}) to show that the Poincar\'e map is a strong contraction. We leave out the details because they are standard. See again \cite{Gucwa2009783} for a related system where more details are provided. 
 
\section{\PS{Blow-up analysis of the Le Corbeiller system}} \seclab{corgeo}
\PS{In the limit $\epsilon \to 0$ the transformed Le Corbeiller system \eqref{corbeiller2} 
limits on the PWS system \eqref{corbeilleryPos}, \eqref{corbeilleryNeg} which has the singular cycle
$\Gamma_0 = \Gamma^1 \cup \Gamma^2$, see \figref{corbeillerPWS}.}

\PS{Motivated by the similarity of the two systems and by the success of the scaling approach for the Hester system, we begin our analysis by considering the original Le Corbeiller system \eqref{corbeiller} by rescaling
\begin{equation}
\eqlab{corbeiller_scaling}
y = \epsilon y_2, 
\end{equation}
i.e. by zooming into the switching manifold $\Sigma$.} This produces the following system
\begin{equation}\eqlab{lc_K2}
\begin{split}
\dot x &= \epsilon \PS{y_2} + a , \\
\epsilon \dot y_2 &= -x + b \epsilon y_2 \left( 2 - e^{y_2} \right).
\end{split}
\end{equation}
System \eqref{lc_K2} is a slow-fast system in standard form, with layer problem
\begin{equation}\eqlab{lc_layer}
\begin{split}
x' &= 0 , \\
y_2' &= -x ,
\end{split}
\end{equation}
which has a degenerate, non-hyperbolic line $L_2 = \{ (x,y_2): x = 0 \}$ \PS{of equilibria}.
Away from $x = 0$ the flow is trivial and regular (upward for $x < 0$, downward for $x > 0$), {see \figref{corbeillerKappa2}}. 
\begin{figure}[h!]
\begin{center}
\includegraphics[width=.7\textwidth]{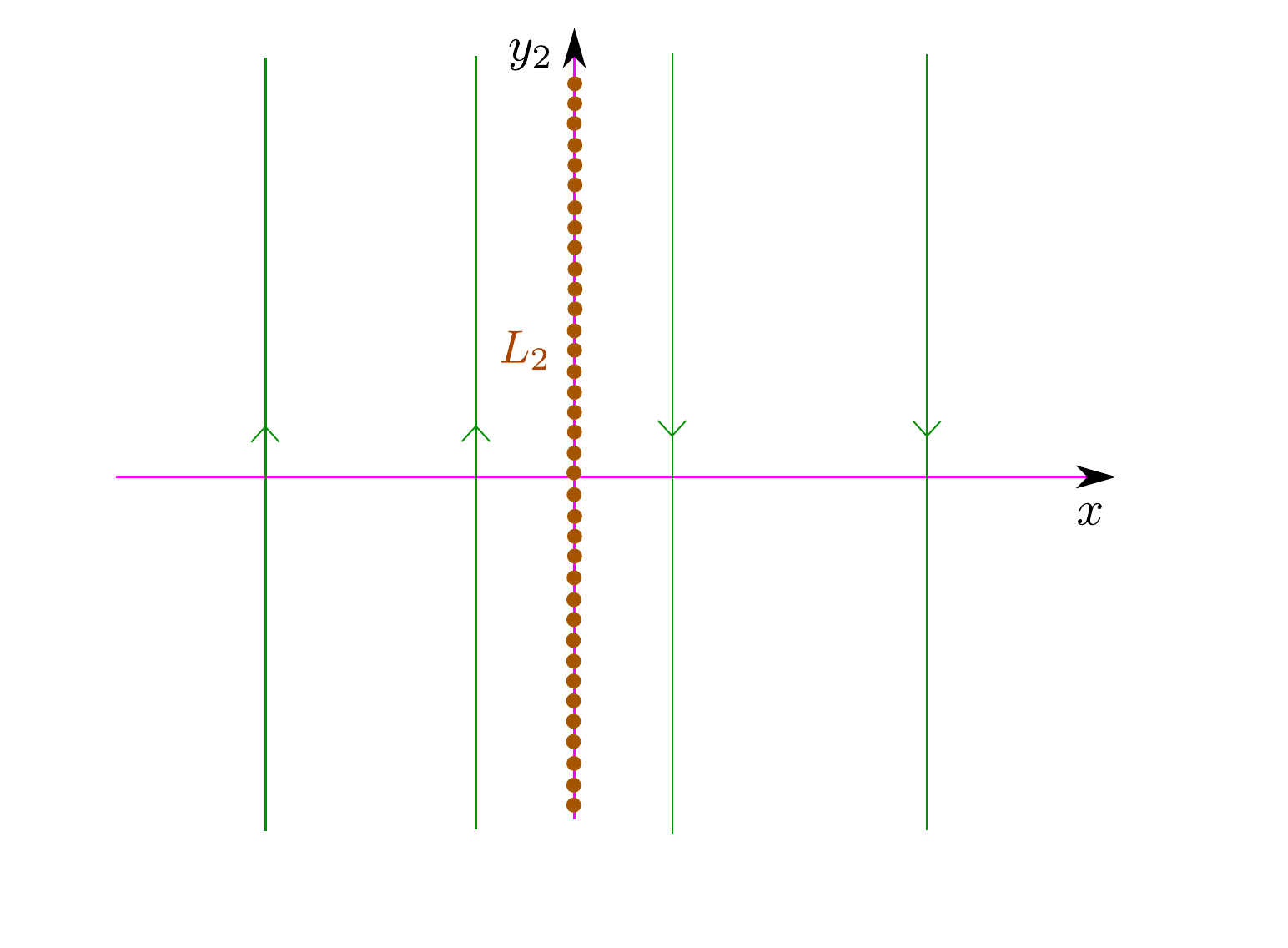}
\caption{Dynamics of system  \eqref{lc_K2}  in the singular limit $\epsilon =0$. 
Orbits of the layer problem \eqref{lc_layer} are shown in green, the flow going upwards for $x<0$ and downwards for $x>0$. The set $L_2$ defined by $x=0$ within $\epsilon=0$ is a line of degenerate singularities.}
 \figlab{corbeillerKappa2}
 \end{center}
              \end{figure}
\PS{
\begin{remark}
In contrast to the analysis of the Hester model based on the rescaling   \eqref{hestery2}  in  \secref{scalingapproaches}
the slow-fast dynamics of the the rescaled system \eqref{lc_K2} is quite degenerate, e.g.
there exists no normally hyperbolic critical manifold. At this stage the rescaled system seems to capture very little of the observed limit cycle. Nevertheless, the flow defined by the rescaled system \eqref{lc_K2} and in particular the nonhyperbolic line
$L_2$ of equilibria, will play an important role in a refined analysis of the limit cycle based on blow-up. More precisely, the flow
of the rescaled system \eqref{lc_K2} will be recovered as the flow on the blow-up of the switching manifold $\Sigma$ to a cylinder, 
see \figref{corbeillerBlowup}. However, the full resolution of the Le Corbeiller model will require more than just
one cylindrical blow-up due to its more singular dependence on $\epsilon$. \end{remark}
}

\PS{
To obtain a full resolution and to connect \eqref{lc_K2} with the PWS system \eqref{corbeilleryPos} and \eqref{corbeilleryNeg} we study again the extended system 
\begin{align}
 \eqlab{corbeiller2extended}
 \begin{split}
 x' &=  \epsilon \frac{y+a}{1+e^{y\epsilon^{-1}}},\\
y' &= \epsilon \left ( \frac{-x+2b y}{1+e^{y\epsilon^{-1}}}-\frac{b y e^{y\epsilon^{-1}}}{1+e^{y\epsilon^{-1}}} \right ),\\
\epsilon'  & =0,
\end{split}
  \end{align}
  obtained by transforming \eqref{corbeiller2} to the fast time scale $\tau = t / \epsilon$ end adding the trivial equation for $\epsilon$.}

The set defined by $(x,y,0)$ is then a plane of equilibria, with the subset given by $y=0$ being extra singular due to the lack of smoothness there. We gain smoothness by applying the blow-up transformation
\begin{align}
\eqlab{cyl_blowup_map1}
r\ge 0,\,(\bar y,\bar \epsilon)\in S^1 \mapsto \begin{cases}
y&=r\bar y,\\
\epsilon &=r\bar \epsilon,
\end{cases}
\end{align}
to the extended system.
\PS{By this blow-up transformation the line $\{(x,0,0)$, $x \in \R\}$ corresponding to the switching manifold $\Sigma \times \{0\}$
is blown-up to a cylinder. Again only the part of the cylinder corresponding to $\bar \epsilon \geq 0$ is relevant for our analysis,
see \figref{corbeillerBlowup}.}

Three coordinate charts
\[
K_1: \bar y = -1, \qquad K_2: \bar \epsilon = 1, \qquad K_3: \bar y = 1,
\]
\PS{are necessary to analyze the dynamics on the blown-up space.}
We will make use of the usual subscript notation to specify coordinates in each chart, defining chart-specific coordinates as follows:
\begin{align}
&K_1:  y = - r_1, \,\,\quad \epsilon = r_1 \epsilon_1 , \eqlab{K1corbeiller}\\
&K_2: y = r_2 y_2 , \,\quad \epsilon = r_2 , \eqlab{K2Here}\\
&K_3: y = r_3, \,\,\,\,\,\,\,\quad  \epsilon = r_3 \epsilon_3.\eqlab{K3} 
\end{align}
Transition maps between a chart $K_i$ and $K_j$ $(i \neq j)$ will be denoted by $\kappa_{ij}$, and are given here by
\begin{equation}
\eqlab{trans1}
\begin{aligned}
&\kappa_{12} : r_1 = - r_2 y_2, && \epsilon_1 = - y_2^{-1} , && y_2 < 0, \\
&\kappa_{23} : y_2 = \epsilon_3^{-1}, && r_2 = r_3 \epsilon_3, && \epsilon_3 > 0 .
\end{aligned}
\end{equation}
\KUK{
We will also adopt the convention of denoting a set $\gamma$ by $\gamma_i$ when viewed in a particular coordinate chart $K_i$.}
\PS{
\begin{remark}
Notice again that the scaling \eqref{corbeiller_scaling} corresponds to the directional chart $K_2$, which is hence customarily referred to
as \SJnew{a} ``scaling chart''. In the following we will call the chart $K_1$ \SJnew{an} ``entry chart'', because in this chart the singular flow relevant for the limit cycle
returns to the edge $(\bar y,\bar \epsilon) = (-1,0)$ of the cylinder  and continues from there on the cylinder, see  see \figref{corbeillerBlowup}.
Note that in charts $K_i$, $i=1,3$ the cylinder corresponds to $r_i=0$. 
These charts provide a compactification of the variable $y_2$ from
the scaling chart $K_2$, i.e. $y_2 \to - \infty$ in the plane $r_2=0$  corresponds to $\epsilon_1 \to 0$ in the plane $r_1=0$, and 
 $y_2 \to  \infty$ in the plane $r_2=0$  corresponds to $\epsilon_3 \to 0$ in $r_3=0$, see \eqref{trans1}. \end{remark}
}

Since it is the easiest chart to analyze we start with the scaling chart $K_2$.
\subsubsection*{\PS{Chart $K_2$}}

\PS{The governing equations  in chart $K_2$ are
\begin{equation}\eqlab{lc_K2extended}
\begin{split}
x' &= \frac{r_2 ( r_2 y_2 + a)}{1 + e^{y_2}} , \\
y'_2 &= \frac{-x + r_2  b y_2 \left( 2 - e^{y_2} \right)}{1 + e^{y_2}},\\
r'_2  & = 0,
\end{split}
\end{equation}
which is system \eqref{lc_K2} written on the fast time scale and multiplied by the positive and smooth factor
\[  \frac{1}{1 + e^{y_2}}.\]
Thus for $r_2$ =0, i.e. on the cylinder, we recover the layer problem and in particular the fully nonhyperbolic line $L_2$ corresponding to
$x=0$, see \figref{corbeillerBlowup}.
}

\subsubsection*{\PS{Chart $K_1$}}
\PS{We now discuss the dynamics in the entry chart $K_1$. The governing equations are}
\begin{equation}\eqlab{sys_K1}
\begin{split}
x' &= r_1 ( a - r_1 ) ,  \\
r_1' &= r_1 \left( x + b r_1 \left( 2 - e^{- \epsilon_1^{-1}}\right) \right) ,   \\
\epsilon_1' &=  - \epsilon_1 \left( x + b r_1 \left( 2 - e^{- \epsilon_1^{-1}}\right) \right) .
\end{split}
\end{equation}
\PS{These equations are obtained by transforming to the new coordinates followed by a
desingularization of the vector field corresponding to dividing out} a factor of $\epsilon_1 (1+e^{-\epsilon_1^{-1}})^{-1}$, \PS{which is smooth for $\epsilon_1 \geq 0$.}
 \PS{Note,
that} it is this division that allows us to recover the PWS system within $\epsilon_{1}=0$, see \lemmaref{K1Lemma}. In particular, there is an isolated equilibrium at $p_1:=(-2 b a, a, 0)$, and two lines of equilibria
\[
l_{s,1} = \{(x, 0, 0) : x \in \mathbb R \} \qquad
\text{and}
\qquad
L_1 = \{(0, 0, \epsilon_1) : \epsilon_1 \geq 0 \}.
\]
The following lemma summarizes  important dynamical features of this system, which are illustrated in \figref{corbeillerBlowup}.

\begin{lemma}\lemmalab{K1Lemma}
	The following hold for system \eqref{sys_K1}:
	\begin{enumerate}
		\item[(i)] The linearization about any point $(x,0,0)$ on the line of steady states $l_{s,1}$ has eigenvalues
		\[
		\lambda  = 0, x, -x ,
		\]
		and thus, \PS{the line $l_{s,1}$ is a line of saddle points except at the }origin $q_1 := (0,0,0)$ which is fully non-hyperbolic. 
		\item[(ii)] The line of steady states $L_1 $ is fully non-hyperbolic, and coincides (where domains overlap) with the line $L_2$ observed in the $K_2$ chart.
		\item[(iii)] \PS{The plane $\epsilon_1=0$ is invariant.  Within this plane 
               points $(x,0,0) \in l_{s,1}$ with  $x > 0$ are hyperbolic repelling, respectively  hyperbolic attracting for $x<0$. 
				The flow in this plane in $r_1 >0$ is precisely the flow of the $y<0$ system \eqref{corbeilleryNeg} multiplied by a factor $r_1$, 
				through the identification $y =-r_1$.			
		The equilibrium  $p_1 = (-2 b a, a, 0 )$ is an unstable focus within the invariant $\epsilon_1=0$ plane for any $b \in (0,1)$.  Furthermore, there is a quadratic tangency between the flow of \eqref{sys_K1}$|_{\epsilon_1 = 0}$ and $l_{s,1}$ at $q_1$, corresponding to the tangency observed in \secref{the_le_corbeiller_system}.   The equilibrium $p_1$ extends to \mbox{$\epsilon_1 >0$} as a line of equilibria which  coincides upon coordinate change 
		\PS{back to the original variables} with the true equilibrium of the system identified in \secref{the_le_corbeiller_system}.
		\item[(iv)] The plane $r_1=0$ is invariant.
		Within this plane points $(x_b,0,0) \in l_{s,1}$ with  $x_b > 0$ are hyperbolic attracting, respectively  hyperbolic repelling for $x_b<0$, 
		with stable, reps. unstable manifolds given by $x = x_b$.\\	}	
					\end{enumerate}
\end{lemma}

\begin{proof}
	The statement (i) follows immediately from the linearization along $l_{s,1}$, and the simple form of the equations when restricted to the invariant plane $r_1 = 0$:
	\begin{equation}
	\begin{split}
	x' &= 0 , \\
	\epsilon_1' &= - \epsilon_1 x .
	\end{split}
	\end{equation}
	The statement (ii) follows by an application of the transition map $\kappa_{12}$ given in \eqref{trans1}, and the statements (iii) and (iv)  follow by simple calculations made using the system governing the dynamics in the invariant plane $\epsilon_1 = 0$:
	\begin{equation}
	\begin{split}
	x' &= r_1 (a - r_1) , \\
	\PS{r_1}' &= r_1 (x + 2 b r_1) .
	\end{split}
	\end{equation}
\end{proof}

\subsubsection*{\PS{Chart $K_3$}}
\label{sec:K3_chart}

Dynamics in chart $K_3$ are governed by
\begin{equation}
\eqlab{eq:sys_K3}
\begin{split}
x' &= r_3 e^{-\epsilon_3^{-1}} ( a + r_3 ) ,  \\
r_3' &= - r_3 \left( b r_3 + e^{-\epsilon_3^{-1}} \left(x - 2 b r_3 \right) \right) ,   \\
\epsilon_3' &= \epsilon_3 \left( b r_3 + e^{-\epsilon_3^{-1}} \left(x - 2 b r_3 \right) \right) ,
\end{split}
\end{equation}
after desingularization through division of the right hand side by \KUK{$\epsilon_3 ( 1+e^{-\epsilon_3^{-1}})^{-1}$}. In particular, there are two lines of equilibria:
\[
l_{e,3} = \{(x, 0, 0) : x \in \mathbb R \} \qquad
\text{and}
\qquad
L_3 = \{(0, 0, \epsilon_3) : \epsilon_3 \geq 0 \}.
\]
Restricted to the invariant planes $r_3 = 0$ and $\epsilon_3 = 0$, we obtain the following equations
	\begin{equation}
	\eqlab{eqn:exp_nh_loss}
	\begin{split}
		x' &= 0 , \\
		\epsilon_3' &= \epsilon_3 x e^{-\epsilon_3^{-1}} ,
	\end{split}
	\end{equation}
	and
	\begin{equation*}
	\begin{split}
	x' &= 0 , \\
	r_3' &= - b r_3^2 ,
	\end{split}
	\end{equation*}
	respectively. It is clear that any point in either $l_{e,3}$ and $L_3$ is fully non-hyperbolic. We notice in particular the `exponential loss 
	of hyperbolicity' in \eqref{eqn:exp_nh_loss} \PS{as $\epsilon_3 \to 0$}.
 The following lemma summarizes the important features of this system, which are illustrated in \figref{corbeillerBlowup}.

\begin{lemma}
	\lemmalab{K3lem}
	The following hold for system \eqref{eq:sys_K3}:
	\begin{enumerate}
		\item[(i)] The line of steady states $l_{e,3}$ is fully non-hyperbolic.
		\item[(ii)] The line of steady states $L_3$ is fully non-hyperbolic, and coincides where domains overlap with the non-hyperbolic line $L_2$ observed in the $K_2$ chart.
		\item[(iii)] \PS{The plane $\epsilon_3=0$ is invariant. The flow in this plane in $r_3 >0$ is precisely the $y>0$ system \eqref{corbeilleryPos} multiplied by a factor $r_3$, through the identification $y =r_3$.} The flow in $\epsilon_3=0,r_3>0$ is parallel to the $r_3-$axis and toward $l_{e,3}$. 
	\item[(iv)]The plane $r_3=0$ is invariant. The flow in this plane for $\epsilon_3 > 0$ is parallel to the $x-$axis, and toward (away from) the line $l_{e,3}$ for $x<0$ ($x>0$).
	\end{enumerate}
\end{lemma}

\begin{proof}
Straightforward. 
\end{proof}

\PS{
We briefly sum up the results obtained by \SJnew{the} above analysis, which are illustrated in  in \figref{corbeillerBlowup}.
 We have achieved a certain desingularization of  \eqref{corbeiller2extended} by means of the
blow-up transformation \eqref{cyl_blowup_map1}. On the cylinder $r=0$ we have recovered the layer problem of \eqref{lc_K2} in particular the
nonhyperbolic line $L$.  In the invariant plane $\bar \epsilon =0$ we have recovered the PWS system \eqref{corbeilleryNeg} and \eqref{corbeilleryPos} 
for $\bar y <0$ respectively $\bar y >0$.  At the invariant line $l_s$ of saddle type we have gained hyperbolicity away from the nonhyperbolic point $q$.
The invariant lines $l_e$ and $L$ are still fully degenerate and will be treated by further blow-ups.}

\begin{figure}[h!]
\begin{center}
\includegraphics[width=.55\textwidth]{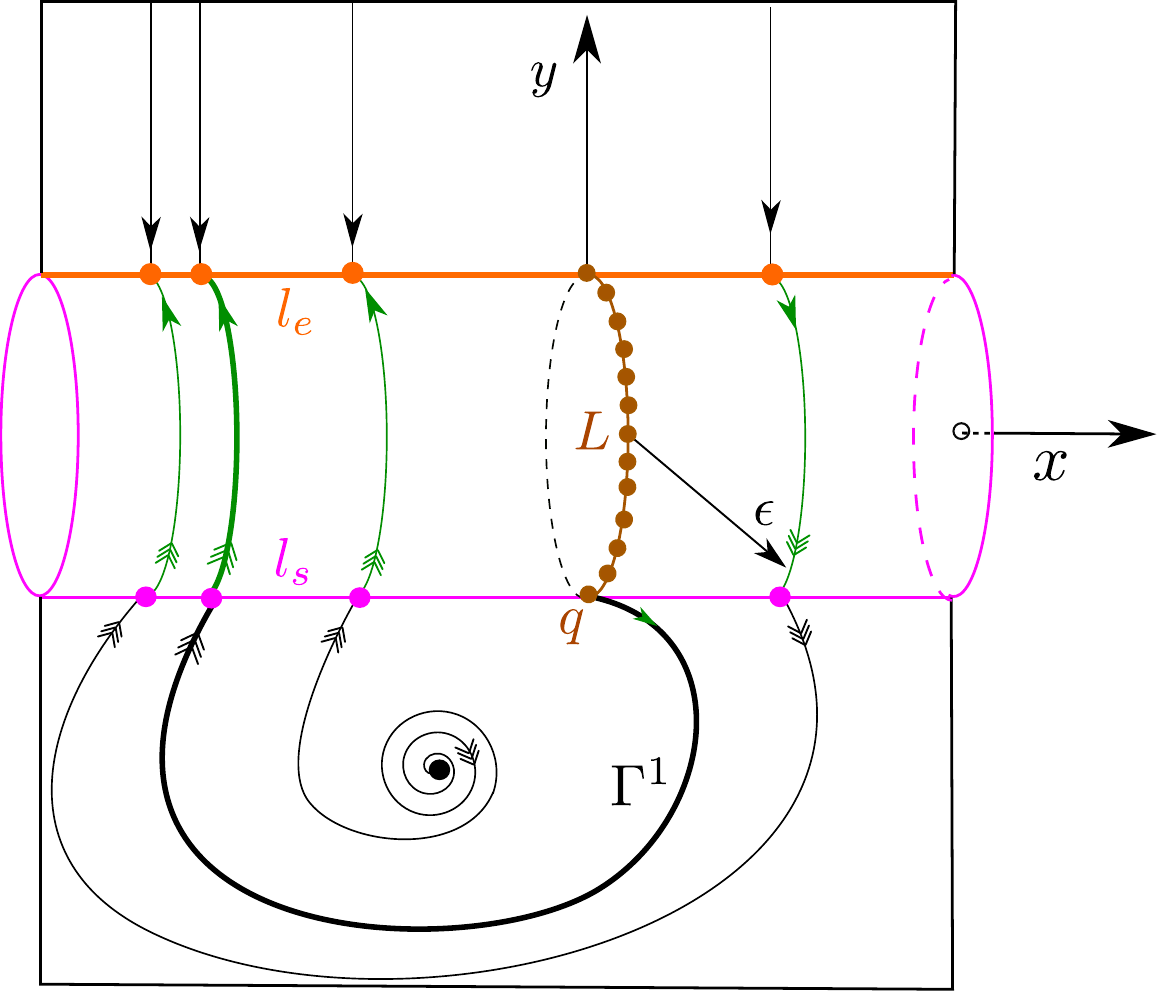}
 \caption{\PS{Dynamics of system \eqref{corbeiller2extended}  after blow-up on the cylinder (magenta), corresponding to  $r=0$, and in the 
plane (black) $\bar \epsilon =0$. We show a view from the top, i.e. from $\bar \epsilon >0$, with the orientation of the  $x,y,\epsilon$ axis being indicated by black arrows labeled with $x,y,\epsilon$. Fast flow on the cylinder is shown in green. 
The line $l_s$ is a line of hyperbolic equilibria, except for the fully degenerate point $q \in l_s \cap L$ (brown).
 The line $l_e$ (orange) and the $L$ (brown) are lines of completely degenerate equilibria. }}
  \figlab{corbeillerBlowup}
 \end{center}
              \end{figure}

\subsection{Blow-up in the exponential regime}
\seclab{ExpBlowup}

To deal with the (exponential) loss of normal hyperbolicity associated with the line $l_{e,3}$ in system \eqref{eqn:exp_nh_loss}, we use the approach put forward in \cite{kristiansen2017a}: introduce the auxiliary variable
\[
q = e^{-\epsilon_3^{-1}},
\]
\SJnew{and} extend the phase space by including an equation for $q'$. \PSnew{For improved readability (in subsequent blow-up transformations) we drop the subscripts in the coordinates
of chart $K_3$ in the following. Thus, the following system}
\begin{equation}
\eqlab{extended_K3}
\begin{split}
x' &= r \epsilon q (a + r), \\
r' &= - r \epsilon \left( b r + q \left( x - 2 b r \right) \right), \\
\epsilon' &= \epsilon^2 \left( b r + q \left( x - 2 b r \right) \right), \\
q' &= q \left( b r + q \left( x - 2 b r \right) \right),
\end{split}
\end{equation}
is obtained after a further multiplication of the right hand side by $\epsilon$; we therefore basically undo the division by $\epsilon$ used to obtain \eqref{eq:sys_K3}. 
It  is worth keeping in mind that
\begin{equation}
\eqlab{eq:Q}
Q = \left\{(x, r, \epsilon, q) : q = e^{-\epsilon^{-1}} \right\}
\end{equation}
is invariant. We will often use this fact, utilizing it when it helps in the analysis.

In  system \eqref{extended_K3}, the non-hyperbolic line \PSnew{of equilibria} $L_3$ in \eqref{eq:sys_K3} shows up as the $Q$-restricted subset:
\begin{align*}
 L_e = \mathcal P \cap Q = \left\{\left(0, 0, \epsilon, e^{-\epsilon^{-1}}\right) : \epsilon \geq 0 \right\},
\end{align*}
of the non-hyperbolic plane \PSnew{of equilibria}
\begin{equation}
\eqlab{nonhyp_plane1}
\mathcal P = \left\{(0, 0, \epsilon, q) : \epsilon, q \geq 0 \right\}.
\end{equation}
The subscript `$e$' (for `exponential') in $L_e$ signifies that we are considering the object in the extended, four-dimensional system \eqref{extended_K3}.
Similarly, the line  \PSnew{of equilibria}
\[
\mathfrak l_e = \left\{(x, 0, 0, 0) : x \in \mathbb R \right\},
\]
can be viewed as an improved version of the \PSnew{line of degenerate equilibria}  $l_{e,3}$.
It is the intersection of two separate degenerate planes: $\{(x,0,\PS{0,q}):\,x\in \mathbb R,\,q\ge 0\}$ 
and 
\begin{equation}
\eqlab{nonhyp_plane2}
\mathcal P_e = \left\{(x, 0, \epsilon, 0) : x \in \mathbb R, \epsilon \geq 0 \right\} .
\end{equation}


\PSnew{Although} the \PSnew{extended} system \eqref{extended_K3} is clearly \PSnew{quite} degenerate, the main \PSnew{advantage of 
this system is}  that the loss of normal hyperbolicity is \PSnew{now \textit{algebraic}, and  blow-up methods are applicable.} 
\PSnew{Therefore,}  we introduce a blow-up of the plane\PSnew{ of equilibria $\mathcal P_e$  
of system \eqref{extended_K3}}  via the map
  \begin{align}
\rho\ge 0,\,(\bar r,\bar q)\in S^1 \mapsto \begin{cases}
r&=\rho\bar r,\\
q &=\rho \bar q.
\end{cases}\eqlab{blowuple}
\end{align}

We are primarily interested in the dynamics observable in the entry chart
\[
\mathfrak K_1: \bar q =1,
\]
which has chart specific coordinates
\begin{align}
\eqlab{mathfrakK1}
r = \rho_1 r_1,  \qquad q = \rho_1.
\end{align}
In this chart, we obtain the system
\begin{equation}
\eqlab{eq:sys_qChart}
\begin{split}
x' &= \rho_1 r_1 \epsilon (a + \rho_1 r_1) , \\
r_1' &= - r_1 (1 + \epsilon) \left( x + b r_1 \left( 1 - 2 \rho_1 \right) \right) , \\
\epsilon' &= \epsilon^2 \left( x + b r_1 \left( 1 - 2 \rho_1 \right) \right) , \\
\rho_1' &= \rho_1 \left( x + b r_1 \left( 1 - 2 \rho_1 \right) \right) ,
\end{split}
\end{equation}
after a suitable desingularization (division by $\rho_1$). Note that the set
\[
Q_1 = \left\{ (x, r_1, \epsilon, \rho_1) : \rho_1 = e^{- \epsilon^{-1} } \right\}
\]
is invariant under the flow induced by \eqref{eq:sys_qChart}. The  non-hyperbolic plane $\mathcal P$ defined in \eqref{nonhyp_plane1} becomes a plane of equilibria for system \eqref{eq:sys_qChart},
\[
\mathcal P_1 = \{(0,0,\epsilon,\rho_1): \epsilon \geq 0, \rho_1 \geq 0\} , 
\]
and
\[
 \mathfrak l_{e,1} = \left\{(x, 0, 0, 0) : x \in \mathbb R \right\}
\]
constitutes a line of equilibria for the system \eqref{eq:sys_qChart}. We also have
\begin{equation}
\eqlab{L1}
L_{e,1} = \mathcal P_1 \cap Q_1 = \left\{\left(0, 0, \epsilon, e^{- \epsilon^{-1}} \right) : \epsilon \geq 0 \right\} ,
\end{equation}
which will be important for our analysis.

\begin{lemma}
	\lemmalab{expK1lemma}
	The following hold for the system \eqref{eq:sys_qChart}:
	\begin{enumerate}
		\item[(i)] Within $\epsilon = 0$, there exists a \KUK{two-dimensional} manifold of equilibria given by
		\[
		C_1 = \left\{\left(x, - \frac{x}{b ( 1 - 2 \rho_1)} , 0, \rho_1 \right) : x \leq 0, \rho_1\in[0,\beta_1]  \right\}.
		\]
		The linearization about any point in $C_1$ has  a single non-trivial eigenvalue
		\[
		\lambda = \frac{x}{b ( 1 - 2 \rho_1)} \leq 0 ,
		\]
		where $\beta_1 > 0$ is chosen sufficiently small for the inequality to hold. Accordingly, $C_1$ is normally hyperbolic and attracting \KUK{within $\epsilon=0$} for $x<0$, and non-hyperbolic along the line $\mathfrak I^1 = \{(0, 0, 0, \rho_1) : \rho_1 \in [0,\beta_1] \}$.
		\item[(ii)] Within $\rho_1 = 0$, there exists a \KUK{two-dimensional} manifold of equilibria given by
		\[
		S_1 = \left\{\left(x, - \frac{x}{b} , \epsilon, 0 \right) : x \leq 0, \epsilon \in \left[0,\beta_2\right]  \right\}.
		\]
		The linearization about any point in $S_1$ has  a single non-trivial eigenvalue
		\[
		\lambda = \frac{x}{b} \leq 0.
		\]
		Accordingly, $S_1$ is normally hyperbolic and attracting within \KUK{within $\rho_1=0$} for $x<0$, and non-hyperbolic along the line $\mathfrak I^2 = \{(0, 0, \epsilon, 0) : \epsilon \in [0,\beta_2] \}$. Note that $\mathfrak I^1 \cap \mathfrak I^2 = \{(0,0,0,0)\}$.
		\item[(iii)] The manifolds $C_1$ and $S_1$ intersect the invariant domain of interest $Q_1$ along a one-dimensional manifold of equilibria $\mathcal C_1$ satisfying
		\[
		\mathcal C_1 = C_1 \cap Q_1 = S_1 \cap Q_1 = C_1 \cap S_1 = \left\{\left(x, - \frac{x}{b}, 0 , 0\right) : x \leq 0 \right\} .
		\]
		Considered within the invariant plane $\epsilon = \rho_1 = 0$, the manifold $\mathcal C_1$ is normally hyperbolic and attracting for $x<0$, being non-hyperbolic at the point $P:\,(\SJ{0,0,} 0,0)$.
		\item[(iv)] The plane $\mathcal P_1$ is fully non-hyperbolic, and $L_{e,1}$ coincides with the image of the non-hyperbolic line $L_3$ observed in system \eqref{eq:sys_K3} (recall \lemmaref{K3lem}).
		\item[(v)] The eigenvalues associated with the linearization along $\mathfrak l_{e,1}$ are given by
		\[
		\lambda = 0, \ -x, \ 0, \ x ,
		\]
		implying that $\mathfrak l_{e,1}$ is a line of partially hyperbolic saddles for $x \neq 0$, and non-hyperbolic for $x = 0$.
	\end{enumerate}
\end{lemma}

\begin{proof}
	The statement (i) follows after linearization of the system obtained by restricting to the invariant hyperplane $\epsilon = 0$:
	\begin{equation}
	\begin{split}
	x' &= 0 , \\
	r_1' &= - r_1 \left( x + b r_1 \left( 1 - 2 \rho_1 \right) \right) , \\
	\rho_1' &= \rho_1 \left( x + b r_1 \left( 1 - 2 \rho_1 \right) \right).
	\end{split}
	\end{equation}
	Similarly, the statement (ii) is obtained after linearization of the system obtained by restricting to the invariant hyperplane $\rho_1 =0$:
	\begin{equation}
	\label{eq:sys_qChart_cylinder}
	\begin{split}
	x' &= 0, \\
	r_1' &= - r_1 (1 + \epsilon) \left( x + b r_1 \right) , \\
	\epsilon' &= \epsilon^2 \left( x + b r_1 \right) .
	\end{split}
	\end{equation}
	Statement (iii) follows immediately from (i)-(ii) and the observation that $\rho_1 = e^{-\epsilon^{-1}} = 0 \implies \rho_1 = \epsilon = 0$.
	
	Statement (iv) follows after linearization of the system \eqref{eq:sys_qChart}, together with an application of the blow-down transformation.
	
	Statement (v) also follows after linearization of the system \eqref{eq:sys_qChart}.
\end{proof}

\begin{figure}[h!]
	\centering
	\includegraphics[scale=0.8]{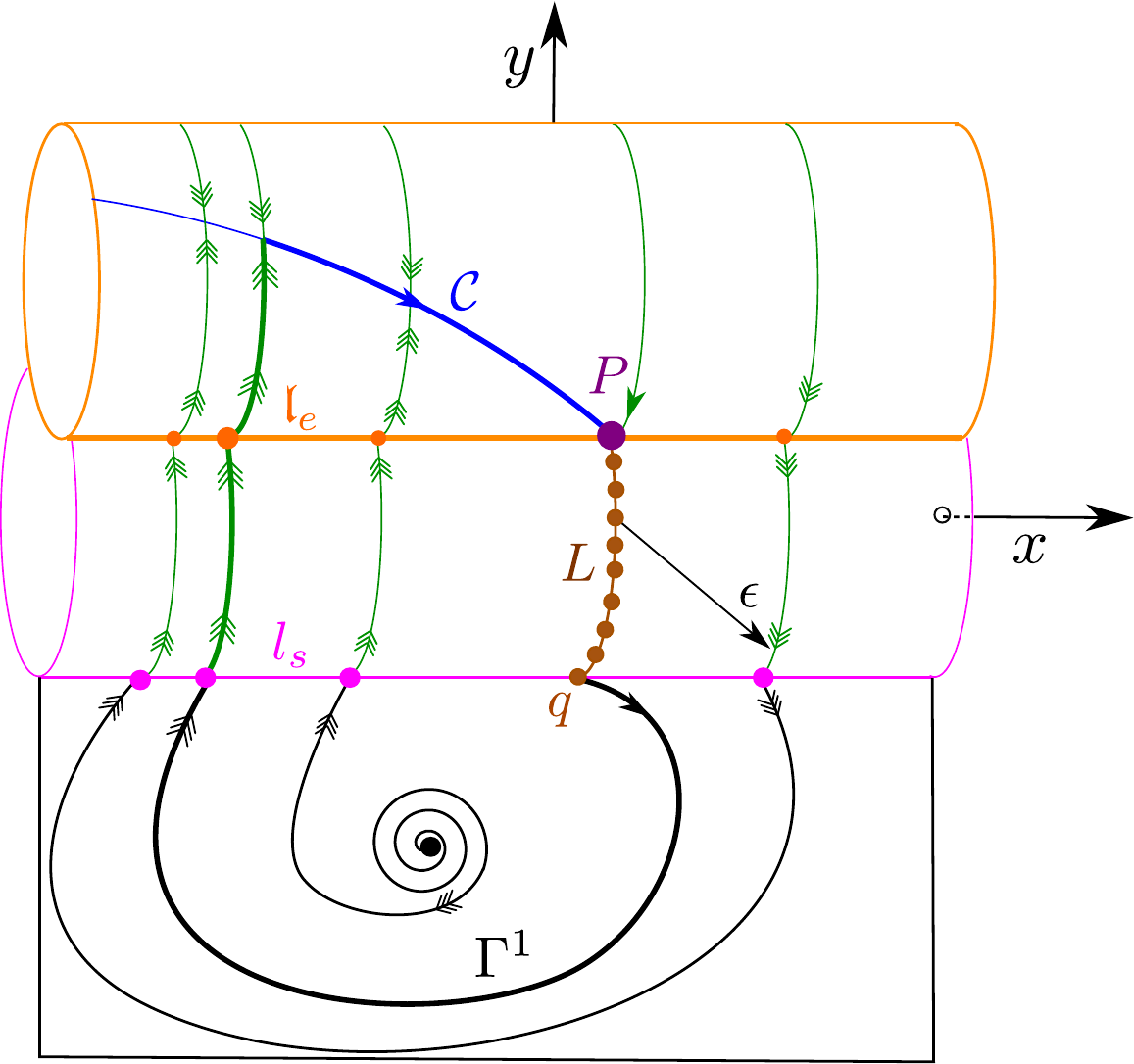}
	\caption{\PSnew{Dynamics  upon blow-up \eqref{blowuple} of $l_e$ on the cylinder (orange), corresponding to $\rho=0$. 
We show a view from the top, i.e. from $\bar \epsilon >0$, with the orientation of the  $x,y,\epsilon$ axis being indicated by black arrows labeled with $x,y,\epsilon$. 
Fast flow on the orange cylinder is shown in green.  In comparison with \figref{corbeillerBlowup} we have gained hyperbolicity at the line of equilibria $\mathfrak l_{e}$ 
and we are able to identify a normally hyperbolic critical manifold $\mathcal C$ (blue).  The line (half-circle)  $L$ (brown) of equilibria and, in particular, the equilibrium  $P$ (purple) are still degenerate and will \PSnew{require additional blow-up transformations.}}
	}\figlab{corbeillerBlowup2}
	\end{figure}

\PSnew{The most important properties of system \eqref{eq:sys_qChart} described in \lemmaref{expK1lemma}  are 
sketched in \figref{corbeillerBlowup2}. }
\PSnew{
\begin{remark}Clearly,  not all properties of the four-dimensional system \eqref{eq:sys_qChart}
are visible in \figref{corbeillerBlowup2}, which focuses on explaining how the second blow-up improves the situation shown in 
 \figref{corbeillerBlowup}. Keep in mind, that the second blow-up \eqref{blowuple}
affects only a small neighborhood of the line $l_e$ in   \figref{corbeillerBlowup}.  The degenerate line $l_e$ is blown up to 
the orange cylinder, corresponding to $\rho=0$, i.e. $\rho_1 =0$ in chart \eqref{mathfrakK1}. In this chart $r_1 >0$ means ``going up''
on the cylinder, which due to $y = r_3 = \rho_1 r_1$ justifies our use of the $y$-axis in the figure.
\end{remark}}

\medskip
The preceding observations are sufficient to prove the following result, which will in turn allow for a proof of  \lemmaref{slowManifold-Corb}.

\begin{lemma}
	\lemmalab{BU_slowManifold}
	Fix a closed interval $I\subset (-\infty,0)$. The system \eqref{eq:sys_qChart} then has a three-dimensional locally invariant manifold $M_1$ taking the form of a graph
	\begin{align}
	r_1 = - \frac{x}{b ( 1 - 2 \rho_1)} + \epsilon \rho_1 m(x, \epsilon, \rho_1) ,\quad x\in I,\,\rho_1\in [0,\beta_{\SJ{1}}],\,\epsilon \in [0,\beta_2],\eqlab{M1Expr}
	\end{align}
	for $\beta_i>0$ $i=1,2$ sufficiently small, and 
	where $m(x, \epsilon, \rho_1)$ is smooth. The manifold $M_1$ contains the manifolds of equilibria $C_1$ and $S_1$ within the invariant hyperplanes $\epsilon=0$ and $\rho_1=0$, respectively, and the dynamics restricted to $M_1$ has $\dot x=a$ to leading order in the $x$-direction, with respect to the original time in \eqref{corbeiller}.
\end{lemma}

\begin{proof}
	Follows from center manifold theory and the statements in \lemmaref{expK1lemma}. 
\end{proof}

\subsection{Proof of \lemmaref{slowManifold-Corb}}\seclab{proofSlowManifold}
To prove \lemmaref{slowManifold-Corb} and the expression in \eqref{corbeillerManifold} we simply restrict $M_1$ in \eqref{M1Expr} to the invariant set $Q_1$ and blow down to the $(x,y)$-coordinates. Using \eqref{K3}, \eqref{mathfrakK1} and by returning the subscript $3$ on $\epsilon$ in \eqref{M1Expr},  we obtain the following equations
\begin{align}
 y = e^{-\epsilon_3^{-1}}\left(- \frac{x}{b}\right) \left(1+  e^{-\epsilon_3^{-1}} \tilde m(x, \epsilon_3, e^{-\epsilon_3^{-1}})\right),\eqlab{yM1Eqn}
 \end{align}
and
\begin{align}
 \epsilon &=e^{-\epsilon_3^{-1}} \left(- \frac{x}{b}\right)\left(1 +  e^{-\epsilon_3^{-1}} \tilde m(x, \epsilon_3, e^{-\epsilon_3^{-1}})\right) \epsilon_3,\eqlab{epsilonM1Eqn}
\end{align}
where $0<\epsilon\ll 1$ is the original small parameter. In these expressions we have also introduced a new $\tilde m$ obtained by expanding the first term in \eqref{M1Expr} about $\rho_1=0$ and setting $-x/b>0$ outside a bracket. It is straightforward to show that  $\tilde m(x,0,0)=2$. 
The expression in \eqref{corbeillerManifold} is the result of solving these equations for $y$ as a function of $\epsilon$. Let
\begin{align}
 Z(s) = W(s^{-1})^{-1},\eqlab{ZW}
\end{align}
for any $s> 0$, where $W$ is the Lambert-W function satisfying $W(t)e^{W(t)}=t$ for every $t>0$. Then 
\begin{align}
 s = Z(s) e^{-Z(s)^{-1}},\eqlab{Zs}
\end{align}
and $Z$ has a continuous extension to $s=0$ with value $Z(0)=0$.
We then solve \eqref{epsilonM1Eqn} by introducing the auxiliary variable $u$ through the expression 
\begin{align}
 \epsilon_3 = Z(\epsilon b/(-x)) \left(1+Z(\epsilon b/(-x)) u\right)^{-1}.\eqlab{uExpr}
\end{align}
Notice that 
\begin{itemize}
\item $u=0$ in this expression by \eqref{Zs} gives the `leading order solution' of \eqref{epsilonM1Eqn} obtained by setting $\tilde m\equiv 0$. 
\item Once we have obtained $\epsilon_3$ as a function of $\epsilon$, we obtain the desired $y$ as a function of $\epsilon$ from \eqref{yM1Eqn} as
\begin{align}
 y = \epsilon \epsilon_3^{-1}.\eqlab{yM1Eqn2}
\end{align}
\end{itemize}
Write
\begin{align}
 z:= Z(\epsilon b/(-x)).\eqlab{zZs}
\end{align}
Then inserting \eqref{uExpr} into \eqref{epsilonM1Eqn} produces the following equation 
\begin{align}
 1=e^{-u}(1+zu)^{-1} \left(1+e^{-z^{-1}} e^{-u} \tilde m(x,z(1+zu)^{-1},e^{-z^{-1}} e^{-u})\right),\eqlab{reducedEqn}
\end{align}
after canceling out a common factor on both sides obtained from \eqref{Zs} with $s=\epsilon b/(-x)$. This equation is smooth in $u,x$ and $z\ge 0$. In particular, $u=z=0$ is a solution for any $x\in I$ and the partial derivative of the right hand side with respect to $u$ gives $-1$ for $u=z=0$, $x\in I$. Therefore, by applying the implicit function theorem, we obtain a locally unique solution $u=\tilde h(x,z)$ of \eqref{reducedEqn} with $\tilde h$ smooth, satisfying $\tilde h(x,0)=0$ for all $x\in I$, and where $z\in [0,\beta_3]$ for $\beta_3>0$ sufficiently small. A simple computation shows that $\tilde h$ is exponentially small with respect to $z$:
\begin{align*}
 \tilde h(x,z) = e^{-z^{-1}}h(x,z),
\end{align*}
for some smooth $h$ satisfying $h(x,0)=\tilde m(x,0,0)=2$. Inserting \eqref{zZs} into this expression produces, by \eqref{uExpr} and \eqref{Zs}, the following locally unique solution of \eqref{epsilonM1Eqn}
\begin{align*}
 \epsilon_3 &= Z(\epsilon b/(-x)) \left(1-\epsilon bx^{-1} h(x,Z(\epsilon b/(-x)))\right)^{-1}.
\end{align*}
%
%
                                                   Finally, by inserting this into \eqref{yM1Eqn2} we obtain the expression
                                                   \begin{align*}
                                                    y = \epsilon Z(\epsilon b/(-x))^{-1} \left(1-\epsilon bx^{-1} h(x,Z(\epsilon b/(-x)))\right),
                                                   \end{align*}
                                                   which by \eqref{ZW} gives \eqref{corbeillerManifold}. 



\subsection{Improved singular cycle and preliminary results}\seclab{hallo}

The analysis thus far is sufficient for the construction of an improved singular orbit and corresponding Poincar\'e map. The singular orbit can now be viewed as a union of distinct orbit segments \PS{$\Gamma^i$, $i=1,2,3,4$, and the line $L$.} The segment $\Gamma^1$ is already defined as the part of the singular cycle contained in the half plane $y < 0$, and $x_d < 0$ is the $x$-coordinate of the first intersection of the trajectory flowed forward from $(0,0)$ with the line $y = 0$ in system \eqref{corbeilleryNeg}; see the discussion leading to the expression \eqref{SingCyc1} in \secref{the_le_corbeiller_system}. The remaining segments $\Gamma^i$ are are listed below, together with the coordinate chart used to define them. 
\begin{equation}
\eqlab{improved_cycle}
\begin{cases}
\Gamma^2 = \left\{(x_d, y_2, 0) : y_2 \in \mathbb R \right\} , &  \text{chart } K_2, \\
\Gamma^3 = \left\{(x_d, r_1, 0, 0) : r_1 \in [0, -b^{-1} x_d ] \right\}  , & \text{chart } \mathfrak K_1, \\
\Gamma^4 = \left\{(x_1, -b^{-1} x_1, 0, 0) : x_1 \in [x_d,0] \right\}  , & \text{chart } \mathfrak K_1, \\
 L = \left\{(0, y_2, 0) : y_2 \in \mathbb R \right\} , & \text{chart } K_2 .
\end{cases}
\end{equation}
\begin{remark}
Note that we are permitting a slight abuse of notation here by allowing $\Gamma^2$ to refer to a different segment as in \secref{the_le_corbeiller_system}, expression \eqref{SingCyc1}. In fact, it is \PS{$\Gamma^2 \cup \Gamma^3\cup \Gamma^4 \cup L$} in \eqref{improved_cycle} that upon blowing 
down to the $(x,y)$-plane becomes $\Gamma^2$ in \eqref{SingCyc1}.
\end{remark}
We can now define an improved singular cycle $\Gamma_0$ by
\[ \Gamma_0 = \Gamma^1  \cup \Gamma^2 \cup \Gamma^3 \cup \Gamma^4 \cup \PS{L},\]
see Figure \ref{fig:quick_fig}. Although this cycle  has improved hyperbolicity properties, it is still completely degenerate along $L$.
 A  \PS{fully nondegenerate} singular cycle, resolving the dynamics around $L$, will be presented in \secref{map2_blowup}, relying on details in \secref{Details} once the remaining prerequisite blow-ups have been defined.
\begin{figure}[h!]
	\centering
	\includegraphics[scale=.8]{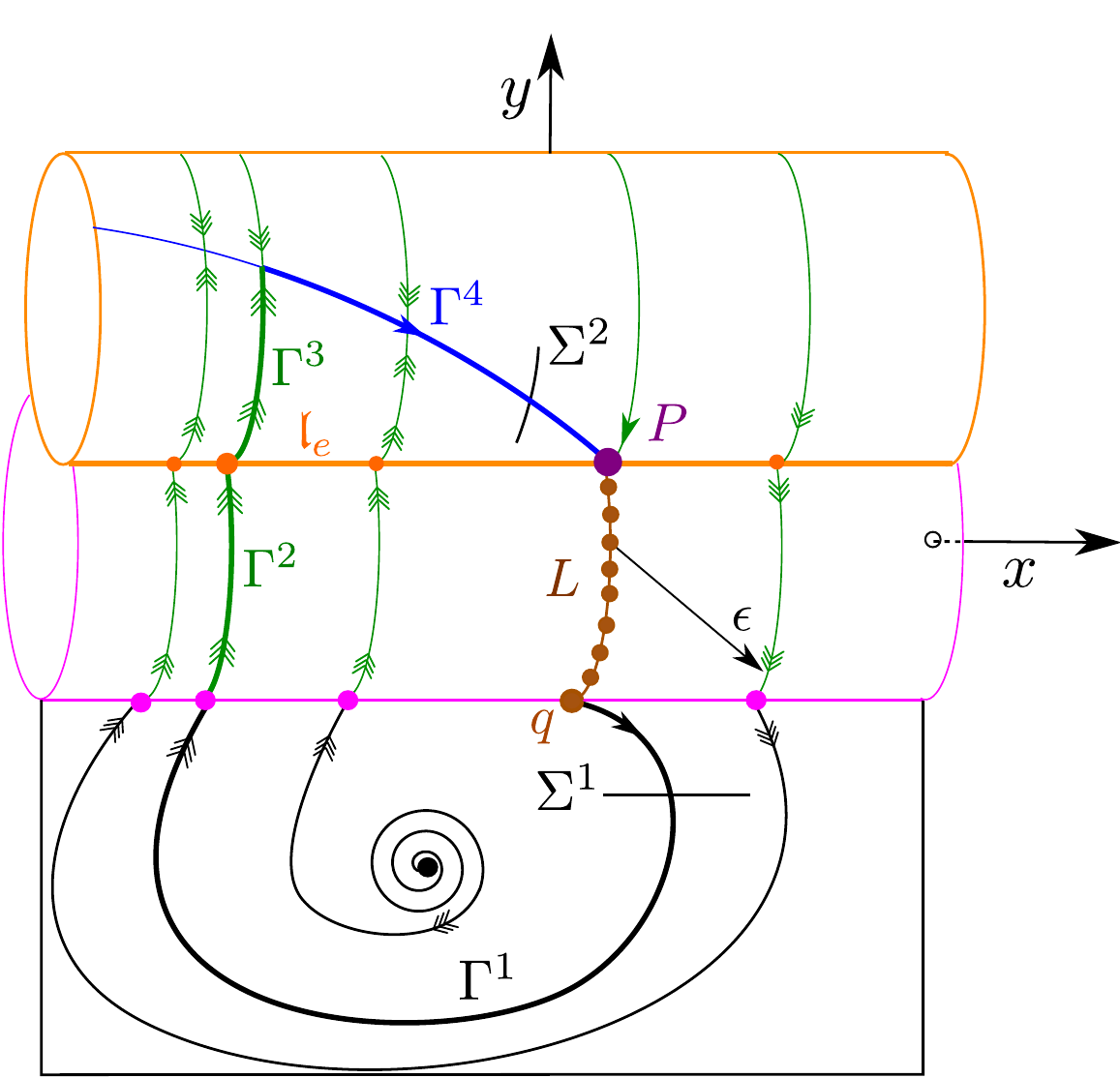}
	\caption{\PSnew{Due to the improved hyperbolicity properties obtained by the second blow-up \eqref{blowuple}, we are able to identify a less degenerate} singular cycle 
	 \PS{\mbox{$\Gamma_0 = \Gamma^1  \cup \Gamma^2 \cup \Gamma^3 \cup \Gamma^4 \cup L$}.
	 The sections $\Sigma^1$ and $\Sigma^2$ used to define the Poincar\'e mapping $\Pi = \Pi^2 \circ \Pi^1$ 	 are also shown.}
}
\label{fig:quick_fig}
\end{figure}

\subsection{The Poincar\'e map $\Pi=\Pi^2\circ \Pi^1$}\seclab{PoincareMap}
Also shown in Figure \ref{fig:quick_fig} are two sections $\Sigma^1$ and $\Sigma^2$, both transversal to $\Gamma_0$. We define these as follows:
\begin{equation}
\begin{cases}
\Sigma^1 = \left\{(x, -\delta, \epsilon) : x \in [-\alpha, \alpha], \epsilon \in [0, \epsilon_0] \right\}, & (x,y,\epsilon),  \\
\Sigma^2 = \left\{(-\eta , r_1, \epsilon, \rho_1) : r_1 \in [0,R], \epsilon \in [0, \beta], \rho_1 \in [0,\beta] \right\} , & \text{chart } \mathfrak K_1,
\end{cases}
\end{equation}
for suitably chosen, small positive constants $\delta, \alpha, \epsilon_0, \eta, R, \beta$. We then define the Poincar\'e map $\Pi : \Sigma^1 \to \Sigma^1$ induced by the flow of the extended system (\eqref{corbeiller2}, $\dot \epsilon = 0$) in terms of the composition
\[
\Pi = \Pi^2 \circ \Pi^1 : \Sigma^1 \to \Sigma^1 ,
\]
where $\Pi^1 : \Sigma^1 \to \Sigma^2$, and $\Pi^2 : \Sigma^2 \to \Sigma^1$ denote transition maps induced by the flow. To prove \thmref{main2} we consider each of $\Pi^1$ and $\Pi^2$ in turn.

First, we consider the map $\Pi^1$ and write 
\begin{align}
\eqlab{Pi1}
 \Pi^1:(x,-\delta,\epsilon) \mapsto (-\eta,r_1^{1}(x,\epsilon),W(r_1^{1}(x,\epsilon)/\epsilon)^{-1},\epsilon W(r_1^{1}(x,\epsilon)/\epsilon) ).
\end{align}
To obtain the last two components of $\Pi^1$, we have used \eqref{K3}\SJnew{,} \eqref{mathfrakK1} and the invariance of the set $Q$. 
\begin{lemma}\lemmalab{map1}
	The following holds for $\alpha>0$ and $\epsilon_0>0$ sufficiently small:
	\begin{itemize}
	\item[(i)] $\Pi^1$ is well-defined and at least $C^1$ with respect to $x$.
	\item[(ii)] The image $\Pi^1(\Sigma^1) \subset \Sigma^2$ is an exponentially thin wedge-shaped region about $M \cap \Sigma^2$,	where $M$ is the center manifold in \lemmaref{BU_slowManifold}. 
	\item[(iii)] In particular, the restricted map $x\mapsto r_{1}^1(x,\epsilon)$ for any $\epsilon \in \SJ{(}0,\epsilon_0]$ is a strong contraction, i.e.
	\begin{align}\eqlab{Pi1Expr}
	\begin{split}
	r_1^1(x,\epsilon) &= \frac{\eta}{b}+\mathcal O(\epsilon \log\epsilon^{-1}),\\
	\frac{\partial}{\partial x}r_{1}^1(x,\epsilon) &= \mathcal O\left(e^{-c/\epsilon}\right) ,
	\end{split}
	\end{align}
	for some constant $c>0$.
	\end{itemize}
\end{lemma}

Given the form of the map in \eqref{Pi1} and \lemmaref{map1}, we write $\Pi^2$ as follows
\begin{equation}
\eqlab{Pi_restricted}
\begin{split}
\Pi^2:  (-\eta,r_1,W(r_1/\epsilon)^{-1},\epsilon W(r_1/\epsilon) ) \mapsto  \left(x^2(r_1, \epsilon), -\delta, \epsilon \right) ,
\end{split}
\end{equation}
for any $r_1\in I$, $\epsilon\in [0,\delta]$. Here $I$ is a small neighborhood of $\eta/b$, recall \eqref{Pi1Expr}. The form in \eqref{Pi_restricted} restricts $\Pi^2$ to the relevant subset of $\Sigma^2$ defined by invariance of $\epsilon$ and the set $Q$, and allows for composition with $\Pi^1$, see \eqref{Pi1}. 

\begin{lemma}\lemmalab{map2}
	The following hold for $R, \beta > 0$ and $\epsilon_0$ sufficiently small:
	\begin{enumerate}
		\item[(i)] $\Pi^2 $ is well-defined and at least $C^1$ with respect to $r_1$.
		\item[(ii)] The restricted map $r_1 \mapsto x^2(r_1, \epsilon)$ for any $\epsilon \in \SJ{(} 0,\epsilon_0]$ is a strong contraction, i.e.
		\begin{equation}
		\frac{\partial}{\partial r_1}x^2\left(r_1, \epsilon\right) = \mathcal O\left(e^{-c/\epsilon}\right) ,
		\end{equation}
		for some constant $c>0$.
	\end{enumerate}
\end{lemma}
\subsection{Proof of \thmref{main2}}\seclab{thm2Proof}
Taken together, \lemmaref{map1} and \lemmaref{map2} show that $\Pi$ is a contraction. \thmref{main2} and the existence of a perturbed limit cycle  $\Gamma_\epsilon$ for all $0<\epsilon\ll 1$ is then a consequence of the contraction mapping theorem. \lemmaref{map1} is proved in \secref{map1}. A rigorous proof of \lemmaref{map2} requires additional blow-ups. We summarise these in in \secref{sec:map2}, see \secref{map2_blowup}, allowing for the proof of \lemmaref{map2} to be outlined in \secref{proofMap2}. The blow-up analysis outlined for the purposes of proving \lemmaref{map2} are given in greater detail in \secref{Details}.

\section{The map $\Pi^1$: Proof of \lemmaref{map1}}
\seclab{map1}

We define additional transversal sections $\Sigma^{1,i}$, $\SJ{i} \in \{2, 3, 4, 5\}$ as in \figref{Lemma1Sections}, and consider the map $\Pi^1$ as a composition
\[
\Pi^1 = \Pi^{1,5} \circ \Pi^{1,4} \circ \Pi^{1,3} \circ \Pi^{1,2} \circ \Pi^{1,1},
\]
where $\Pi^{1,1}: \Sigma^1 \to \Sigma^{1,2}$, $\Pi^{1,5}: \Sigma^{1,5} \to \Sigma^2$, and $\Pi^{1,i} : \Sigma^{1,i} \to \Sigma^{1,i+1}$ for $i = 2,3,4$ denote transition maps induced by the flow. To prove \lemmaref{map1}, we consider each map in turn. The first three mappings are more standard, so we will just summarise the findings.

\subsection{The transition maps $\Pi^{1,i}$ for $i=1,2,3$}\seclab{Pi123}

The mapping $\Pi^{1,1}:(x,-\delta,\epsilon)\mapsto (x^{1,1}(x,\epsilon),-\delta,\epsilon)$ is a diffeomorphism by regular perturbation theory. Notice that for any $c>0$, we have $\vert x^{1,1}(x,\epsilon)-x_d\vert \le c$ for all $x\in [-\alpha,\alpha]$, $\epsilon\in [0,\epsilon_0]$ upon taking $\delta>0$, $\alpha>0$ and $\epsilon_0>0$ sufficiently small. Working in the chart $K_1$, see \lemmaref{K1Lemma}, the second (local) map $\Pi^{1,2}$ is also standard, see e.g. \cite[Theorem 4.2]{Gucwa2009783}, being of the following $C^1$ form $x\mapsto x^{1,2}(x)+\mathcal O(\epsilon\log \epsilon^{-1})$. Here $x^{1,2}(x)$ is the base point of the local stable manifold intersecting $\Sigma^{1,2}$ at $(x,-\delta,0)$\SJnew{, and} the order $\mathcal O(\epsilon \log \epsilon^{-1})$ does not change upon differentiation with respect to $x$. Working in \SJnew{chart} $K_2$, $\Pi^{1,3}$ is a diffeomorphism of the form $x\mapsto x^{1,3}(x,\epsilon)$\SJnew{,} with $x^{1,3}$ smooth and $x^{1,3}(x,0)=x$ for all $x$. In total, the composition $\Pi^{1,3}\circ \Pi^{1,2}\circ \Pi^{1,1}$ is $C^1$ with respect to $x$ of the following form $x\mapsto \tilde x^{1,3}(x)+\mathcal O(\epsilon\log \epsilon^{-1})$ with $\tilde x^{1,3}$ smooth. 

The mappings $\Pi^{1,4}$ and $\Pi^{1,5}$ are less standard because they occur in the exponential regime. We therefore include some more details in the following.


\begin{figure}[h!]
	\centering
	\includegraphics[scale=0.8]{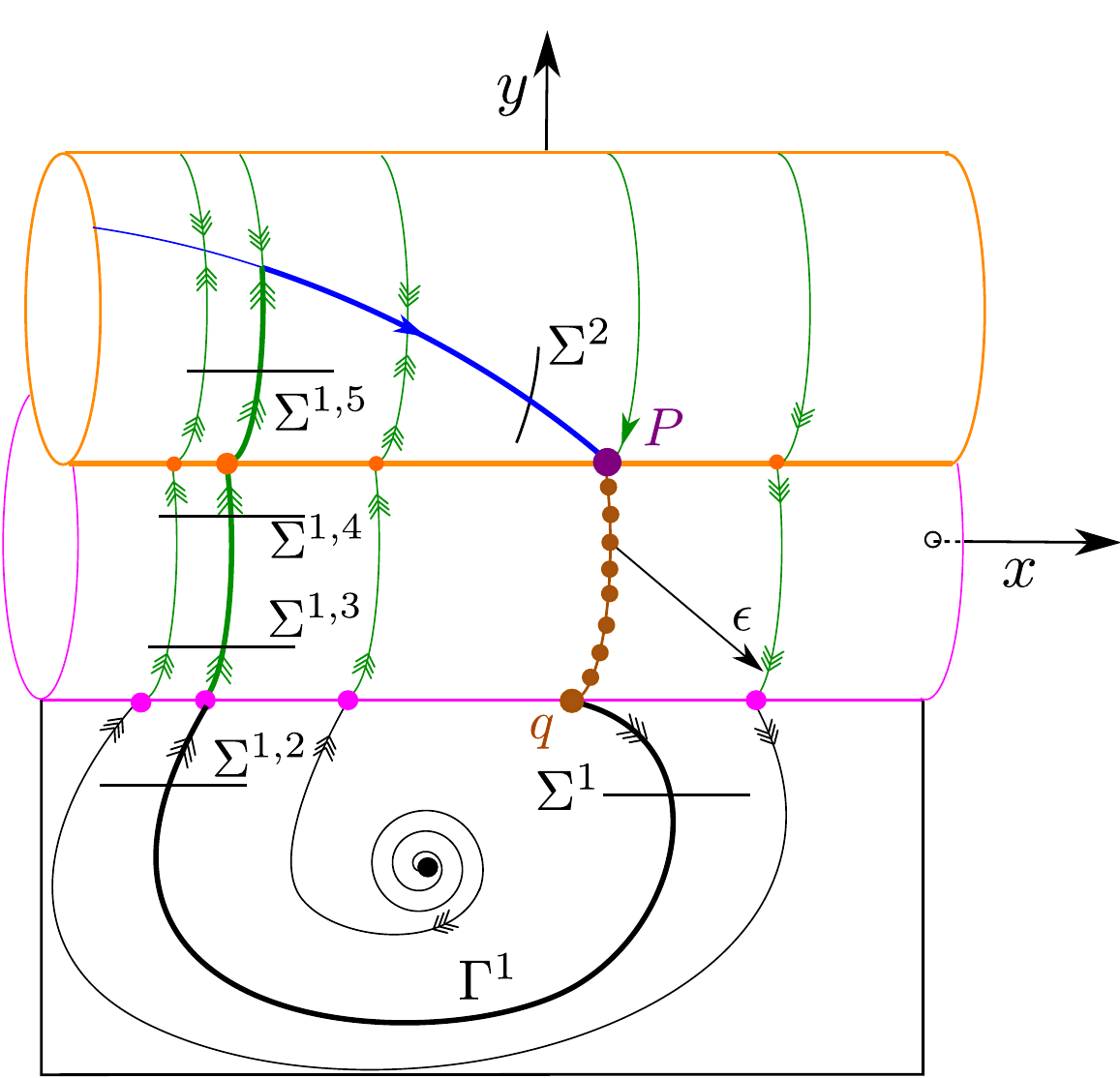}
	\caption{To prove \lemmaref{map1} we decompose \PSnew{$\Pi^1: \Sigma^1 \to \Sigma^2$} into further transition mappings $\Pi^{1,i}$, $i=1,\ldots,5$. \PSnew{The relevant new sections $\Sigma^{1,i}$, $i=2,\ldots,5$ are shown.}
		}\figlab{Lemma1Sections}
\end{figure}

\subsection{The local transition map $\Pi^{1,4} : \Sigma^{1,4} \to \Sigma^{1,5}$}\seclab{Pi4}

We work in chart $\mathfrak K_1$, for which
\[
\Sigma^{1,4} = \left\{\left(x, r_1, \delta, e^{-\delta^{-1}} \right) : |x_d - x| \leq \alpha_4, r_1 \in \left[0, e^{\delta^{-1}} R_4 \right] \right\},
\]
and define
\[
\Sigma^{1,5} = \left\{\left(x, R_5, \epsilon, \rho_1 \right) : |x_d - x| \leq \alpha_5, \epsilon \in \left[0, \beta_5 \right] , \rho_1 \in \left[0, e^{-\beta_5^{-1}} \right] \right\} ,
\]
as a small section about the intersection $\Gamma_0 \cap \{r_1=R_5\}$. Some of the expressions simplify by setting 
\begin{align}
 R_5 = e^{-\delta^{-1}},\eqlab{R5Delta}
\end{align}
so we will adopt this choice in the following.
The interval for $\rho_1$ is chosen by restriction to the invariant set $Q_1$. For ease of calculations, we drop the subscripts in \eqref{eq:sys_qChart} and translate $(x_d, 0, 0,0)$ to the origin by introducing $\tilde x = x - x_d$. After dividing the right hand side of the new system by the locally positive factor $-\left(x_d + \tilde x + b r (1 - 2 b \rho)\right)$, we obtain the system
\begin{equation}
\eqlab{map14sys}
\begin{split}
\tilde x ' &= \frac{ - \rho r \epsilon ( a + \rho r ) }{ x_d + \tilde x + b r (1 - 2 b \rho)}  , \\
r' &=  r (1 + \epsilon) ,  \\
\epsilon' &= - \epsilon^2 , \\
\rho' &= - \rho.
\end{split}
\end{equation}
We consider the system \eqref{map14sys} in order to characterize the map $\Pi^{1,4}$.

\begin{proposition}
	\proplab{map14prop}
	The map $\Pi^{1,4}: \Sigma^{1,4} \to \Sigma^{1,5}$ is well-defined for $\delta > 0$ and $\alpha_4>0$ sufficiently small, and given by
	\begin{align*}
	\Pi^{1,4} : \left(x , r, \delta, e^{-\delta^{-1}} \right) \mapsto
	\bigg(x + \mathcal O (r \log r^{-1} ), e^{-\delta^{-1}}&, W \left(1/r \delta \right)^{-1},r \delta W \left(1/r \delta \right) \bigg).
	\end{align*}
	Here the order of $\mathcal O (r\log r^{-1})$ does not change by differentiation with respect to $x$.
\end{proposition}

\begin{proof}
%
	Consider a solution of \eqref{map14sys} satisfying
	\[
	\begin{cases}
	\tilde x(0) &= \tilde x_{in}, \\
	r(0) &= r_{in}, \\
	\epsilon(0) &= \delta, \\
	\rho(0) &= e^{-\delta^{-1}} ,
	\end{cases}
	\qquad
	\begin{cases}
	\tilde x(T) &= \tilde x_{out}, \\
	r(T) &= R_5, \\
	\epsilon(T) &= \epsilon_{out}, \\
	\rho(T) &= \rho_{out} .
	\end{cases}
	\]
	The system
	\begin{equation}
	\begin{split}
	r' &=  r (1 + \epsilon) ,  \\
	\epsilon' &= - \epsilon^2 , \\
	\rho' &= - \rho ,
	\end{split}
	\end{equation}
	decouples from \eqref{map14sys}, and can be integrated directly to give
	\[
	r(t) = r_{in} e^t \left(1 + \delta t \right) , \qquad \epsilon(t) = \frac{\delta}{1 + \delta t} , \qquad \rho(t) = e^{-\delta^{-1}} e^{-t} .
	\]
		One can use these equations to obtain the following expression for the transition time $T$:
	\[
	T = - \delta^{-1} + W \left(\frac{R_5}{r_{in} \delta e^{-\delta^{-1}}} \right) .
	\]
	By the above we obtain
	\[
	\epsilon(T) = W \left(\frac{R_5}{r_{in} \delta e^{-\delta^{-1}}} \right)^{-1}, \qquad \rho(T) = r_{in} \delta e^{-\delta^{-1}} R_5^{-1} W \left(\frac{R_5}{r_{in} \delta e^{-\delta^{-1}}} \right).
	\]
	(These expressions also follow from the conservation of the original $\epsilon$ and the set $Q$, see \eqref{K3} and \eqref{mathfrakK1}: $r_{in}\delta e^{-\delta^{-1}} = R_5 \epsilon(T) e^{-\epsilon(T)^{-1}}$.)
	Furthermore,
	\[
	\tilde x' = \mathcal O(\rho(t) r(t) \epsilon(t))) = \mathcal O( r_{in}),
	\]
	which is $C^1-$smooth, guaranteeing that the leading order is well behaved with respect to integration. Hence
	\[
	\tilde x_{out} - \tilde x_{in} = \mathcal O( r_{in} T).
	\]
	The result then follows from \eqref{R5Delta} and using the asymptotics \eqref{wasymp} of $W$.
\end{proof}

\subsection{The transition map $\Pi^{1,5} : \Sigma^{1,5} \to \Sigma^2$}\seclab{Pi5}

We continue to work in chart $\mathfrak{K}_1$.

\begin{proposition}
	\proplab{map15prop}
	For $\eta>0$ sufficiently small, the map $\Pi^{1,5}$ is well-defined and exponentially contracting in the $r_1-$coordinate. In particular, if $r_1^{1,5}(x,\epsilon,\rho_1)$ is the $r_1$-coordinate of $\Pi^{1,5}(x,R_5,\epsilon,\rho_1)$ then $r_1^{1,5}$ is $C^1$ with respect to $x$, satisfying the following estimates
	\begin{align*}
	 r_1^{1,5}(x,\epsilon,\rho_1)=\frac{\eta}{b} + \mathcal O(\rho_1),\\
	 \frac{\partial}{\partial x} r_1^{1,5}(x,\epsilon,\rho_1) = \mathcal  O(e^{-c/(\epsilon\rho_1)}),
	\end{align*}
for some $c>0$. 
%
\end{proposition}

\begin{proof}
	Trajectories with initial conditions in $\Sigma^{1,5}$ are strongly attracted to their base-points on the invariant center manifold $M_1$ in \lemmaref{BU_slowManifold}, and track the slow flow after reaching a local tubular neighborhood of $M_1$. 
To leading order, the flow on $M_1$ is determined by
	\begin{equation}
	\begin{split}
	x' &= - \frac{a}{b} {x \epsilon \rho_1}(1 + \mathcal O(\epsilon+\rho_1))),
	\end{split}
	\end{equation}
	so that $x' > 0$ and $\Pi^{1,5}$ is therefore well-defined. Center manifold theory implies exponential contraction ($e^{-c t}$) along one-dimensional fibers with base points on $M_1$. Since 
	\begin{itemize}
	\item $r_1(t)\ge \nu>0$ for some $\nu>0$ sufficiently small during the transition from $\Sigma^{1,4}$ to $\Sigma^{1,5}$;
	\item $r_1\rho_1\epsilon=\text{const}.$, see \eqref{K3} and \eqref{mathfrakK1};
	\end{itemize}
	we can estimate the travel time, where $x$ changes by $\mathcal O(1)$, to be of order $\mathcal O(1/(\epsilon \rho_1))$. The result therefore follows.
\end{proof}

\subsection{Proof of \lemmaref{map1}}
The analysis of the mappings $\Pi^{1,i}$, $\SJ{i \in \{1, 2, 3, 4, 5\}}$ in \secref{Pi123}, \secref{Pi4} and \secref{Pi5} prove \lemmaref{map1} upon composition. In particular, the exponential contraction in the $r_1$-coordinate in \lemmaref{map1}, is a corollary of \propref{map15prop} upon using that $\epsilon = r_1 \rho_1 \epsilon_3$, see \eqref{K3} and \eqref{mathfrakK1}, with $r_1\approx \eta/b$.

%

\section{The map $\Pi^2$: Proof of \lemmaref{map2}}
\seclab{sec:map2}
\PSnew{The analysis of the map $\Pi^2$ requires good control of the flow close to the line $L$ of degenerate equilibria $L$, see 
Figure \ref{fig:quick_fig}.}
Additional blow-ups are necessary in order to prove \lemmaref{map2}. In this section we summarise the blow-up transformations and dynamical features \PS{leading to the} construction of \PS{the final nondegenerate} singular cycle $\Gamma_0$ as shown in \figref{fig:lc_singular_cycle}. The cycle $\Gamma_0$ has a total of eight distinct segments $\Gamma^i$, $i \in \{1, \ldots , 8\}$. We have already identified the segments $\Gamma^i$ for $i \in \{1,2,3,4 \}$ in \eqref{improved_cycle}; \PS{the segments $\Gamma^i$ for $i \in \{ 5,6,7,8 \}$ will be defined by three additional blow-ups needed to resolve the degeneracies of the line $L$.}
This section is included here for \PS{expository} purposes: a more \PS{detailed technical} presentation is given in \secref{Details}.

\PSnew{
By looking at  \figref{quick_fig} and by extrapolating from what we did so far, it is 
natural to expect that a  straightforward (cylindrical) blow-up of the non-hyperbolic line (circle) $L$ will be necessary for the 
construction of the final singular cycle $\Gamma_0$.} 
\PSnew{However, it turns out that  additional  difficulties arise  close to the point $P$  which is the endpoint of $L$ on the (orange)
cylinder, corresponding to  $\rho=0$.
 This is somewhat expected, given that this transition corresponds dynamically to a transition between regimes in which exponential terms are dominant (the upper horizontal cylinder) to a regime in which algebraic terms are dominant (the vertical cylinder obtained after blow-up of $L$). The analysis of this transition requires  two  additional spherical blow-ups shown in \figref{fig:lc_singular_cycle}, which are necessary \KUK{for resolving the degeneracies} and `connecting' the two regimes.}

\begin{figure}[h!]
	\centering
	\includegraphics[scale=0.9]{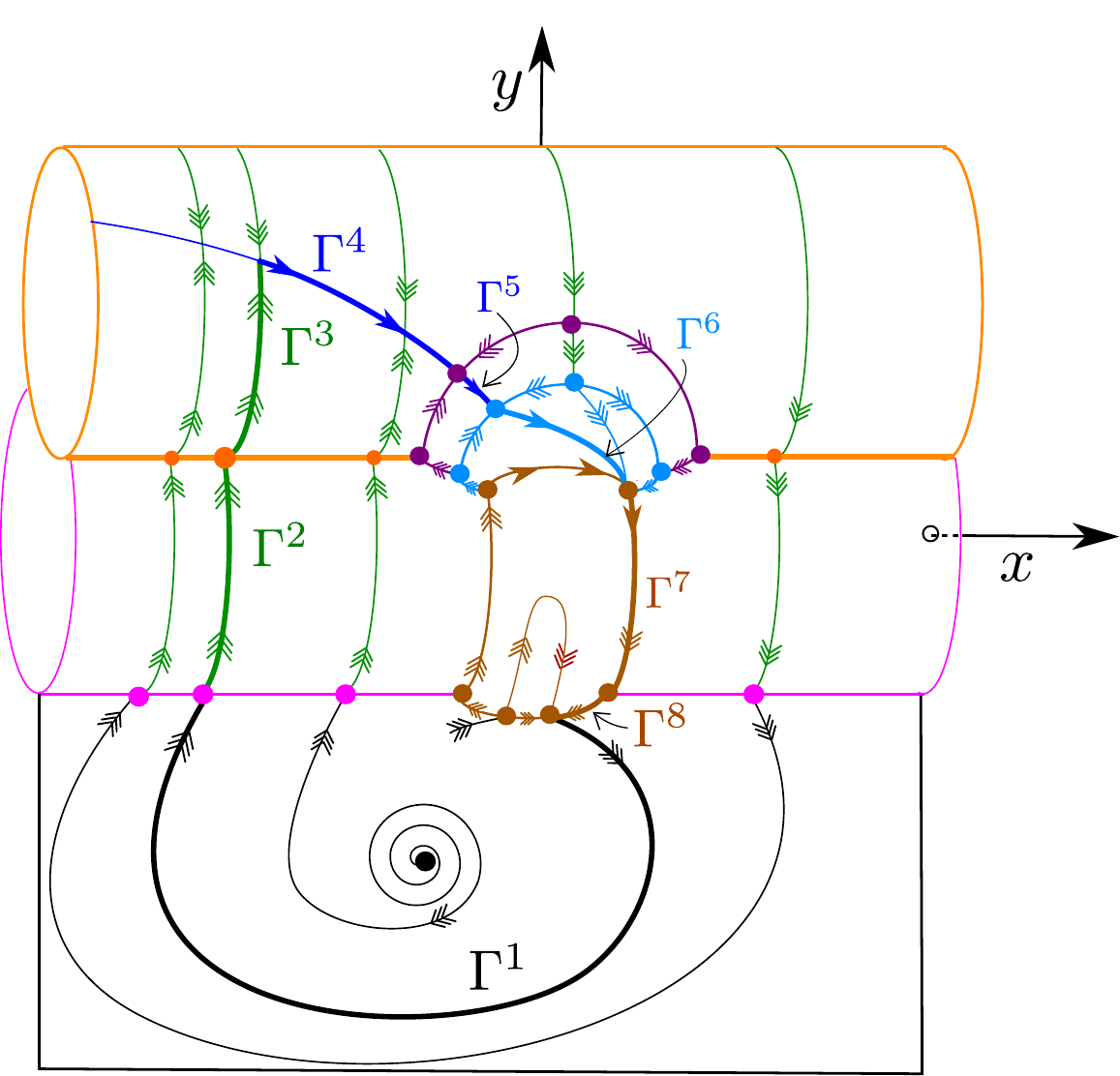}
	\caption{The \PS{fully nondegenerate} singular cycle obtained from the repeated blow-up transformations, visualized in \figref{bu2}, summarized in \secref{map2_blowup} and detailed in \secref{Details}. }
	\figlab{fig:lc_singular_cycle}
\end{figure}

\subsection{Successive blow-ups and the \PS{fully nondegenerate} singular cycle $\Gamma_0$}
\seclab{map2_blowup}

The required sequence of blow-up transformations is sketched in \figref{bu2}, and outlined in the following.

\begin{figure}
	\centering
		\subfigure[]{\includegraphics[width=.45\textwidth]{./figures/corbeillerBlowup2}}
		\subfigure[]{\includegraphics[width=.45\textwidth]{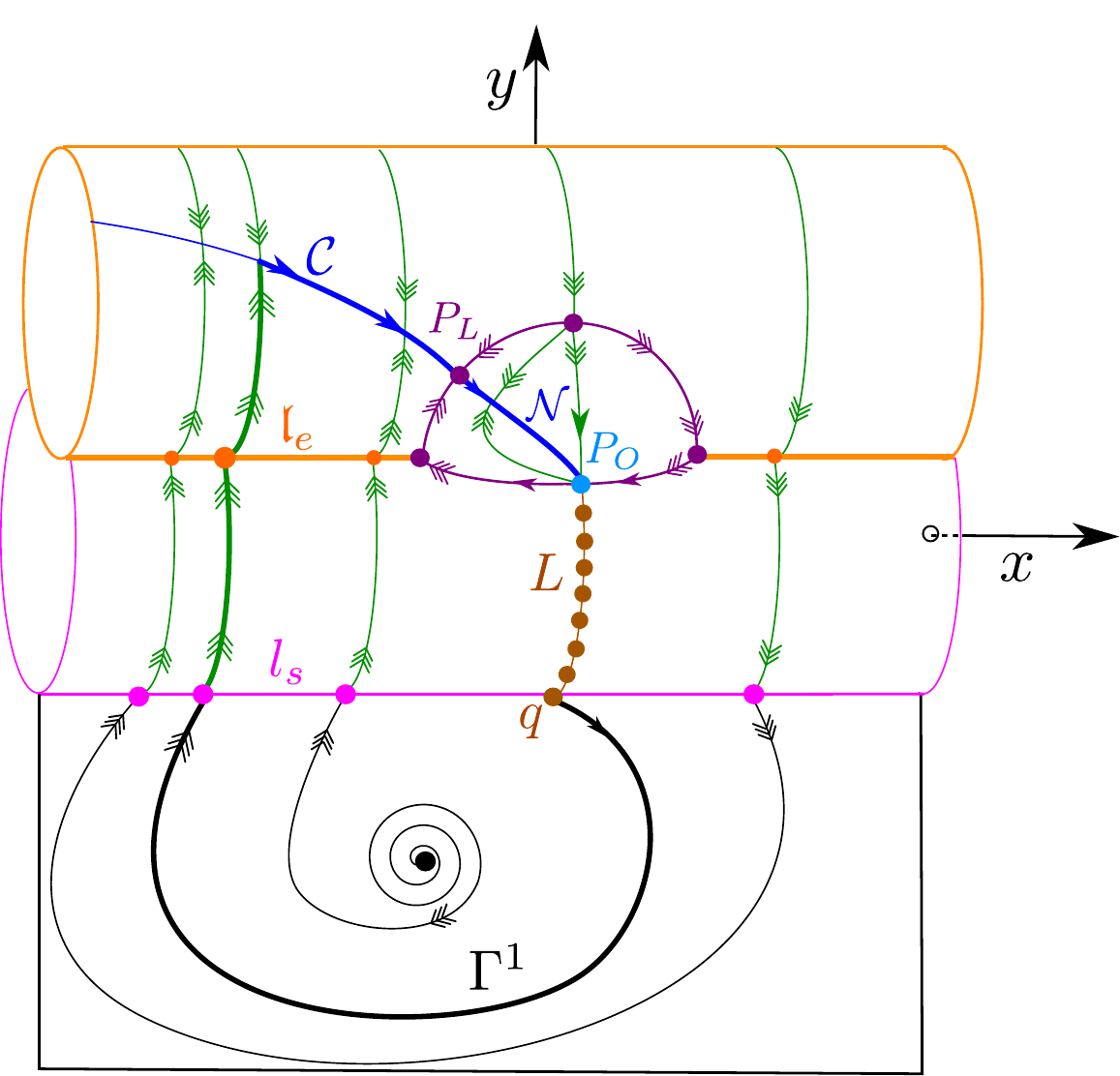}}

		\subfigure[]{\includegraphics[width=.45\textwidth]{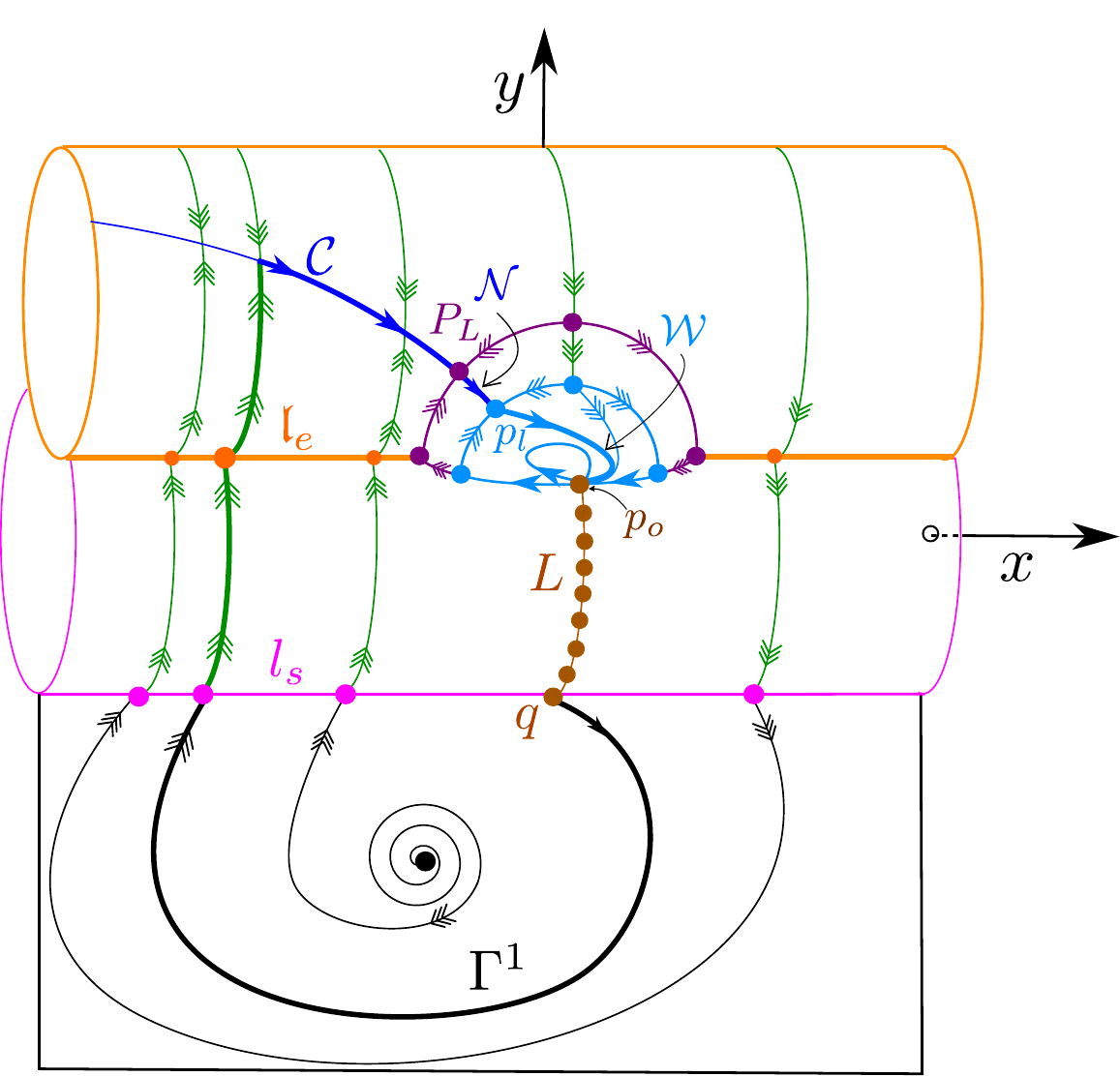}}
		\subfigure[]{\includegraphics[width=.45\textwidth]{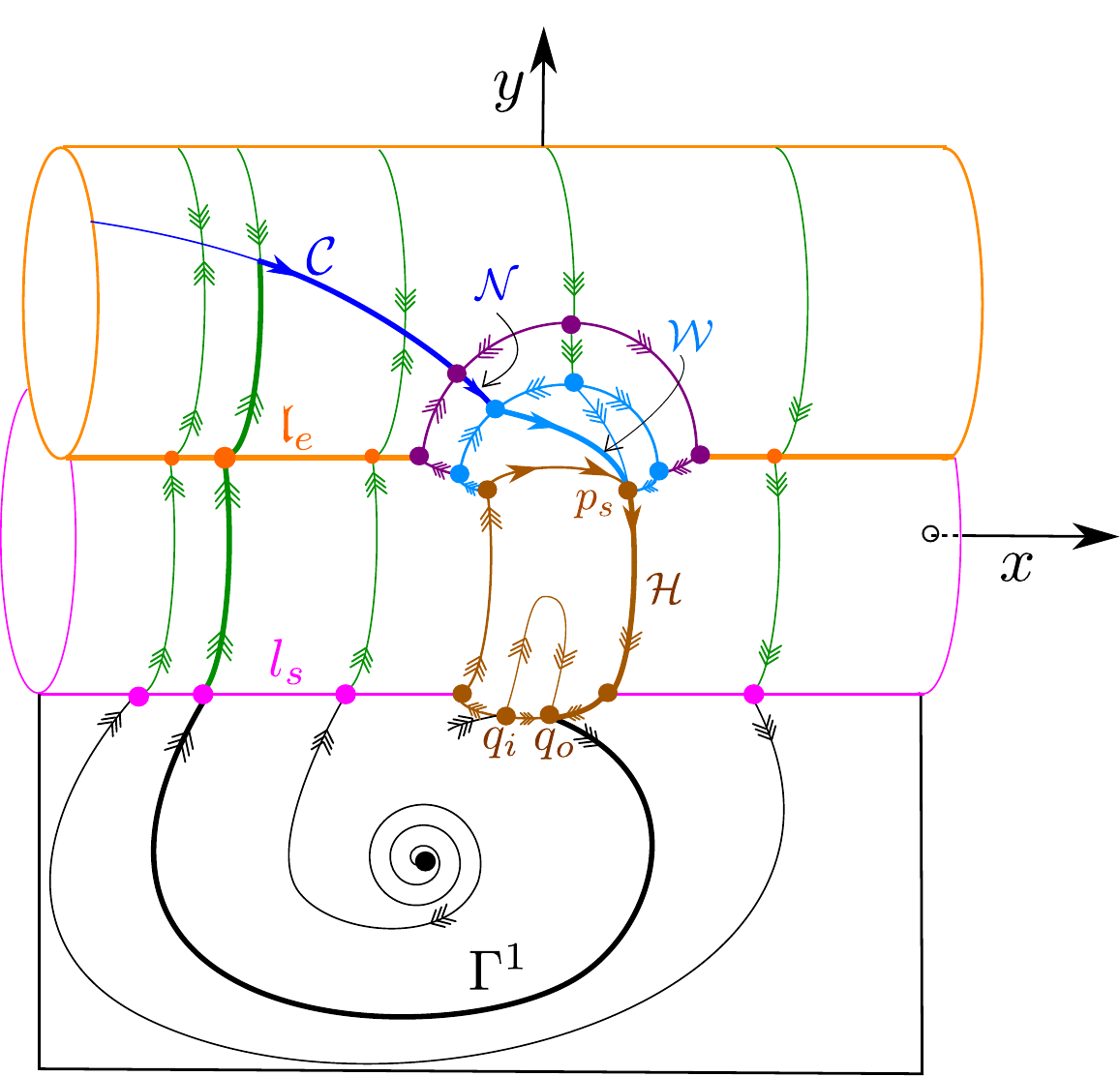}}
			\caption{\PSnew{Sequence of blow-up transformations needed to desingularize the dynamics near the line of degenerate
 equilibria  $L$ and to connect the exponential regime and the critical  manifold $\mathcal C$ with the orbit $\Gamma^1$ (black) in the plane $\bar \epsilon =0$. 
In (a) we show again \figref{corbeillerBlowup2}.  The degenerate equilibrium $P$ (purple) is blown up to a sphere (purple) in (b), 
as  described in \secref{further_blowup_and_the_singular_cycle}. On this sphere another degenerate equilibrium $P_O$ (cyan)
is identified, which is blown-up to another sphere (cyan)  shown in (c), as  detailed in \secref{spherical_blowup}.
Finally, the still remaining line 
of degenerate equilibria $L$ (brown) is blown-up to a cylinder 
(brown) shown in (d)  and analyzed in \secref{vertical_cylinder_top} and \secref{L_blowup}.}}
\figlab{bu2}
\end{figure}

By \lemmaref{expK1lemma}, \secref{ExpBlowup}, the manifold $\mathcal C$ terminates in a non-hyperbolic point $P$ (purple) at the endpoint of the non-hyperbolic curve $L_e$ (the origin in chart $\mathfrak K_1$ coordinates) corresponding to the endpoint of the vertical non-hyperbolic line $L$ (brown); see \figref{bu2}a. 
In order to resolve this degeneracy, this point \PSnew{is} blown-up to a sphere as in \figref{bu2}b; this is done in \secref{further_blowup_and_the_singular_cycle} as part of a larger `cylinder of spheres' blow-up. Note that this sphere is
the `outer sphere' (also purple) in \figref{fig:lc_singular_cycle}. The manifold $\mathcal C$ \KUK{then} terminates in a partially hyperbolic and attracting point $P_L$ on the equator of \KUK{this sphere}. One identifies the extension of the singular cycle as a normally hyperbolic attracting curve of equlibria $\mathcal N$ (blue) on this sphere, terminating in another non-hyperbolic point $P_O$ (cyan), as shown in \figref{bu2}b.

The non-hyperbolic point $P_O$ is analyzed in \secref{spherical_blowup} by means of a second spherical blow-up. This is the `inner sphere' (also cyan) sketched in \figref{fig:lc_singular_cycle} and \figref{bu2}c. We identify an attracting, partially hyperbolic point $p_l$ (cyan) on the equator of the sphere as the endpoint of $\mathcal N$. The relevant thing for the construction of $\Gamma_0$ is the existence of a unique center manifold $\mathcal W$ (cyan) emanating from $p_l$, which we can identify with $\Gamma^6$; compare \figref{fig:lc_singular_cycle} and \figref{bu2}c. Phase plane arguments (the system is no longer slow-fast) are given in \secref{spherical_blowup} (see \lemmaref{sphere2_phase_plane} in particular) in order to show that $\mathcal W$ terminates in yet another non-hyperbolic point $p_O$ (brown), from which the non-hyperbolic line $L$ (also brown) emanates.

All remaining non-hyperbolicity is resolved by means of cylindrical blow-up of $L$ in \secref{vertical_cylinder_top} and \secref{L_blowup}, leading to the fully resolved scenario sketched in \figref{fig:lc_singular_cycle}, see also \figref{bu2}d. In \secref{vertical_cylinder_top}, a cylindrical blow-up applied locally near the inner blow-up sphere leads to the identification of a partially hyperbolic saddle point $p_s$ (brown). The orbit $\mathcal W$ terminates at $p_s$, and the next component of $\Gamma_0$ can be identified with the unstable manifold $W^u(p_s)$, which is shown in \lemmaref{cyl_top} to lie along the outside `edge' of the cylinder, denoted by $\mathcal H$ (also brown) in \figref{bu2}d. The remaining analysis is carried out in \secref{L_blowup}, after it is shown in \lemmaref{vert_cyl_trans_map_top} that the dynamics in the local cylindrical blow-up described above can be mapped to those obtained in a cylindrical blow-up of the line $L$
\SJ{in} the algebraic regime, i.e. 
in charts $K_i$, $i=1,2,3,$ subsequent to the (first) cylindrical blow-up leading to the lower horizontal cylinder in \figref{fig:lc_singular_cycle}. The analysis leads to the identification of the final two cycle segments $\Gamma^7$ and $\Gamma^8$ in \figref{fig:lc_singular_cycle}. The first can be identified with $\mathcal H$, which is shown in \lemmaref{K11lem} to be invariant and terminating in a partially hyperbolic singularity at $q_s$. The second can be identified with the unstable manifold $W^u(q_s)$, which is invariant, and shown (also in \lemmaref{K11lem}) to terminate in a hyperbolic (resonant) saddle singularity $q_o$. Although the analysis here is to some extent standard, it is frequently complicated by the fact that transition across different coordinate charts becomes non-trivial after the application of multiple blow-up transformations.


%


To summarise, we define the \PS{fully nondegenerate} singular cycle $\Gamma_0$ by
\[
\Gamma_0 = \Gamma^1  \cup \Gamma^2 \cup \Gamma^3 \cup \Gamma^4 \cup \Gamma^5 \cup \Gamma^6 \cup \Gamma^7 \cup \Gamma^8,
\]
where $\Gamma^i$ for $i \in \{1,2,3,4\}$ have already been defined in \eqref{improved_cycle} and where
\begin{equation}
\eqlab{Gamma0_segments}
\begin{cases}
\Gamma^5 = \mathcal N , \\
\Gamma^6 = \mathcal W, \\
\Gamma^7 = \mathcal H, \\
\Gamma^8 = \left\{(r_{11}, 0, 0) : r_{11} \in \left[0, \frac{1}{2a} \right] \right\} , & \text{chart } K_{11} .
\end{cases}
\end{equation}
Regarding $\Gamma^8$, the coordinates specified by the chart $K_{11}$ are defined in \secref{K11}, see also \eqref{K11eqns}. 
\KUK{Note also that the union $\Gamma^5\cup \Gamma^6 \cup \Gamma^7\cup \Gamma^8$ of the segments in \eqref{Gamma0_segments} becomes $L$ in \eqref{improved_cycle} upon blowing down.} 
Each segment has improved hyperbolicity properties that allow us to describe the mapping $\Pi^2$ and prove \lemmaref{map2}.

\subsection{Proof of \lemmaref{map2}: A summary}\seclab{proofMap2}

\mbox{\PS{Similar}} to our analysis of the map $\Pi^1$, the map $\Pi^2$ can be analyzed by defining additional transversal sections $\Sigma^{2,i}$, $ i\in \{2, \ldots , 9\}$ as in \figref{Lemma2Sections}, and considering the composition
\[
\Pi^2 = \Pi^{2,9} \circ \Pi^{2,8} \circ \Pi^{2,7} \circ \Pi^{2,6} \circ \Pi^{2,5} \circ \Pi^{2,4} \circ \Pi^{2,3} \circ \Pi^{2,2} \circ \Pi^{2,1} .
\]
Here $\Pi^{2,1}: \Sigma^2 \to \Sigma^{2,2}$, $\Pi^{2,9}: \Sigma^{2,9} \to \Sigma^1$, and $\Pi^{2,i} : \Sigma^{2,i} \to \Sigma^{2,i+1}$ for $i \in \{2, \ldots 8 \}$ denote transition maps induced by the flow.

\begin{figure}[h!]
	\centering
	\includegraphics[scale=1.0]{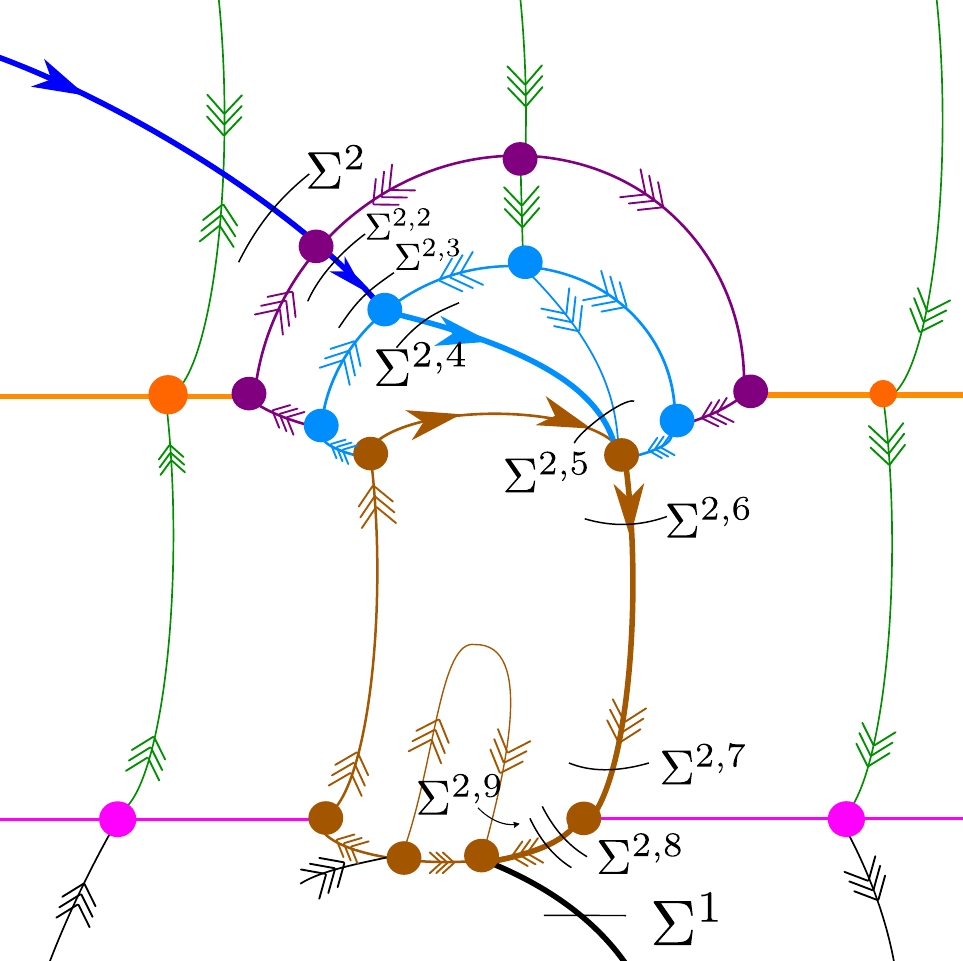}
	\caption{Illustration of the sections $\Sigma^{2,i}$, \SJ{$i\in\{2,\ldots, 9\}$} relevant for the description of the map $\Pi^2$ as the composition of several transition maps $\Pi^{2,i}$, \SJ{$i\in\{1,\ldots,9\}$}. }\figlab{Lemma2Sections}
\end{figure}

%

 To prove \lemmaref{map2}, we have to consider the maps $\Pi^{2,i}$ in turn. However, each mapping is standard so we just focus on  \PS{explaining} in what coordinates the mappings are described, \KUK{making references to equations, sections and lemmas to come below}, and summarise the findings as follows:

  \textit{The transition map $\Pi^{2,1}$}  is described in the coordinates specified by the chart $\mathcal K_1$, see \eqref{cyl_spheres_K2}. By center manifold theory and \lemmaref{sphere1K1} \KUK{below} the mapping is well-defined and exponentially contracting in the direction transverse to \SJ{$\mathcal W$}. 

  \textit{The transition map $\Pi^{2,2}$} is described in the coordinates specified by the chart $\mathcal K_2$, see \eqref{cyl_spheres_K2}. The properties of this map are similar to $\Pi^{2,1}$ since the line of partially hyperbolic equilibria $\mathcal N$ extends all the way down to $P_O$.  However, $P_O$ is non-hyperbolic and it is important to highlight that in $\mathcal K_2$ we are leaving the `exponential regime'. We can therefore reduce the phase space dimension again, see \eqref{Q12P}. This produces a $3$-dimensional system.

  \textit{ The transition map $\Pi^{2,3}$} is described in the coordinates specified by the chart $\widetilde{\mathcal K}_1$, see \eqref{shpere_K2}. Center manifold theory and \lemmaref{sphere2_K1} \KUK{below} show that the mapping is well-defined and exponentially contracting in the direction transverse to $\mathcal W$.

  \textit{ The transition map $\Pi^{2,4}$} is a diffeomorphism by regular perturbation \SJ{theory} applied in the coordinates specified by the chart $\widetilde{\mathcal K}_2$, see \eqref{shpere_K2}.
  
   \textit{ The transition map $\Pi^{2,5}$} is described in the coordinates specified by the chart $\widehat K_{31}$, see \eqref{Torus_coords} and the equations of motion in \eqref{vert_cyl_top_K13}, associated with the blow-up of $L$. In these coordinates we gain hyperbolicity, as indicated in \figref{bu2}d, see also the local picture in \figref{lc_upper_L_blowup}. Within the center manifold, the point $p_s$ (see \figref{bu2}d) is a hyperbolic saddle upon further desingularization, having $\mathcal H$ as an unstable manifold. The mapping $\Pi^{2,5}$ is therefore well-defined and (algebraically) contracting.  
  
   \textit{The transition map $\Pi^{2,6}$}. By \lemmaref{vert_cyl_trans_map_top} \KUK{below}, we can transform the result on $\Pi^{2,5}$ into the chart $K_{21}$, see \eqref{K21}, associated with blowing up $L$ in the original scaling chart $K_2$, recall \eqref{K2Here}. Here we can also describe the mapping $\Pi^{2,6}$ as a diffeomorphism by regular perturbation theory using the invariance of the line $\mathcal H$. See also \lemmaref{K21lem}.

   \textit{The transition map $\Pi^{2,7}$} is described  in the coordinates specified by the chart $K_{11}$, see \eqref{K11K12}, as a local passage near the semi-hyperbolic saddle $q_s$. The mapping is well-defined and non-expanding; the details being similar to those in \cite[Theorem 4.2]{Gucwa2009783}.
   
    \textit{The transition map $\Pi^{2,8}$} is described in the coordinates specified by the chart $K_{12}$, see \eqref{K11K12}, as passage near the resonant hyperbolic saddle $q_o$ with a one dimensional unstable manifold $\Gamma^1$. The details are similar to those in \cite[lemma 3.6]{2019arXiv190806781U}.

 Each map is one-dimensional and at least $C^1$ upon restricti\SJnew{on} to \SJ{$0 < \epsilon\ll 1$} and the set $Q$. This completes the (sketched) proof of \lemmaref{map2} upon composition. Further details can be found below and in the forthcoming PhD-thesis \cite{jelbert}.

\section{Blow-up analysis for the map $\Pi^2$}
\seclab{Details}

The blow-up transformations and main features in the analysis for the proof of \lemmaref{map2} are presented in this section. Our analysis divides into two parts: (i) an understanding of the transition between the exponential and algebraic regimes, and (ii) a blow-up analysis describing the dynamics near the non-hyperbolic line $L$ in the algebraic regime. Part (i) is considered in \secref{exp_alg_transition}, and focuses (among other things) on the spherical blow-ups shown in \figref{fig:lc_singular_cycle}. Part (ii) is considered in \secref{L_blowup}.

\subsection{Exiting the exponential regime}
\seclab{exp_alg_transition}

Here we consider the transition out of the `exponential regime'. In terms of \figref{fig:lc_singular_cycle}, our aim is to understand the manner in which trajectories move from \KUK{the} upper horizontal cylinder (exponential regime), to the vertical cylinder (algebraic regime). 

\subsubsection{Blow-up near the non-hyperbolic line $L_{e,1}$ in chart $\mathfrak K_1$}\seclab{further_blowup_and_the_singular_cycle}

We start in the exponential regime, chart $\mathfrak K_1$, dropping the subscripts in system \eqref{eq:sys_qChart}:
\begin{equation}
\eqlab{4d_sys}
\begin{split}
x' &= \rho r \epsilon (a + \rho r) , \\
r' &= - r (1 + \epsilon) \left( x + b r \left( 1 - 2 \rho \right) \right) , \\
\epsilon' &= \epsilon^2 \left( x + b r \left( 1 - 2 \rho \right) \right) , \\
\rho' &= \rho \left( x + b r \left( 1 - 2 \rho \right) \right) .
\end{split}
\end{equation}
By \lemmaref{expK1lemma}, the one-dimensional manifold $\mathcal C$ identified with the cycle segment $\Gamma^4$ terminates at $(0,0,0,0) \in \mathcal P$, where $\mathcal P : x = r = 0$ constitutes an entire plane of non-hyperbolic fixed points for \eqref{4d_sys}. One could proceed by blowing up the entire plane $\mathcal P$, however only the dynamics near the curve $L_e \subset \mathcal P$ will be relevant for the transition (recall that $L_e$ and $L_3$ coincide where domains overlap; see again \lemmaref{expK1lemma}). This motivates a blow-up of $(x,r,\rho)=(0,0,0)$ in the following form
\begin{align}
\eqlab{CylSpheresBlowup}
\nu \ge 0,\,(\bar x,\bar r, \bar \rho)\in S^2 \mapsto \begin{cases}
x&=\nu \bar x,\\
r &=\nu \bar r, \\
\rho &= \nu \bar \rho.
\end{cases}
\end{align}
We introduce the coordinate charts
\[
\mathcal K_1: \bar x = - 1 , \qquad \mathcal K_2: \bar \rho = 1 ,
\]
for which we have chart-specific coordinates given by\
\begin{equation}
\eqlab{cyl_spheres_K2}
\begin{aligned}
\mathcal K_1 : x &= - \nu_1, && r = \nu_1 r_1, && \rho = \nu_1 \rho_1, \\
\mathcal K_2 : x &= \nu_2 x_2, && r = \nu_2 r_2, && \rho = \nu_2.
\end{aligned}
\end{equation}

\begin{remark}
Notice that $\epsilon$ is not transformed by \eqref{CylSpheresBlowup}. Geometrically, \eqref{CylSpheresBlowup} therefore blows up
\begin{equation}
\eqlab{L_nearby}
\mathcal L = \left\{(0, 0, \epsilon, 0) : \epsilon \geq 0 \right\} \subset \mathcal P ,
\end{equation}
to a `cylinder of spheres' $CS = \{\nu = 0\} \times S^2 \times \mathbb R$. Note that each $CS \cap \{\epsilon = const. \}$ is an invariant sphere in $CS$. Notice also that 
\begin{itemize}
	\item [(I)]$L_e$ and $\mathcal L$ are tangent at $(0,0,0,0)$, and 
	\begin{align}\eqlab{LeCapMathL}L_e \cap \mathcal L = \{(0,0,0,0)\}.\end{align}
	\item[(II)] Considered in terms of its parameterization \eqref{L1}, $L_e$ is flat at $(0,0,0,0)$, and thus flat with respect to the line $\mathcal L$ as $\epsilon \to 0$; see \figref{Flat}.
\end{itemize}
Geometrically, it seems more natural to blow-up $L_e$. To do this one would rectify $L_e$ and apply the cylindrical blow-up transformation \eqref{CylSpheresBlowup} along (the transformed) $L_e$. However, our approach avoids this unnecessary coordinate transformation. Besides (a) only the sphere $CS\cap \{\epsilon=0\}$, which is the same for both approaches (recall \eqref{LeCapMathL}), will be relevant, and (b) once we leave the exponential regime and enter $\bar \rho>0$ of the sphere \eqref{CylSpheresBlowup}, we will apply a subsequent blow-up transformation that effectively blows up $L_e$ (or $L$).


\begin{figure}[h!]
	\centering
	\includegraphics[scale=0.6]{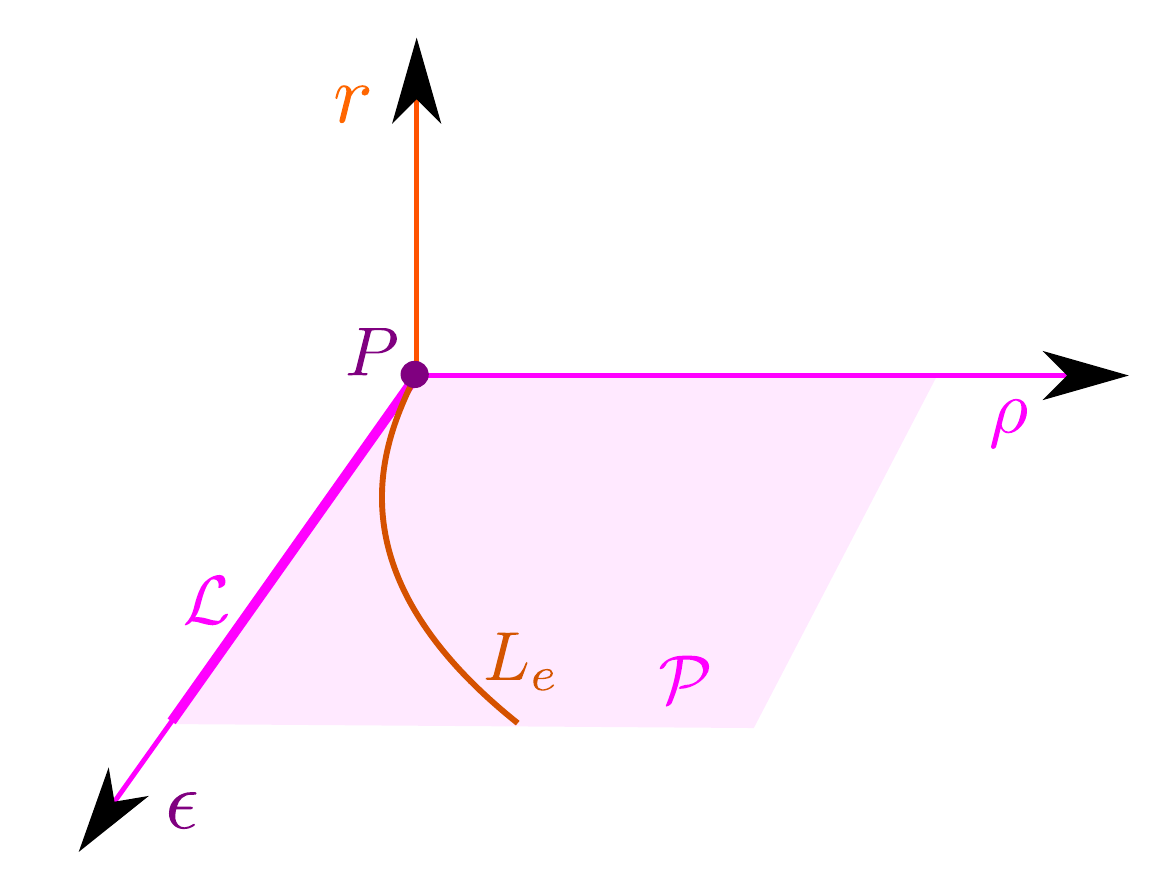}
	\caption{Illustration within $\{x=0\}$ of the lines $\mathcal L$ (magenta) and $L_e$ (brown) as subsets of the plane $\mathcal P:\,x=r=0$ (shaded magenta).}\figlab{Flat}
\end{figure}

\end{remark}

We will adopt the following notational convention: given an object $\gamma$ identified in chart $\mathcal K_1$, we denote its image under the blow-up transformation \eqref{CylSpheresBlowup} by $\gamma'$, and it's image in a particular coordinate chart $\mathcal K_i$ by $\gamma'_i$ (this will help to avoid confusion given the earlier dropping of subscripts etc).


The transition maps between coordinates in charts $\mathcal K_i$, $i = 1,2$ are given by
\begin{equation}
\eqlab{CylSpheresTransMaps}
\begin{aligned}
\kappa_{12}': r_1 &= - x_2^{-1} r_2, && \rho_1 = - x_2^{-1}, && \nu_1 = - \nu_2 x_2 , && x_2 < 0, \\
\kappa_{21}': x_2 &= - \rho_1^{-1}, &&  \SJ{r_2} = \rho_1^{-1} r_1, && \nu_2 = \nu_1 \rho_1 , && \rho_1 > 0 .
\end{aligned}
\end{equation}
Recall also that the set $Q = \left\{\left(x,r,\epsilon,\rho \right) : \rho = e^{-\epsilon^{-1}} \right\}$ is invariant for the system \eqref{4d_sys}. In chart $\mathcal K_i$ coordinates, then, the following are invariant:
\begin{align}\eqlab{Q12P}
Q_1' = \left\{\left(r_1, \epsilon, \rho_1, \nu_1 \right) : \nu_1 \rho_1 = e^{-\epsilon^{-1}} \right\},  \qquad Q_2' = \left\{\left(x_2, r_2,  \epsilon, \nu_2 \right) : \nu_2 = e^{-\epsilon^{-1}} \right\} .
\end{align}

\subsubsection*{$\mathcal K_1$ Chart}

The dynamics in chart $\mathcal K_1$ are governed by
\begin{equation}
\eqlab{sphere1K1}
\begin{split}
r_1' &= r_1 \left( (1 + \epsilon) \left(1 - b r_1 (1 - 2 \nu_1 \rho_1 ) \right) + \rho_1 r_1 \epsilon \left(a + \rho_1 r_1 \nu_1^2 \right) \right)  , \\
\epsilon' &= - \epsilon^2 \left(1 - b r_1 (1 - 2 \nu_1 \rho_1 ) \right) , \\
\rho_1' &= - \rho_1 \left( 1 - b r_1 (1 - 2 \nu_1 \rho_1 ) - \rho_1 r_1 \epsilon \left(a + \rho_1 r_1 \nu_1^2 \right) \right) , \\
\nu_1' &= - \nu_1 \rho_1 r_1 \epsilon \left(a + \rho_1 r_1 \nu_1^2 \right).
\end{split}
\end{equation}
after a suitable desingularization (division by $\nu_1$). The manifold of equilibria $C_1$ in \lemmaref{expK1lemma} shows up as a manifold of equilibria $C_1'$ for the system in the invariant hyperplane $\epsilon = 0$ given by
\begin{equation}
\begin{split}
r_1' &= r_1 \left(1 - b r_1 (1 - 2 \nu_1 \rho_1 ) \right)  , \\
\rho_1' &= - \rho_1 \left(1 - b r_1 (1 - 2 \nu_1 \rho_1 ) \right) , \\
\nu_1' &= 0,
\end{split}
\end{equation}
and the manifold of equilbira $S_1$ in \lemmaref{expK1lemma} shows up as a manifold of equilibria $S_1'$ for the system in the invariant hyperplane $\rho_1 = 0$, given by
\begin{equation}
\begin{split}
r_1' &= r_1 \left( (1 + \epsilon) \left(1 - b r_1 \right) \right)  , \\
\epsilon' &= - \epsilon^2 \left(1 - b r_1 \right) , \\
\nu_1' &= 0 .
\end{split}
\end{equation}
We are interested in the intersection $C_1' \cap S_1'$, since this is contained within the domain of interest $Q_1'$. In particular, this intersection constitutes a line of equilbria
\[
\mathcal C_1' = C_1' \cap Q_1' = S_1' \cap Q_1' = C_1' \cap S_1' = \left\{\left(b^{-1}, 0 , 0 , \nu_1 \right) : \nu_1 \geq 0 \right\} ,
\]
in the $\epsilon = \rho_1 = 0$ plane. The line $\mathcal C_1'$ terminates at the point $P_L : (b^{-1},0,0,0)$, which sits on the equator of the invariant sphere segment $\nu_1 = \epsilon = 0$. In fact, within the invariant plane $\nu_1 = \epsilon = 0$ we have
\begin{equation}
\begin{split}
r_1' &= r_1 \left(1 - b r_1 \right)  , \\
\rho_1' &= - \rho_1 \left(1 - b r_1 \right) ,
\end{split}
\end{equation}
which has a line of equilibria
\[
\mathcal N_1' = \left\{\left(b^{-1}, 0, \rho_1, 0 \right) : \rho_1 \geq 0 \right\} 
\]
emanating from $P_L$. Finally, we identify a line of equilibria along the positive $\nu_1-$axis, i.e. along
\[
\mathfrak l_{e,1}' = \left\{(0, 0, 0, \SJ{\nu_1}) : \SJ{\nu_1 \geq 0} \right\} .
\]

\begin{lemma}\lemmalab{sphere1K1}
	The following hold for system \eqref{sphere1K1}:
	\begin{enumerate}
		\item[(i)] The point $P_L$ is partially hyperbolic with a single non-zero eigenvalue $\lambda = -1$, and there exists a corresponding three-dimensional center manifold $M_1'$ tangent to the center eigenspace $E^c(P_L)$, which can be chosen to be the extension of the manifold $M_1$.
		
		Locally, $M_1'$ contains the two-dimensional manifolds of equilbria $C_1'$ and $S_1'$ as restrictions
		\[
		M_1'\big|_{\epsilon = 0} = C_1', \qquad M_1'\big|_{\rho_1 = 0} = S_1',
		\]
		and the one-dimensional manifolds of equilbria $\mathcal C_1'$ and $\mathcal N_1'$ as restrictions
		\[
		M_1'\big|_{\epsilon = \rho_1 = 0} = \mathcal C_1', \qquad M_1'\big|_{\epsilon = \nu_1 = 0} = \mathcal N_1' .
		\]
		The variables $\epsilon, \rho_1$ ($r_1$) are increasing (decreasing) along $M_1' \setminus \left(C_1' \cup S_1' \cup \mathcal N_1' \right)$.
		\item[(ii)] The eigenvalues associated with the linearization along $\mathfrak l_{e,1}'$ are given by
		\[
		\lambda = 1, \ 0, \ -1, \ 0,
		\]
		implying that $\mathfrak l_{e,1}'$ is a line of partially hyperbolic saddles.
	\end{enumerate}
\end{lemma}

\begin{proof}
	Assertion (i) is standard center manifold theory. Assertion (ii) is a direct calculation.
\end{proof}

\subsubsection{$\mathcal K_2$ Chart}

The system in chart $\mathcal K_2$ is given by
\begin{equation}
\eqlab{eq:cyl_spheres_chart_2}
\begin{split}
x_2' &= - x_2 \left( x_2 + b r_2 (1 - 2 \nu_2) \right) + r_2 \epsilon \left( a + r_2 \nu_2^2 \right), \\
r_2' &= - r_2 (2 + \epsilon ) \left( x_2 + b r_2 (1 - 2 \nu_2) \right), \\
\epsilon' &= \epsilon^2 \left( x_2 + b r_2 (1 - 2 \nu_2) \right), \\
\nu_2' &= \nu_2 \left( x_2 + b r_2 (1 - 2 \nu_2) \right),
\end{split}
\end{equation}
after a suitable desingularization (division by $\nu_2$). At this point we note that the analysis may be simplified by restricting to the invariant set $Q_2'$. By doing so we reduce the dimension by 1 \KUK{after eliminating $\nu_2$}, and consider the reduced system on this set
\begin{equation}
\begin{split}
\eqlab{sphere1K2}
x_2' &= - x_2 \left( x_2 + b r_2 \left(1 - 2 e^{- \epsilon^{-1}} \right) \right) + r_2 \epsilon \left( a + r_2 e^{- 2 \epsilon^{-1}} \right), \\
r_2' &= - r_2 (2 + \epsilon ) \left( x_2 + b r_2 \left(1 - 2 e^{- \epsilon^{-1}}\right) \right), \\
\epsilon' &= \epsilon^2 \left( x_2 + b r_2 \left(1 - 2 e^{- \epsilon^{-1}}\right) \right).
\end{split}
\end{equation}
Note that the system obtained by restricting to $\epsilon = 0$ in \eqref{sphere1K2} is equivalent to the system obtained by restricting to $\nu_2 = \epsilon =0$ in system \eqref{eq:cyl_spheres_chart_2}, and given by
\begin{equation}
\eqlab{sphere1K2eps0}
\begin{split}
x_2' &= - x_2 \left( x_2 + b r_2 \right), \\
r_2' &= - 2 r_2 \left( x_2 + b r_2 \right) .
\end{split}
\end{equation}
Here we identify the line of equilbria
\[
\mathcal N_2' = \left\{\left(x_2, - b^{-1} x_2, 0 \right) : x_2 \leq 0 \right\}
\]
corresponding to the extension of $\mathcal N'$ in chart $\mathcal K_2$, as well as a line of equilibria
\[
L_{e,2}' = \left\{\left(0, 0, \epsilon\right) : \epsilon \geq 0  \right\} .
\]

\begin{lemma}
	\lemmalab{sphere1K2}
	The following holds for the system \eqref{sphere1K2}:
	\begin{enumerate}
		\item[(i)] The line of equlibria $\mathcal N_2'$ has a single \KUK{non-trivial} eigenvalue $\lambda = x_2 \leq 0$. Hence $\mathcal N_2'$ is normally hyperbolic and attracting for $x_2 < 0$, and terminates in a non-hyperbolic point at the origin $P_O:(0,0,0)$.
		\item[(ii)] There exists a unique center manifold $M_2'$ with base along compact subsets of $\mathcal N_2$ bounded away from $P_O$. The manifold $M_2'$ can be identified with the extension of the manifold $M_1'$ identified in chart $\mathcal K_1$ coordinates in \lemmaref{sphere1K1}. The variable $r_2$ is decreasing along $M_2' \setminus \mathcal N_2'$.
		\item[(iii)] The line $L_{e,2}'$ is non-hyperbolic, and coincides where domains overlap with the non-hyperbolic line $L_3$ identified in chart $K_3$ coordinates.
	\end{enumerate}
\end{lemma}

\begin{proof}
	Statement (i) follows immediately after linearization of the system \eqref{sphere1K2eps0} along $\mathcal N_2'$.
	
	The statement (ii) follows from center manifold theory, together with uniqueness of the manifold $M_1'$ described in \lemmaref{sphere1K1} and an application of the transition map $\kappa_{12}'$ in \eqref{CylSpheresTransMaps}.
	
	The statement (iii) follows after linearization of the system \eqref{sphere1K2} and an application of the blow-down transformation.
\end{proof}

\begin{figure}[h!]
	\centering
	\includegraphics[scale=0.9]{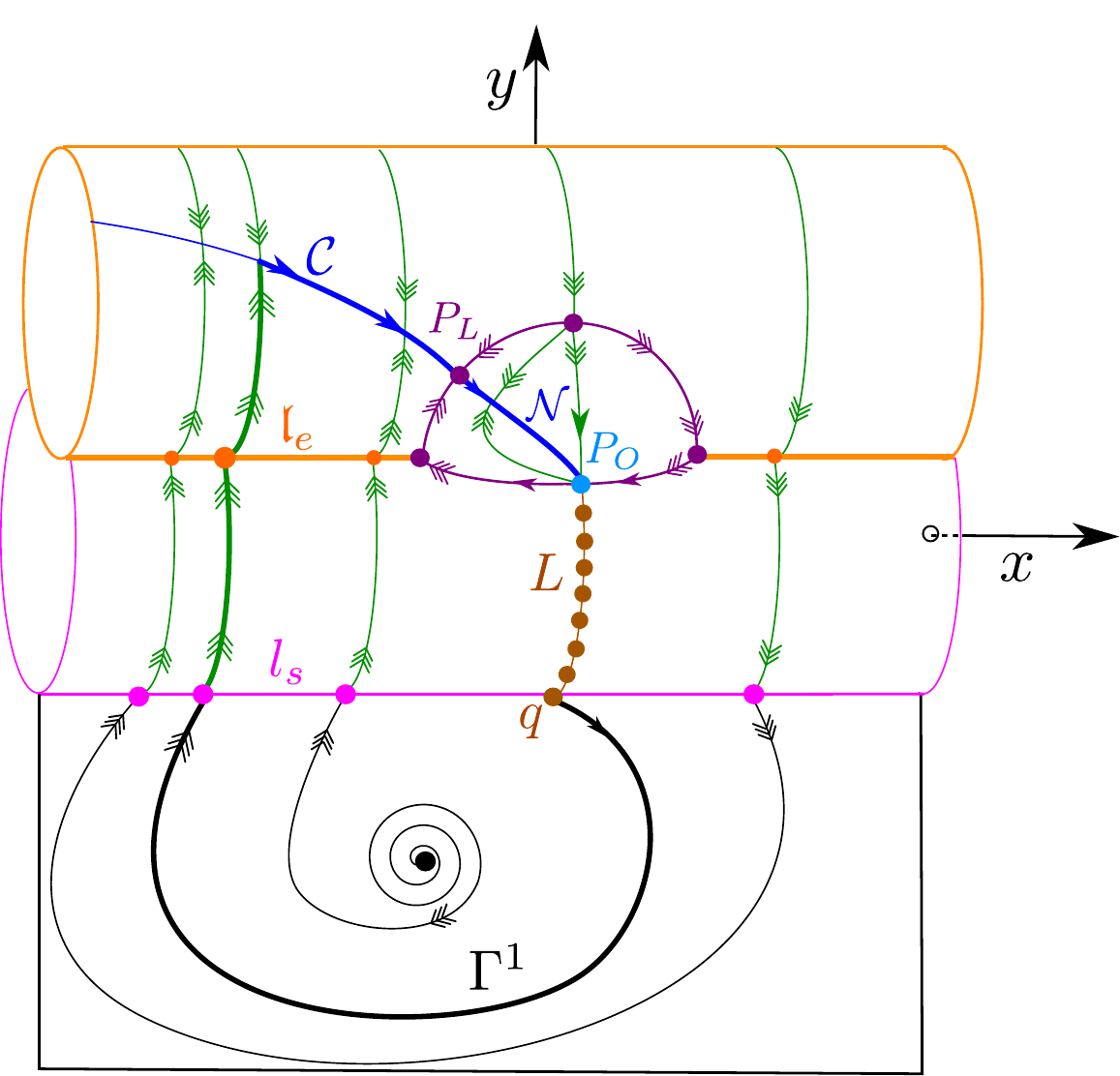}
	\caption{\PSnew{Illustration of the blowup of the degenerate equilibrium $P$ (purple) to the  `outer sphere' (purple) 
within $\bar \epsilon=0$. On the sphere we find the  attracting critical manifold $\mathcal N$ (blue) with reduced flow
connecting $P_L$ to the degenerate equilibrium $P_O$ (cyan). The fast flow on the sphere is shown in green.} 
 }	\figlab{fig:cylinder_of_spheres}
\end{figure}

The situation is illustrated in \figref{fig:cylinder_of_spheres}.

\begin{remark}
	\figref{fig:cylinder_of_spheres} shows dynamics not visible in charts $\mathcal K_1$ or $\mathcal K_2$. To gain a complete picture of the dynamics one must look in the additional coordinate charts $\mathcal K_3: \bar r = 1$ and $\mathcal K_4: \bar x = 1$. In particular, a hyperbolic saddle $P_R$ is identified at $(\bar x, \bar r, \bar \rho) = (0, 1, 0), \  \epsilon = 0$ in  chart $\mathcal K_3$, and a partially hyperbolic line of saddle-type equlibria is identified along the positive $\bar x-$axis in chart $\mathcal K_4$. We omit the details here for expository reasons.
\end{remark}

\subsubsection{Spherical blow-up of the point $P_O$}
\seclab{spherical_blowup}

Here we consider the dynamics near the point $P_O$ in system \eqref{sphere1K2} by means of a spherical blow-up. We drop subscripts in \eqref{sphere1K2}, and define the blow-up by the transformation
\begin{align}
\eqlab{SphericalBlowup}
\sigma \ge 0,\,(\bar x,\bar r, \bar \epsilon)\in S^2 \mapsto \begin{cases}
x &= \sigma \bar x,\\
r &= \sigma \bar r, \\
\epsilon &= \sigma \bar \epsilon.
\end{cases}
\end{align}
We are primarily interested in the dynamics observable in the phase directional charts
\[
\widetilde{\mathcal K}_1 : \bar r = 1, \qquad \widetilde{\mathcal K}_2 : \bar \epsilon = 1 ,
\]
for which we introduce the chart specific coordinates
\begin{equation}
\eqlab{shpere_K2}
\begin{aligned}
&\widetilde{\mathcal K}_1 : x = \sigma_1 x_1, &&r = \sigma_1, && \epsilon = \sigma_1 \epsilon_1 , \\
&\widetilde{\mathcal K}_2 : x = \sigma_2 x_2, &&r = \sigma_2 r_2, && \epsilon = \sigma_2 .
\end{aligned}
\end{equation}
We adopt the following notational convention: given an object $\gamma$ identified in chart $\mathcal K_2$, we denote its image under the blow-up transformation \eqref{SphericalBlowup} by $\tilde{\gamma}$, and it's image in a particular coordinate chart $\widetilde{\mathcal K}_i$ by $\tilde \gamma_i$.

The transition map between charts $\widetilde{\mathcal K}_1$ and $\widetilde{\mathcal K}_2$ is given by
\begin{equation}
\begin{aligned}
&\tilde \kappa_{12} : x_1 = r_2^{-1} x_2, &&\epsilon_1 = r_2^{-1}, && \sigma_1 = \sigma_2 r_2, && r_2 > 0, \\
&\tilde \kappa_{21} : x_2 = \epsilon_1^{-1} x_1,  && r_2 = \epsilon_1^{-1}, &&\sigma_2 = \sigma_1 \epsilon_1, && \epsilon_1 > 0.
\end{aligned}
\end{equation}

\subsubsection*{$\widetilde{\mathcal K}_1$ Chart}

The dynamics in chart $\widetilde{\mathcal K}_1$ are governed by
\begin{equation}
\eqlab{Sphere2K1}
\begin{split}
x_1' &= \epsilon_1 \left(a + \sigma_1 e^{-2 (\sigma_1 \epsilon_1)^{-1}} \right) + x_1 \left(1 + \sigma_1 \epsilon_1 \right) \left(x_1 + b \left(1 - 2 e^{-(\sigma_1 \epsilon_1)^{-1}} \right) \right) , \\
\epsilon_1' &= 2 \epsilon_1 \left(1 + \sigma_1 \epsilon_1 \right) \left(x_1 + b \left(1 - 2 e^{-(\sigma_1 \epsilon_1)^{-1}} \right) \right) ,  \\
\sigma_1' &= - \sigma_1 \left(2 + \sigma_1 \epsilon_1 \right) \left(x_1 + b \left(1 - 2 e^{-(\sigma_1 \epsilon_1)^{-1}} \right) \right) ,
\end{split}
\end{equation}
after a suitable desingularization (division by $\sigma_1$). We identify an equilibrium for the system \eqref{Sphere2K1} at $p_r: (0,0,0)$, as well as a line of equilibria
\[
\widetilde{\mathcal N}_1 = \left\{(-b, 0, \sigma_1) : \sigma_1 \geq 0 \right\} ,
\]
corresponding to the extension of the line of equilbria $\mathcal N_2'$ observed in chart $\mathcal K_2$. Note that $\widetilde{\mathcal{N}}_1$ terminates at $p_l : (-b, 0, 0)$.

\begin{lemma}
	\lemmalab{sphere2_K1}
	The following hold for system \eqref{Sphere2K1}:
	\begin{enumerate}
		\item[(i)] The point $p_l : (- b, 0, 0)$ is partially hyperbolic with a single non-zero eigenvalue $\lambda = -b < 0$. There exists a corresponding two-dimensional center manifold $\widetilde M_1$ tangent to the center eigenspace $E^c(p_l)$, which can be chosen to be the extension of the manifold $M_1'$.
		
		The manifold $\widetilde M_1|_{\epsilon_1 = 0}$ contains the one-dimensional manifold $\widetilde{\mathcal N}_1$ as a normally hyperbolic and attracting manifold of equilbria. Moreover, $\widetilde M_1|_{\sigma_1 = 0}$ contains a unique one-dimensional center manifold $\widetilde{\mathcal W}_1$, \SJ{and the slow flow on $\widetilde{\mathcal W}_1$ is increasing in the $\epsilon_1-$coordinate}.
		\item[(ii)] The equilibrium $p_r: (0, 0, 0)$ is a hyperbolic saddle with eigenvalues
		\[
		\lambda = b, \ 2 b, \ -2 b .
		\]
		The stable manifold $W^s(p_r)$ is contained in $x_1 = \epsilon_1 = 0$, which is invariant.
	\end{enumerate}
\end{lemma}

\begin{proof}
	The statement (i) follows after linearization of the system \eqref{Sphere2K1} and an application of the center manifold theorem. In particular, one obtains the following graph expression for $\widetilde{\mathcal W}_1$ via the usual matching approach:
	\[
	\widetilde{\mathcal W}_1 : x_1 = -b + \frac{a}{1 + b} \epsilon_1 + \mathcal O(\epsilon_1^2) .
	\]
	Hence the dynamics on $\widetilde{\mathcal W}_1$ are governed by
	\[
	\begin{split}
	x_1' = \frac{a \epsilon_1}{1+b} + \mathcal O(\epsilon_1^2) , \\
	\epsilon_1' = \frac{2 a \epsilon_1^2}{1+b} + \mathcal O(\epsilon_1^3) ,
	\end{split}
	\]
	from which assertion (i) follows.
	
	The statement (ii) follows after linearization of the system \eqref{Sphere2K1}, and the observation that restriction to $x_1 = \epsilon_1 = 0$ gives
	\[
	\begin{split}
	x_1' &= 0, \\
	\sigma_1' &= - 2 b \sigma_1 .
	\end{split}
	\]
\end{proof}

\subsubsection*{$\widetilde{\mathcal K}_2$ Chart}\seclab{MathcalK2PP}

The dynamics in chart $\widetilde{\mathcal K}_2$ are governed by
\begin{equation}
\eqlab{sphere2_eqns}
\begin{split}
x_2' &= r_2 \left( a + \sigma_2 r_2 e^{-2 \sigma_2^{-1}} \right) - x_2 \left(1 + \sigma_2 \right) \left(x_2 + b r_2 \left(1 - 2 e^{-\sigma_2^{-1}} \right) \right) , \\
r_2' &= - 2 r_2 \left(1 + \sigma_2 \right) \left(x_2 + b r_2 \left(1 - 2 e^{-\sigma_2^{-1}} \right) \right) ,  \\
\sigma_2' &= \sigma_2^2  \left(x_2 + b r_2 \left(1 - 2 e^{-\sigma_2^{-1}} \right) \right) ,
\end{split}
\end{equation}
after a suitable desingularization (division by $\sigma_2$). The system \eqref{sphere2_eqns} has a line of equilibria
\begin{equation}
\eqlab{L_2e}
\widetilde{L}_{e,2} = \left\{(0, 0, \epsilon) : \epsilon \geq 0 \right\} .
\end{equation}

\begin{lemma}\lemmalab{sphere2_phase_plane}
	The following hold for system \eqref{sphere2_eqns}:
	\begin{enumerate}
		\item[(i)] The unique one-dimensional center manifold $\widetilde{\mathcal W}_2 = \tilde \kappa_{12}(\widetilde{\mathcal W}_1)$ is contained \KUK{within} $\sigma_2 = 0$, and \KUK{forward asymptotic to} the nonhyperbolic point $p_o : (0,0,0)$. In particular, $\mathcal W_2$ approaches $p_o$ tangent to the positive $x_2-$axis.
		\item[(ii)] The line $\widetilde{L}_{e,2}$ is non-hyperbolic, and coincides where domains overlap with the non-hyperbolic line $L_{e,2}'$ observed in the $\mathcal K_2'$ chart (and hence with the non-hyperbolic line $L_3$ observed in chart $K_3$).
	\end{enumerate}
\end{lemma}


\begin{figure}
	\begin{center}
		\subfigure[]{\includegraphics[width=.52\textwidth]{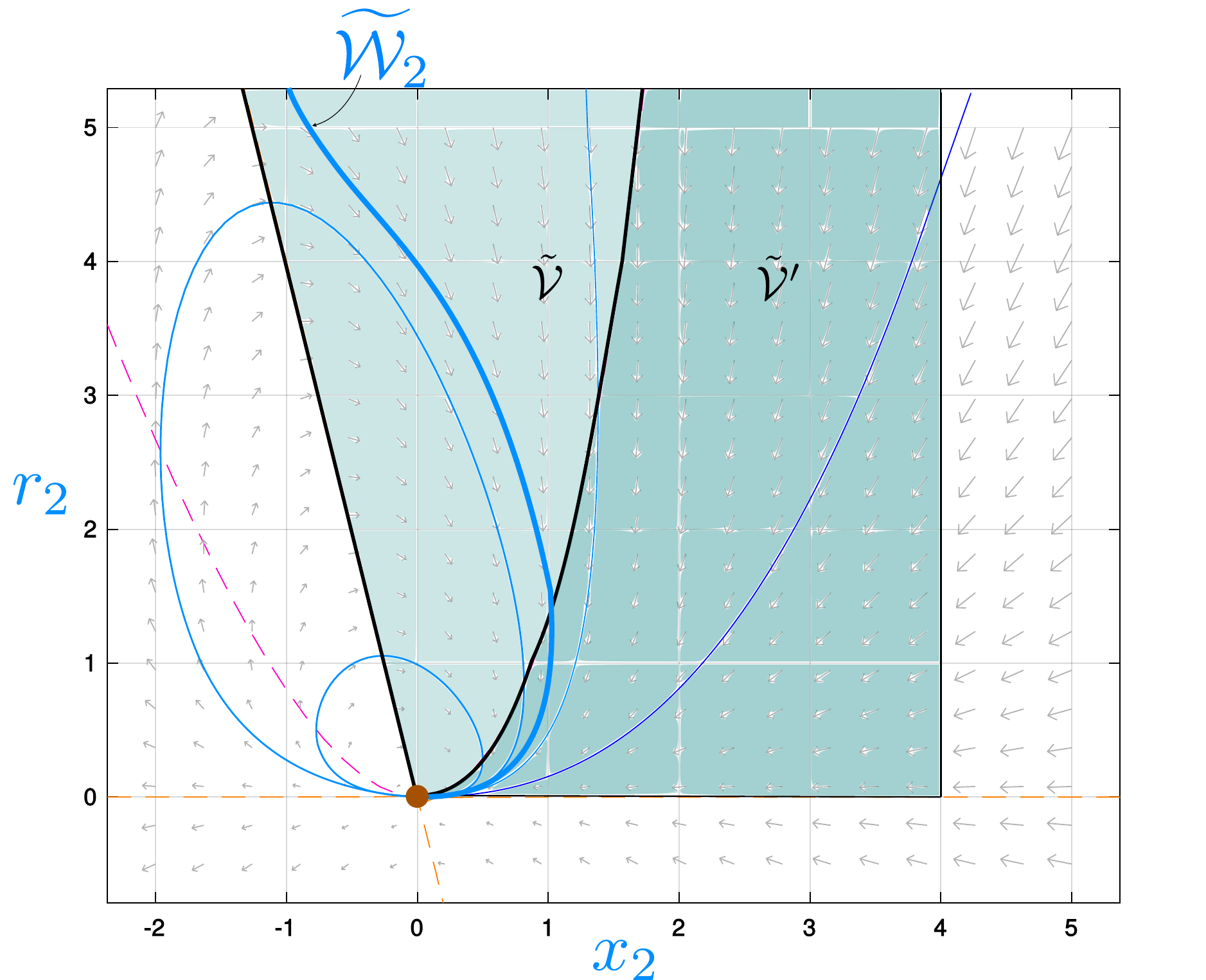}}
		\label{Sphere2_phase_plane_a}
		\subfigure[]{\includegraphics[width=.42\textwidth]{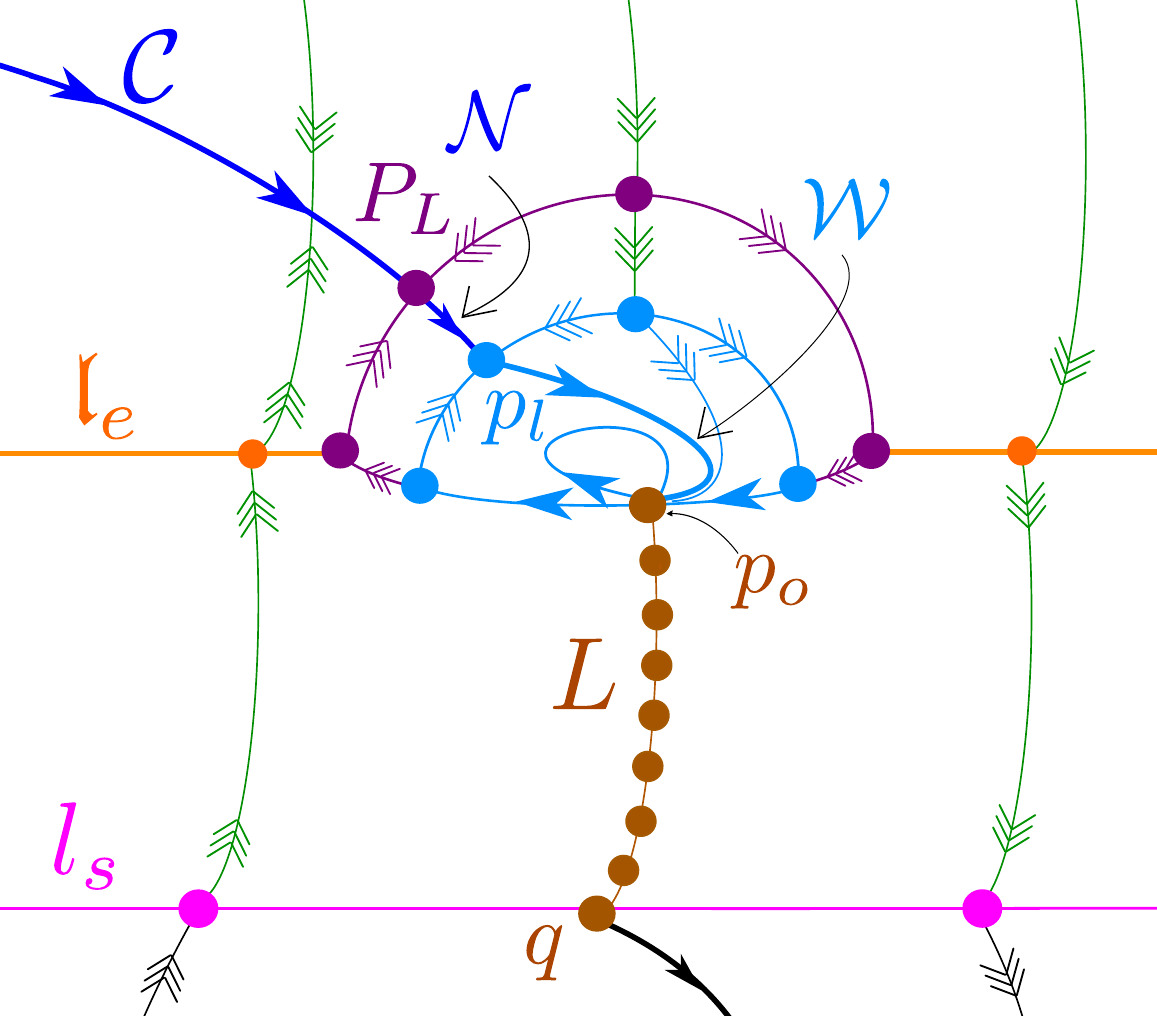}}
		\caption{\PSnew{Blow-up of the degenerate equilibrium $P_O$ (cyan) from  \figref{fig:cylinder_of_spheres} to a sphere (cyan).
On the sphere the equilibrium $p_l$ is connected to the degenerate equilibrium $p_0$ (brown) by the manifold (heteroclinic orbit) 
\SJnew{$\mathcal W$} (cyan). In (a): The phase plane $\sigma_2=0$  in chart $\widetilde{\mathcal K}_2$
 used in the proof of \lemmaref{sphere2_phase_plane} regarding the asymptotic properties of the local version $\widetilde{\mathcal W}_2$ of $\mathcal W$ in chart $\widetilde{\mathcal K}_2$. The plane  $\sigma_2=0$ covers the sphere $\sigma=0$ viewed from $\bar \epsilon =1$.  In (b): A global picture for comparison.   }}
		\figlab{Sphere2_phase_plane}
	\end{center} 
\end{figure}

\begin{proof}
	In order to prove the assertion (i) we consider the system in the invariant plane $\sigma_2 = 0$:
	\begin{equation}
	\eqlab{planar_sphere2}
	\begin{split}
	x_2' &= a r_2 - x_2 \left(x_2 + b r_2 \right) , \\
	r_2' &= - 2 r_2 \left(x_2 + b r_2 \right) .
	\end{split}
	\end{equation}
	The system \eqref{planar_sphere2} has a single non-hyperbolic equilibrium $p_o$ at $(0,0)$. Moreover, the region
	\[
	\mathcal V = \left\{(x_2, r_2): -b^{-1} r_2 \leq x_2 \leq a b^{-1}, r_2 \geq 0 \right\} 
	\]
	bounded by the $x_2-$axis, the $r_2-$nullcline $\{(-b^{-1} r_2, r_2): r_2 \geq 0\}$, and the vertical asymptote $\{(a b^{-1}, r_2): r_2 \geq 0 \}$ in the $x_2-$nullcline, is forward invariant. In particular, the $x_2-$axis is invariant with dynamics
	\[
	x_2' = -x_2^2 ,
	\]
	so that $x_2' < 0$ for $x_2 \neq 0$; see \figref{Sphere2_phase_plane}. Now define a compact subset $\widetilde{\mathcal V} \subset \mathcal V$ by
	\[
	\widetilde{\mathcal V} = \left\{(x_2, r_2) \in \mathcal V : r_2 \leq c_2 \right\} ,
	\]
	and choose $c_2 >0$ sufficiently large so that by \lemmaref{sphere2_K1}, $\widetilde{\mathcal W_2} = \kappa_{12}(\widetilde{\mathcal W}_1)$ enters $\widetilde{\mathcal V}$ transversally through $r_2 = c_2$.
	Since $\widetilde{\mathcal V}$ is compact and forward invariant, the Poincar\'e-Bendixon theorem applies, and $\widetilde{\mathcal W_2}$ is forward asymptotic to $p_o$ at $(0,0)$.
	\begin{figure}[h!]
		\centering
		\includegraphics[scale=0.9]{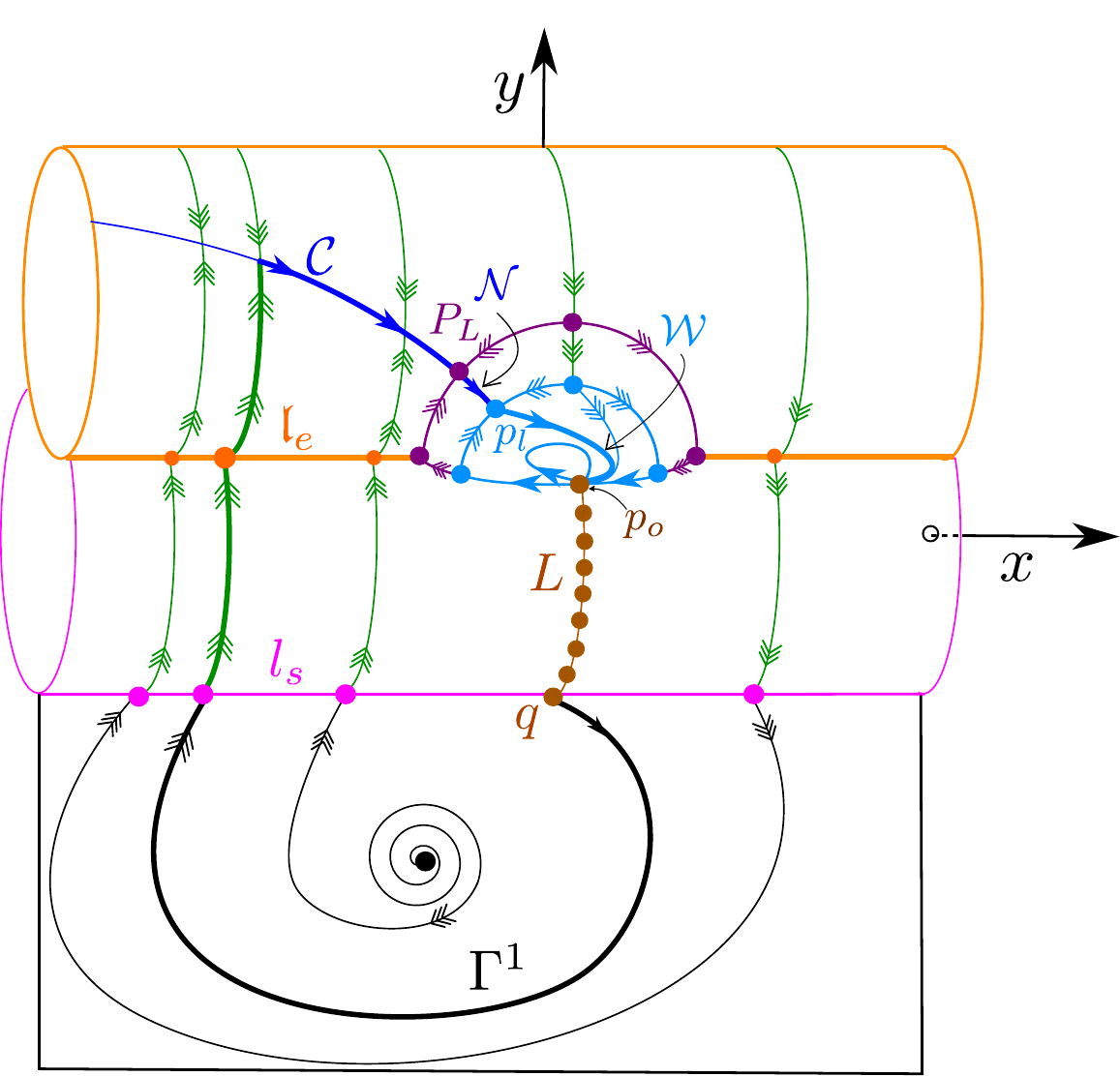}
		\caption{\PSnew{Results after two spherical blow-ups: outer sphere (purple), inner sphere (cyan).} }
		\figlab{corbeillerSphere2}
	\end{figure}
	To see that $\widetilde{\mathcal W_2}$ approaches $(0,0)$ tangent to the positive $x_2-$axis, notice that trajectories in $\widetilde{\mathcal V} \setminus \{(0,0)\}$ reach the forward invariant region
	\[
	\widetilde{\mathcal V}' = \left\{(x_2, r_2) \in \widetilde{\mathcal V} : r_2 \leq \frac{x_2^2}{a - b x_2} \right\}
	\]
	bounded above by the component of the $x_2-$nullcline in the positive quadrant in finite time. Hence $\widetilde{\mathcal W_2}$ approaches $(0,0)$ from within $\widetilde{\mathcal V}'$, and therefore tangent to the positive $x_2-$axis.
	
	\
	
	Statement (ii) is a straightforward calculation and application of \KUK{the} (successive) blow-down transformations.
\end{proof}

The situation is sketched in \figref{corbeillerSphere2}.

\subsubsection{Blow-up of $\widetilde{L}_{e,2}$}
\seclab{vertical_cylinder_top}
\PSnew{
Finally, in this subsection, we consider a cylindrical blow-up of the non-hyperbolic line $\widetilde{L}_{e,2}$ identified in equation \eqref{L_2e}.  This is done in a neighborhood of the point $p_0$ covered by chart $\widetilde{\mathcal K}_2$.
In \secref{lower_L_blowup} we will carry out a similar cylindrical blow-up of the line $L$ within the first cylindrical blow-up 
as defined in \eqref{cyl_blowup_map1},  which needs to be carried out  in coordinate charts $K_i$, $i = 1,2,3$
defined in \eqref{K1corbeiller}, \eqref{K2Here},  and \eqref{K3}. There, we will also show that these two cylindrical
blow-ups match up.}

We start with system \eqref{sphere2_eqns}, drop the subscripts, introduce a hat notation \PSnew{(needed in the process of matching the results obtained here with the results obtained in  \secref{lower_L_blowup})}, and define \SJnew{a weighted} blow-up transformation by the map
\begin{align}
\eqlab{VertCylBlowup2}
\hat s \ge 0,\,(\bar{\hat x}, \bar{\hat r} )\in S^1 \mapsto \begin{cases}
\hat x &= \hat s \bar{\hat x} ,\\
\hat r &= \hat s^2 \bar{\hat r} .
\end{cases}
\end{align}
We are primarily interested in the dynamics observable in coordinate charts
\[
\widehat K_{31} : \bar{\hat x} = 1, \qquad \widehat K_{32} : \bar{\hat r} = 1 ,
\]
for which we introduce chart specific coordinates
\begin{equation}
\eqlab{Torus_coords}
\begin{aligned}
&\widehat K_{31} : \hat x = \hat s_1, && \hat r = \hat s_1^2 \hat r_1 ,\\
&\widehat K_{32} : \hat x = \hat s_2 \hat x_2, && \hat r = \hat s_2^2.
\end{aligned}
\end{equation}
The transition map between charts $\widehat K_{31}$ and $\widehat K_{32}$ is given by
\begin{equation*}
\begin{aligned}
&\hat \kappa_{3132} : \hat s_1 = \hat s_2 \hat x_2, && \hat r_1 = \hat x_2^{-2},  && \hat x_2 > 0, \\
&\hat \kappa_{3231} : \hat s_2 = \hat s_1 \hat r_1^{1/2} , && \hat x_2 = \hat r_1^{-1/2} , && \hat r_1 > 0.
\end{aligned}
\end{equation*}
The subscript notation, although a little cumbersome, will be helpful in when considering the dynamics in coordinate charts covering the lower portion of the blown-up \SJ{line (}circle\SJ{)} $L$.

\subsubsection*{$\widehat K_{31}$ Chart}

The equations in the $\widehat K_{31}$ chart are given by
\begin{equation}
\eqlab{vert_cyl_top_K13}
\begin{split}
\hat r_1' &= - 2 \hat r_1^2 \left(a + \sigma \hat r_1 \hat s_1^2 e^{-2 \sigma^{-1}} \right) , \\
\sigma' &= \sigma^2 \left(1 + b \hat r_1 \hat s_1 \left(1 - 2 e^{-\sigma^{-1}} \right) \right) ,  \\
\hat s_1' &= - \hat s_1 \left( - \hat r_1 \left(a + \sigma \hat r_1 \hat s_1^2 e^{-2 \sigma^{-1}} \right) + (1 + \sigma) \left(1 + b \hat r_1 \hat s_1 \left(1 - 2 e^{-\sigma^{-1}} \right) \right) \right) ,
\end{split}
\end{equation}
after a suitable desingularization (division by $\hat s_1$). The system \eqref{vert_cyl_top_K13} has a single equilibrium at $p_s : (0,0,0)$.

\begin{lemma}\lemmalab{cyl_top}
	The following holds for system \eqref{vert_cyl_top_K13}:
	\begin{enumerate}
		\item[(i)] The equilibrium $p_s$ is partially hyperbolic with a single nonzero eigenvalue $\lambda = -1$ and a corresponding two-dimensional local center manifold $\widehat M_s$ given by $\hat s_1=0$. The 
		variable $\sigma$ is increasing along $\widehat M_s \cap \{\sigma > 0\}$ while $\hat r_1$ is decreasing along $\widehat M_s\cap \{\hat r_1>0\}$.
		\item[(ii)] The strong stable manifold $W^s(p_s)$ lies within $\hat r_1 = \sigma = 0$, and \KUK{the} $\hat r_1-, \ \sigma-, \ \hat s_1-$axes are all invariant. In particular, $\hat r_1$ is decreasing along the $\hat r_1-$axis, $\sigma$ is increasing along the $\sigma-$axis (which we denote by $\mathcal H$), and $\hat s_1$ is decreasing along the $\hat s_1-$axis. Hence, $p_s$ is a non-hyperbolic saddle.
	\end{enumerate}
\end{lemma}

\begin{proof}
	The statement (i) follows after linearization at $p_s$, and an application of the center manifold theorem.
	
	Invariance of the $\hat r_1-, \ \sigma-, \ \hat s_1-$axes follows immediately from the form taken by the equations when restricted to the respective axes. In $\{\hat s_1 = \hat r_1 = 0 \}$ we have
	\[
	\sigma' = \sigma^2 ,
	\]
	in $\{\hat s_1 = \sigma = 0 \}$ we have
	\[
	\hat r_1' = -2 a \hat r_1^2 ,
	\]
	and in $\{\sigma = \hat r_1 = 0 \}$ we have
	\[
	\hat s_1' = - \hat s_1 .
	\]
	The assertion (ii) follows.
\end{proof}

\subsubsection*{$\widehat K_{32}$ Chart}

We omit the details in chart $\widehat K_{32}$ for the sake of brevity, simply noting that calculations reveal no equilibria and an invariant flow along the `equator' $\hat s_2 = \sigma = 0$, as indicated in \figref{lc_upper_L_blowup}.
\begin{figure}
	\begin{center}
		\subfigure[]{\includegraphics[trim={0 0 0 0},scale=.6]{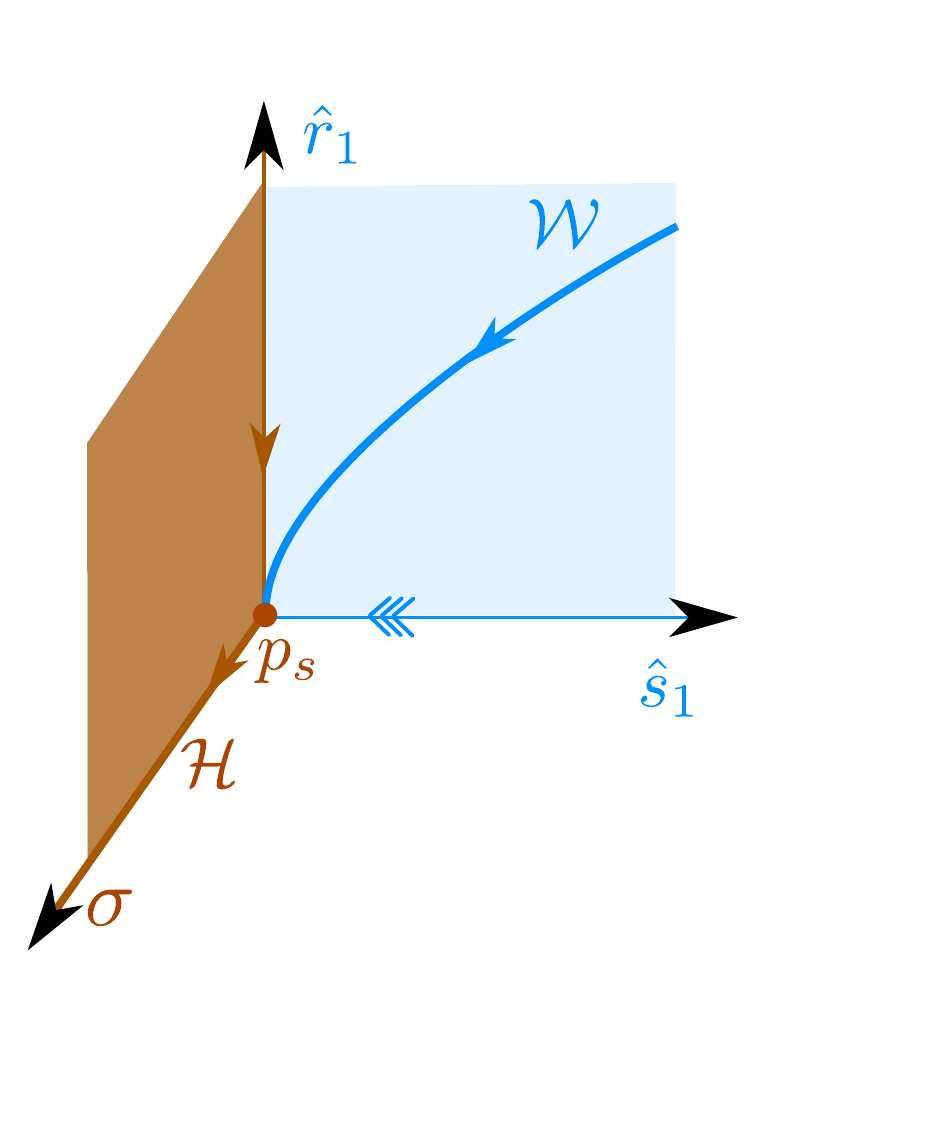}}
		\subfigure[]{\includegraphics[trim={0 -0.6cm 0 0},scale=1.3]{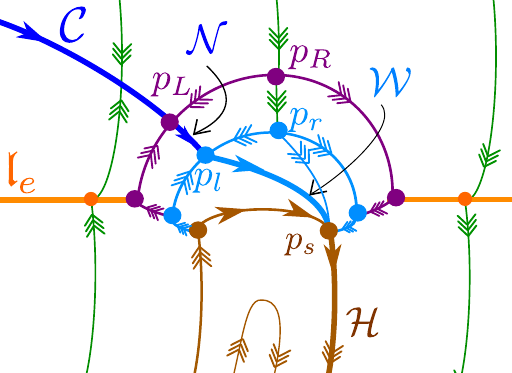}}
		\caption{\PSnew{Blow-up of the line of equilibria ${L}_e$ (brown) to a cylinder (brown).
In (a): Local picture near the hyperbolic equilibrium $p_s$ in chart $\widehat K_{31}$. The cylinder corresponds to the plane 
$\hat s_1=0$ (brown). The plane $\sigma=0$ corresponds to the (inner) sphere (cyan).
 By the cylindrical blowup of $L_e$, we have gained hyperbolicity. The line $\mathcal H$ (brown) is the (non-hyperbolic) unstable manifold of $p_s$ (also brown). In (b): A global picture for comparison.}}
		\figlab{lc_upper_L_blowup}
	\end{center} 
\end{figure}

\subsection{Blow-up of $L$ in the algebraic regime}
\seclab{L_blowup}

In order to \PS{obtain the fully nondegenerate} singular cycle, it remains to blow-up the vertical non-hyperbolic line (circle) $L$ in the algebraic regime. We return to the dynamics observable after the (first) cylindrical blow-up, as defined in \eqref{cyl_blowup_map1}, i.e. the problem considered in coordinate charts $K_i$, $i = 1,2,3$. In \secref{lower_L_blowup} we introduce the blow-up of $L$, and show that the dynamics observed in \secref{vertical_cylinder_top} can be related to the dynamics observed in this blow-up in an overlapping domain. In \secref{blow-up_vert_cyl_lower} we derive the qualitative properties of the dynamics associated with the `lower part' of the vertical cylinder in \figref{fig:lc_singular_cycle}.

\subsubsection{Blow-up of $L$}
\seclab{lower_L_blowup}

We consider the dynamics near the lower portion of the non-hyperbolic line (circle) $L$. We introduce a secondary weighted blow-up defined via the transformation
\[
\left(x, r, \left( \bar y, \bar \epsilon \right) \right) \mapsto \left(s, \left(\bar x, \bar r \right), \left( \bar y, \bar \epsilon \right) \right),
\]
where
\begin{align}
\label{cyl_blowup_map2}
s \ge 0, \ (\bar x,\bar r) \in S^1 \mapsto
\begin{cases}
x &= s \bar x,\\
r &= s^2 \bar r.
\end{cases}
\end{align}
Composing this with the map \eqref{cyl_blowup_map1}, we obtain
\begin{align}
\eqlab{VertBlowup}
s \ge 0, \ (\bar y, \bar \epsilon) \in S^1, \  (\bar x,\bar r) \in S^1 \mapsto
\begin{cases}
x &= s \bar x,\\
y &= s^2 \bar r \bar y, \\
\epsilon &= s^2 \bar r \bar \epsilon.
\end{cases}
\end{align}
Geometrically, the transformation \eqref{VertBlowup} blows up the circle of non-hyperbolic points $L$ to the torus $\{s=0\} \times S^1 \times S^1$, for which only the subset defined by $\bar r\ge 0$ and \SJnew{$\bar \epsilon \ge 0$} is relevant. In total, six coordinate charts are necessary for an understanding of the main dynamical features:
\begin{equation}\eqlab{K11K12}
\begin{aligned}
&K_{11}: \bar y = -1, \ \bar x = 1, && K_{21} : \bar \epsilon = 1, \ \bar x = 1, && K_{31} : \bar y = 1, \ \bar x = 1 , \\
&K_{12}: \bar y = -1, \ \bar r = 1, && K_{22} : \bar \epsilon = 1, \ \bar r = 1, && K_{32} : \bar y = 1, \ \bar r = 1 .
\end{aligned}
\end{equation}
In particular $\bar x = 1$ in charts $K_{i1}$ and $\bar r = 1$ in charts $K_{i2}$, for $i = 1,2,3$, and the subscript $i$ signifies the `visible region' of the first (horizontal) cylinder defined by the blow-up transformation \eqref{cyl_blowup_map1}. For charts $K_{1j}$ covering the region visible in $\bar y = -1$ we have chart specific coordinates
\begin{align}
&K_{11}: x = s_1, && y = - s_1^2 r_{11},  &&  \epsilon = s_1^2 r_{11} \epsilon_{1} , \nonumber\\
&K_{12}: x = s_2 x_2 ,  && y = - s_2^2 ,  &&  \epsilon = s_2^2 \epsilon_{1} .\nonumber
\end{align}
For charts $K_{2j}$ covering the region visible in $\bar \epsilon= 1$ we have chart specific coordinates
\begin{equation}
\eqlab{K21}
\begin{aligned}
&K_{21}: x = s_1, &&y = s_1^2 r_{21} y_{2}, && \epsilon = s_1^2 r_{21} , \\
&K_{22}: x = s_2 x_2 , && y = s_2^2 y_{2} , && \epsilon = s_2^2.
\end{aligned}
\end{equation}
For charts $K_{3j}$ covering the region visible in $\bar y = 1$ we have chart specific coordinates
\begin{equation}
\eqlab{K_3132}
\begin{aligned}
&K_{31}: x = s_1, && y = s_1^2 r_{31}, && \epsilon = s_1^2 r_{31} \epsilon_{3} ,\\
&K_{32}: x = s_2 x_2 , && y = s_2^2 , && \epsilon = s_2^2 \epsilon_{3} . 
\end{aligned}
\end{equation}
The transition maps between overlapping charts are given by
\begin{align*}
\eqlab{vert_cylinder_trans_maps}
&\kappa_{1112} : s_1 = s_2 x_2, && r_{11} = x_2^{-2} , && \  &&  x_2 > 0, \\
&\kappa_{1121} : r_{11} = -r_{21} y_2 , && \epsilon_1 = - y_2^{-1} ,   &&  \  &&   y_2 < 0, \\
&\kappa_{1122} : s_1 = s_2 x_2, && r_{11} = -x_2^{-2} y_2,  &&  \epsilon_1 = - y_2^{-1} ,  &&   x_2 > 0, y_2 < 0, \\
&\kappa_{1221} : s_2 = \left(-r_{21} y_2 \right)^{1/2}, && \epsilon_1 = -y_2^{-1},  &&  x_2 = \left(-r_{21} y_2 \right)^{-1/2} ,  &&   y_2 < 0, r_{21} > 0, \\
&\kappa_{2122} : s_1 = s_2 x_2, && r_{21} = x_2^{-2} ,  && \  &&  x_2 > 0, \\
&\kappa_{2131} : r_{21} = r_{31}\epsilon_3, && y_2 = \epsilon_3^{-1},  && \  &&  \epsilon_3 > 0, \\
&\kappa_{2132} : s_1 = s_2 x_2, &&  r_{21} = x_2^{-2} \epsilon_3,  && y_2 = \epsilon_3^{-1} ,   &&  x_2, \epsilon_3 > 0, \\
&\kappa_{2231} : s_2 = s_1 \left(r_{31} \epsilon_3\right)^{1/2} , && x_2 = s_1 \left(r_{31} \epsilon_3\right)^{-1/2} ,  &&  y_2 = \epsilon_3^{-1} ,   &&  r_{31}, \epsilon_3 > 0, \\
&\kappa_{3132} : s_1 = s_2 x_2, && r_{21} = x_2^{-2},  &&  \  && x_2 > 0 ,
\end{align*}
and their inverses can be computed directly using these expressions if necessary.

\

 We will focus in this section on the dynamics observable in charts $K_{ij}$, $i = 1,2$, $j = 1,2$. The dynamics in charts $K_{ij}$ with $i = 3$ have already been considered in \secref{vertical_cylinder_top}, as is shown in the following result.

\begin{lemma}
	\lemmalab{vert_cyl_trans_map_top}
	Coordinates $(r_{31}, \epsilon_3, s_1)$ in chart $K_{31}$ are related to coordinates $(\hat r_1, \sigma, \hat s_1)$ in chart $\widehat K_{31}$ via
	\begin{equation}
	\eqlab{exp_to_alg1}
	r_{31} = \sigma^{-1} \hat r_1, \qquad \epsilon_3 = \sigma e^{\sigma^{-1}}, \qquad s_1 = \sigma e^{-\sigma^{-1}} \hat s_1 , \qquad \sigma > 0.
	\end{equation}
	Coordinates $(x_2, \epsilon_3, s_2)$ in chart $K_{32}$ are related to coordinates $(\hat x_2, \sigma, \hat s_2)$ in chart $\widehat K_{32}$ via
	\[
	x_2 = \sigma^{1/2} e^{-(2 \sigma)^{-1}} \hat x_2, \qquad \epsilon_3 = \sigma, \qquad s_2 = \sigma^{1/2} e^{-(2 \sigma)^{-1}} \hat s_2 , \qquad \sigma > 0.
	\]
\end{lemma}

\begin{proof}
	The expressions given for the coordinates $(r_{31}, \epsilon_3, s_1)$ in chart $K_{31}$ are obtained by composing blow-up maps \eqref{K3}, \eqref{mathfrakK1}, \eqref{shpere_K2} and the coordinates for $\widehat K_{31}$ in \eqref{Torus_coords}. In the notation below,  we avoid dropping subscripts and append them at each coordinate change (except where the `hat notation' suffices). Explicitly, the first three compositions give
	\begin{align*}
	& x  = \nu_2 x_2 , \\
	&y = r_3 = \rho_1 r_{31} = \nu_2^2 r_{312} , \\
	& \epsilon = r_3 \epsilon_3  = \nu_2 r_{312} \epsilon_3 , \\
	& q  = \rho_1= \nu_2 ,
	\end{align*}
	and subsequent restriction to $\nu_2 = e^{-\epsilon_3^{-1}}$ gives
	\[
	x = e^{-\epsilon_3^{-1}} x_2 , \qquad y = e^{- 2 \epsilon_3^{-1}} r_{312} , \qquad \epsilon = e^{-\epsilon_3^{-1}} r_{312} \epsilon_3 .	
	\]
	The last two compositions give
	\begin{align*}
	x &= e^{-\epsilon_3^{-1}} x_2 = e^{- \sigma_2^{-1}} x_2 \sigma_2 = \sigma_2 e^{-\sigma_2^{-1}} \hat s_1 , \\
	y &= e^{- 2 \epsilon_3^{-1}} r_{312} = e^{- 2 \sigma_2^{-1}} \sigma_2 r_{312} = \sigma_2 \hat r_1 e^{-2 \sigma_2^{-1}} \hat s_1^2 , \\
	\epsilon &= e^{-\epsilon_3^{-1}} r_{312} \epsilon_3 = e^{- \sigma_2^{-1}} \sigma_2^2 r_{312} = \sigma_2^2 \hat r_1 e^{-\sigma_2^{-1}} \hat s_1^2 .
	\end{align*}	
	Dropping the subscript in $\sigma_2$ and comparing with the $K_{31}$ coordinates given in \eqref{K_3132} yields the result.
	
	The expressions given for the coordinates $(x_2, \epsilon_3, s_2)$ in chart $K_{32}$ are obtained by a similar argument: composing blow-up maps \eqref{K3}, \eqref{mathfrakK1}, \eqref{shpere_K2} and the coordinates for $\widehat K_{32}$ in \eqref{Torus_coords} gives
	\[
	x = \sigma \hat x_2 e^{-\sigma^{-1}} \hat s_2, \qquad y = \sigma e^{-2 \sigma^{-1}} \hat s_2^2, \qquad \epsilon = \sigma^2 e^{-\sigma^{-1}} \hat s_2^2 ,
	\]
	(where we have dropped the subscript in $\sigma_2$), and direct comparison with the expression for $K_{32}$ coordinates in \eqref{K_3132} yields the desired result.
\end{proof}

We obtain the following corollary.

\begin{cor}
	\label{trans_cor}
	The (invariant) $\sigma-$axis in system \eqref{vert_cyl_top_K13} is mapped to
	\[
	\mathcal H_{31} = \left\{(0, \epsilon_3, 0) : \epsilon_3 > 0 \right\}
	\]
	under the transformation \KUK{defined by the equations} \eqref{exp_to_alg1}, which is invariant for the system obtained in chart $K_{31}$ coordinates. Dynamics on $\mathcal H_{31}$ are governed by
	\begin{equation}
	\eqlab{H31_dynamics}
	\epsilon_3' = \epsilon_3 e^{-\epsilon_3^{-1}} .
	\end{equation}
\end{cor}

\begin{proof}
	This is an immediate consequence of \lemmaref{vert_cyl_trans_map_top} and the form of the equations obtained in chart $K_{31}$, which are given by
	\[
	\begin{split}
	r_{31}' &= - r_{31} \left(b r_{31} s_1 +e^{-\epsilon_3^{-1}} \left(1 - 2b r_{31} s_1 + 2 r_{31} a +2 r_{31} s_1^2 \right) \right) , \\
	\epsilon_3' &= \epsilon_3 \left(b r_{31} s_1 +e^{-\epsilon_3^{-1}} \left(1 - 2b r_{31} s_1 \right) \right) , \\
	s_1' &= s_1 r_{31} e^{- \epsilon_3^{-1}} \left(a + s_1^2 r_{31} \right) ,
	\end{split}
	\]
	after a suitable desingularization (division by $s_1$). The expression in \eqref{H31_dynamics} follows by restriction to $r_{31} = s_1 = 0$.
\end{proof}



\subsubsection{Blow-up for the lower part of $L$}
\seclab{blow-up_vert_cyl_lower}

\lemmaref{vert_cyl_trans_map_top} shows how the dynamics in the transitional regime can be related to the dynamics in the algebraic regime, after application of the blow-up transformation \eqref{VertBlowup}. Moreover, 
\lemmaref{cyl_top} and Corollary \ref{trans_cor} are sufficient for an understanding of the main dynamical features and in particular, the construction of $\Gamma_0$. Hence, we restrict attention here to dynamics in charts $K_{11}, K_{12}, K_{21}, K_{22}$ only in this section.

\subsubsection{$K_{22}$ Chart}

We can determine the equations in the chart $K_{22}$ by considering the system \eqref{lc_K2} on the fast time scale with $\epsilon = r_2$, i.e.
\begin{equation}
\eqlab{K2_fast}
\begin{split}
x' &= r_2^2 \left(a + r_2 y \right), \\
y_2' &= -x + b r_2 y_2 \left(2 - e^{y_2}\right),
\end{split}
\qquad \qquad \SJ{r_2 \ll 1 ,}
\end{equation}
and then apply the secondary \KUK{transformation defined by the equations
\begin{equation}
\eqlab{secondary_K22}
 \begin{split}
x &= s_2 x_2, \\ r &= s_2^2.
\end{split}
\end{equation}
This produces the following system
\begin{equation}
\eqlab{vert_cyl_K22}
\begin{split}
x_2' &= a + s_2^2 y_2 ,   \\
y_2' &= - x_2 + b s_2 y_2 (2 - e^{y_2}),
\end{split}
\qquad \qquad s_2 \ll 1,
\end{equation}
after a suitable desingularization (division by $s_2$).}

\begin{lemma}
	\lemmalab{K22lem}
	The system \eqref{vert_cyl_K22} is a regular perturbation problem, with leading order dynamics on compact domains determined by the dynamics of the limiting system
	\begin{equation}
	\eqlab{vert_cylinder_regular}
	\begin{split}
	x_2' &= a ,   \\
	y_2' &= - x_2 ,
	\end{split}
	\end{equation}
	for which all orbits are of the form
	\[
	y_2(x_2) = - \frac{x_2^2}{2 a} + c_0,
	\]
	for constants $c_0$.
\end{lemma}

\begin{proof}
	This follows by direct integration of the equations \eqref{vert_cylinder_regular}.
\end{proof}

\subsubsection{$K_{21}$ Chart}

The equations in chart $K_{21}$ can be determined by considering the system \eqref{K2_fast} and applying the secondary transformation \KUK{defined by the equations
\begin{align*}
x &= s_1, \\ 
r_2 &= s_1^2 r_{21} .
\end{align*}}
We obtain the following system,
\begin{equation}
\eqlab{vert_cyl_K21}
\begin{split}
y_2' &= -1 + b s_1 r_{21} y_2 \left(2 - e^{y_2} \right) , \\
r_{21}' &= -2 s_1^2 r_{21}^3 \left(a - s_1^2 r_{21} y_2 \right) ,   \\
s_1' &= s_1^3 r_{21}^2 \left(a - s_1^2 r_{21} y_2 \right)  , 
\end{split}
\end{equation}
after applying a time desingularization (division by $s_1$).

\begin{lemma}
	\lemmalab{K21lem}
	The system \eqref{vert_cyl_K21} is invariant in subspaces $r_{21} = 0$, $s_1 =0$, and along the invariant line
	\begin{align*}
	 \mathcal H_{21} = \left\{(y_2,0,0): y_2\in \mathbb R\right\}.
	\end{align*}
In all three subspaces $y_2$ is the only dynamic variable, with dynamics governed by
	\begin{equation}
	\eqlab{K21_H21}
	y_2' = - 1.
	\end{equation}
\end{lemma}

\begin{proof}
	Straightforward restriction.
\end{proof}

\subsubsection{$K_{12}$ Chart}

The equations in the chart $K_{12}$ can be determined by considering the system \eqref{sys_K1} and applying the secondary transformation
\[
x = s_2 x_2 , \qquad r_1 = s_2^2 .
\]
One obtains the system
\begin{equation}
\eqlab{vert_cyl_K12}
\begin{split}
x_2' &= a - s_2^2 - \frac{1}{2} x_2 \left(x_2 + b s_2 \left(2 - e^{- \epsilon_1^{-1}} \right) \right) ,   \\
\epsilon_1' &= - \epsilon_1 \left(x_2 + b s_2 \left(2 - e^{- \epsilon_1^{-1}} \right) \right) , \\
s_2' &= \frac{1}{2} s_2 \left(x_2 + b s_2 \left(2 - e^{- \epsilon_1^{-1}} \right) \right) ,
\end{split}
\end{equation}
after a suitable desingularization (division by $s_2$). The system \eqref{vert_cyl_K12} has three equilibria:
\[
q_i = \left(- \sqrt{2 a}, 0, 0 \right), \qquad p_{12} = \left(-2 b \sqrt{a}, \sqrt{a} ,0\right), \qquad q_o = \left(\sqrt{2 a}, 0, 0 \right) .
\]

\begin{lemma}
	\lemmalab{K12lem}
	The following holds for the system \eqref{vert_cyl_K12}:
	\begin{enumerate}
		\item[(i)] The equilbria $q_i$ and $q_o$ are hyperbolic saddles with eigenvalues
		\[
		\lambda_{1,i} = \sqrt{\frac{a}{2}}, \qquad \lambda_{2,i} = \sqrt{2 a}, \qquad \lambda_{3,i} = - \sqrt{\frac{a}{2}} ,
		\]
		and
		\[
		\lambda_{1,o} = - \sqrt{\frac{a}{2}}, \qquad \lambda_{2,o} = - \sqrt{2 a}, \qquad \lambda_{3,o} = \sqrt{\frac{a}{2}} ,
		\]
		respectively. There is a strong resonance in each case due to the relation $\lambda_{1,i/o} = \lambda_{2,i/o} + \lambda_{3,i/o}$.
		\item[(ii)] The equilibrium $p_{12}$ is an unstable focus within the invariant $\epsilon_1=0$ plane for any $b > 0$, $a \in (0,2)$, and coincides upon coordinate change with the true equilibrium of the system $p$.
		\item[(iii)] The lines
		\[
		\mathcal G_{\pm,12} = \left\{\left(\SJ{\pm \sqrt{2a}}, \epsilon_1, 0 \right) : \epsilon_1 \geq 0 \right\}
		\]
		are invariant.
	\end{enumerate}
\end{lemma}

\begin{proof}
	Statements (i) and (ii) are immediate upon linearization of the system \eqref{vert_cyl_K12} and an application of the blow-down map, respectively. To prove the statement (iii), consider the system in the invariant plane $s_2 = 0$:
	\begin{equation}
	\begin{split}
	x_2' &= a - \frac{1}{2} x_2^2 ,   \\
	\epsilon_1' &= - \epsilon_1 x_2 .
	\end{split}
	\end{equation}
	It is easy to verify that this system has invariant lines along $x_2 = \pm \sqrt{2a}$, $\epsilon_1 \geq 0$.
\end{proof}

\subsubsection{$K_{11}$ Chart}\seclab{K11}

The equations in the chart $K_{11}$ can be determined by considering the system \eqref{sys_K1} and applying the secondary transformation
\begin{align}
x = s_1 , \qquad r_1 = s_1^2 r_{11}.\eqlab{K11eqns}
\end{align}
One obtains the system
\begin{equation}
\eqlab{vert_cyl_K11}
\begin{split}
r_{11}' &= r_{11} \left(1 + b s_1 r_{11} \left(2 - e^{-\epsilon_1^{-1}}\right) - 2 r_{11} \left(a - r_{11} s_1^2 \right) \right) , \\
\epsilon_1' &= - \epsilon_1 \left(1 + b s_1 r_{11} \left(2 - e^{-\epsilon_1^{-1}}\right) \right), \\
s_1' &= s_1 r_{11} \left(a - r_{11} s_1^2 \right) ,
\end{split}
\end{equation}
\SJ{after a time desingularization (division by $s_1$),} for which there are two equilibria:
\[
q_s = (0,0,0), \qquad q_{o,2} = \left(0, \frac{1}{2 a}, 0 \right).
\]

\begin{lemma}
	\lemmalab{K11lem}
	The following holds for the system \eqref{vert_cyl_K11}:
	\begin{enumerate}
		\item[(i)] The equilibrium $q_s$ is partially hyperbolic with a eigenvalues
		\[
		\lambda = 1, -1, 0 .
		\]
		The equilibrium $q_{o,2}$ coincides with the hyperbolic saddle $q_{o}$ observed in chart $K_{12}$.
		\item[(ii)] The lines
		\[
		\mathcal G_{+,11} = \left\{\left(\frac{1}{2 a}, \epsilon_1, 0 \right) : \epsilon_1 \geq 0 \right\}, \qquad \mathcal H_{11} = \left\{(0, \epsilon_1, 0 ) : \epsilon_1 \geq 0 \right\},
		\]
		are invariant, with $\epsilon_1$ decreasing along $\mathcal G_{11}$, and decreasing along $\mathcal H_{11}$.
	\end{enumerate}
\end{lemma}

\begin{proof}
	The statement (i) follows immediately by a linearization of \eqref{vert_cyl_K11} and an application of the transition map $\kappa_{1112}$.
	
	To prove the statement (ii), consider the system in the invariant plane $s_1 = 0$:
	\begin{equation}
	\begin{split}
	r_{11}' &= r_{11} \left(1 - 2 a r_{11} \right) , \\
	\epsilon_1' &= - \epsilon_1 .
	\end{split}
	\end{equation}
	The equations decouple, and the lines along $r_{11} = 1/2a, \epsilon_1 \geq 0$ and $r_{11} = 0, \epsilon_1 \geq 0$ are invariant with dynamics governed in each case by
	\[
	\epsilon_1' = - \epsilon_1 .
	\]
	Hence $\epsilon_1$ is decreasing along $\mathcal G_{11}$, and decreasing along $\mathcal H_{11}$.
\end{proof}

\begin{figure}[h!]
	\centering
	\includegraphics[scale=0.8]{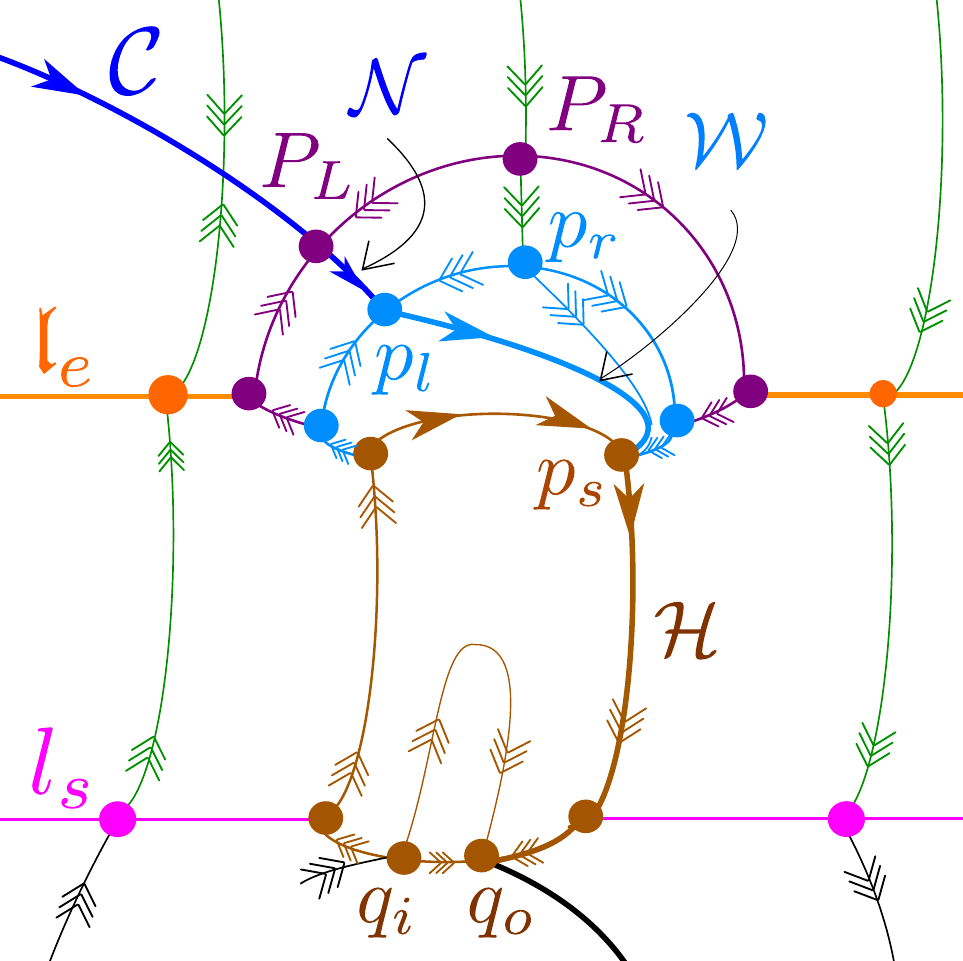}
	\caption{\PSnew{The complete desingularization near the line $L$ of degenerate equilibria.} }\figlab{lc_map2_blowup}
\end{figure}

Taken together, \lemmaref{K22lem}, \lemmaref{K21lem}, \lemmaref{K12lem} and \lemmaref{K11lem} imply the main dynamical features associated with the lower portion of the blow-up of $L$, as sketched in \figref{lc_map2_blowup}. Note the resemblance to the regular fold when viewed `from below' (as one might expect due to the presence of the quadratic tangency in the PWS system). For further details on the regular fold see \cite{2019arXiv190806781U}. This resemblance is further exemplified by the fact that the invariant lines $\mathcal G_\pm$ observed in the $K_{1j}$ charts connect in the region of the cylinder visible in the $K_{2j}$ charts in \figref{lc_map2_blowup}; this is shown below.

\begin{proposition}
	The invariant lines $\mathcal G_\pm$ observed in charts $K_{1j}$ connect along the invariant parabola given in chart $K_{22}$ coordinates by
	\[
	y_2(x_2) = - \frac{x_2^2}{2 a},\ r_2=0.
	\]
\end{proposition}

\begin{proof}
	Starting from chart $K_{11}$ and applying the relevant blow-down transformation, we may parameterize the $\mathcal G_+$ in chart $K_1$ coordinates as
	\[
	\mathcal G_{+,1} = \left\{\left(x, \frac{x^2}{2 a}, \epsilon_1 \right) : \epsilon_1 \geq 0 \right\}  .
	\]
	Applying the $\kappa_{12}$ transition map in \eqref{trans1} gives the following parameterization for $\mathcal G_+$ in chart $K_2$\SJnew{,}
	\[
	\mathcal G_{+,2} = \left\{\left(x, y_2, - \frac{x^2}{2 a y_2} \right) : y_2  \leq 0 \right\} ,
	\]
	\SJnew{and a}pplying the secondary blow-up transformation \eqref{secondary_K22} and expressing $\mathcal G_{+,2}$ in chart $K_{22}$ coordinates gives
	\[
	r_2 = s_2^2 = - \frac{x^2}{2 a y_2} = - \frac{s_2^2 x_2^2}{2 a y_2} \qquad \implies \qquad y_2(x_2) = - \frac{x_2^2}{2 a}.
	\]
	A similar argument applies for $\mathcal G_-$, and the result follows.
\end{proof}

\section{Summary and outlook}\seclab{outlook}
In this paper, we have proved existence of limit cycles in two prototypical examples with singular exponential nonlinearities. Our approach was geometric and consisted of the following: Firstly, we applied a proper normalization by dividing the right hand side by a suitable factor, producing a PWS limit as $\epsilon\rightarrow 0$. Secondly, we applied a modification of the blow-up approach following e.g. \cite{2019arXiv190806781U,kristiansen2018a} to deal with degeneracies of this type, recall \secref{hereee}. For the Hester system, this basically led to a system with desirable hyperbolicity properties allowing us to perturb away from a singular cycle. Under the assumptions \eqref{hesterCondition}, the exponential nonlinearities did therefore not provide any obstacles for this result. However, if \eqref{hesterCondition} is violated by $\gamma\ge 1$, then $(0,0)$ is a stable node for \eqref{hesteryNeg} and $\Gamma^1$ is therefore asymptotic to this point. This leads to a more degenerate singular cycle in \figref{hesterBlowup}, with the orbit within $y<0$ connecting to the orange circle at $x=0$. This point is extra singular due to the exponentials and studying limit cycles in this case would require use of the same machinery used to describe the oscillations in the Le Corbeiller problem. 

For the Le Corbeiller problem, we did not recover the essential geometric structure, recall \figref{hesterKappa2}, by a simple scaling. Instead the `slow manifold' was hidden within a separate `exponential' blow-up, see $\mathcal C$ in \figref{corbeillerBlowup2}. This also led to more complicated asymptotics, recall \lemmaref{slowManifold-Corb}, which we were able to capture using the method in \cite{kristiansen2017a}.

\KUK{In conclusion, our geometric approach for studying the Hester and the Le Corbeiller systems is general and we therefore expect that it can be applied to different problems of this kind, including the ones discussed in the introduction. This will be part of future work in the area.} We also expect that it is possible to use these methods to explain the `canard explosion' phenomena that occurs in both systems when small Hopf cycles grow to the `relaxation-type' oscillations described in our main results.

\section*{Acknowledgments}

MW was supported by the Australian Research Council DP180103022 grant. SJ would also like to thank the Technical University of Denmark (DTU) and the second author, for hospitality during their stay, without which this work may not have been possible.

\bibliography{refs}
\bibliographystyle{plain}

 \end{document}